\documentclass[10pt,a4paper]{article}

\usepackage[utf8]{inputenc}
\usepackage[T1]{fontenc}
\usepackage[english]{babel}

\usepackage{mathpazo} 

\usepackage[a4paper,margin=3cm]{geometry}
\usepackage{enumitem}
\usepackage{setspace}
\usepackage{comment}
\usepackage{subcaption}

\usepackage{amsmath,amssymb,amsthm,mathtools,bm,esint}

\numberwithin{equation}{section}

\usepackage{graphicx}
\usepackage{tikz}
\usepackage{pgfplots}
\pgfplotsset{compat=1.18}
\usepgfplotslibrary{fillbetween}
\usepackage{pgfplotstable}
\usetikzlibrary{intersections}
\usepackage{float}

\usepackage{hyperref}
\hypersetup{
	colorlinks=true,
	linkcolor=blue!50!black,
	citecolor=red!60!black,
	urlcolor=blue!70!black,
	pdftitle={Scientific Article},
	pdfauthor={Your Name},
	pdfsubject={Mathematics Paper},
	pdfkeywords={PDE, Eigenvalue problem, Maximum principle}
}

\theoremstyle{plain}
\newtheorem{theorem}{Theorem}[section]
\newtheorem{lemma}[theorem]{Lemma}
\newtheorem{proposition}[theorem]{Proposition}
\newtheorem{corollary}[theorem]{Corollary}

\numberwithin{theorem}{section}

\theoremstyle{definition}

\theoremstyle{remark}
\newtheorem{remark}[theorem]{Remark}

\numberwithin{equation}{section}
\numberwithin{theorem}{section}  

\newcommand{\R}{\mathbb{R}}

\usepackage[
backend=biber,
style=numeric,       
sorting=nty,         
doi=false,
isbn=false,
url=false,
eprint=false,
language=english,
giveninits=true,
maxbibnames=99,
maxcitenames=99    
]{biblatex}

\DeclareFieldFormat[article]{title}{#1}
\DeclareFieldFormat[inbook]{title}{#1}
\DeclareFieldFormat[incollection]{title}{#1}
\DeclareFieldFormat[inproceedings]{title}{#1}
\DeclareFieldFormat[book]{title}{#1}

\usepackage{enumitem} %

\newlist{properties}{enumerate}{1}
\setlist[properties,1]{label=\textnormal{(P\arabic*)},ref=P\arabic*}

\vspace{-0.6cm}
\title{\textbf{Eigenvalue Curves under the $C^{0,\alpha}$ Maximum Principle}}
\author{Your Name\thanks{Department of Mathematics, University Name, Country.}}
\date{\today}

\begin{document}
\vspace{-0.6cm}
	\title{Principal eigenvalues, maximum principles and estimates for Lane-Endem related\\ systems in nondivergence form	
}
\vspace{-0.6cm}
\author{\textbf{Leandro G. Fernandes Jr. \footnote{\textit{E-mail addresses}:
leandro.fernandes@ufrr.br (L.
Fernandes)}}\\ {\small\it Coordenação do Curso de Bacharelado em Matem\'{a}tica,
Universidade Federal de Roraima,}\\ {\small\it Av. Cap. Ene Garcêz 2413, Bairro Aeroporto, Boa Vista, Roraima, 69310-000, Brazil}\\
\textbf{Edir Júnior Ferreira Leite \footnote{\textit{E-mail addresses}:
edirleite@ufscar.br (E.J.F. Leite)}}\\ {\small\it Departamento de Matem\'{a}tica,
Universidade Federal de São Carlos,}\\ {\small\it São Carlos, SP, 13565-905, Brazil.}}

\date{}{

\maketitle

\markboth{abstract}{abstract}
\addcontentsline{toc}{chapter}{abstract}
\vspace{-0.6cm}
\hrule \vspace{0.1cm}

{\bf Abstract.} We study non-variational systems\vspace{-0.3cm}
\[
\hspace{3cm}\begin{cases}
-\,\mathcal{L}_1 u=\lambda\,a(x)\,\vert u\vert^{\alpha_1-1}u\vert v\vert^{\beta_1-1}v & \text{in }\Omega,\\[2pt]
-\,\mathcal{L}_2 v=\mu\,b(x)\,\vert v\vert^{\alpha_2-1}v\vert u\vert^{\beta_2-1}u & \text{in }\Omega,\hspace{4cm}\mbox{{\bf(S)}}\\[2pt] 
u=v=0 & \text{on }\partial\Omega,
\end{cases}
\]
where $\mathcal{L}_1$ and $\mathcal{L}_2$ are uniformly elliptic second-order operators in nondivergence form on a bounded open set $\Omega\subset\mathbb{R}^n$, $n \ge 1$, with $C^{1,1}$ boundary, and the coefficients $a, b$ belong to $L^p(\Omega),$ $p>n$. The exponents satisfy 
$0\le \alpha_1,\alpha_2<1$, and $\beta_1,\beta_2 > 0$ obey  
$\beta_1\beta_2=(1-\alpha_1)(1-\alpha_2)$, including the well known Lane-Emden system ($\alpha_1=0=\alpha_2$). Our aim in this paper is to study principal eigenvalues, maximum principles and related estimates for {\bf(S)}. More specifically, in comparison to the Lane-Emden system, the main findings in this paper for the general system {\bf(S)} are the following:
\begin{itemize}
\item[(i)] The ABP estimate valid for the Lane-Emden system is extended to {\bf(S)}. 
\item[(ii)] The region where the weak maximum principle {\bf (WMP)} holds is located below the principal curve if and only if {\bf(S)} is the Lane-Endem system. (see Figures 1 and 2). However, the region of validity of the weak minimum principle {\bf (WmP)} remains below the principal curve regardless of the values $\alpha_1$ and $\alpha_2$.
\item[(iii)] The strong maximum principle {\bf (SMP)} fails in general for {\bf(S)}. A counterexample is constructed if $\alpha_1,\alpha_2$ are not both zero. However, a local {\bf (SMP)} remains valid for {\bf(S)}. This strongly contrasts the Lane-Enden case where {\bf (SMP)} holds exactly where {\bf (WMP)} holds. 
 \item[(iv)] An antimaximum principle is established for the Lane-Enden system for operators in nondivergence form. Regarding the gerenal case {\bf(S)}, a counterexample for the antimaximum principle is constructed if $\alpha_1,\alpha_2$ are not both zero.
\end{itemize} 

Some additional consequences are entailed, such as a {\bf (WMP)} in small domains $\Omega$ for all $(\lambda, \mu)$ in an entire quadrant in the plane, and the existence of nonnegative solution when the right-hand side of the system {\bf(S)} is perturbed by nonnegative functions in $L^p(\Omega),$ $p>n$.

 \vspace{0.5cm}
\hrule\vspace{0.2cm}

\setcounter{page}{1}

\vspace{0.20cm}
	
	\section{Introduction and main results}
	
The study of principal spectral curves associated with cooperative elliptic systems 
has a long and rich history. Numerous contributions have been devoted to the analysis 
of their existence, qualitative properties, and applications; 
see, for instance, \cite{MR1726799, MR2742481, MR4072483, MR4732371, MR1255895, MR1765542, MR4525730}. 
An important contribution in this direction was provided by Montenegro~\cite{MR1765542}, 
who developed a general spectral theory for nonlinear Lane–Emden type systems involving 
uniformly elliptic second-order operators, namely,
\begin{equation}\label{Montenegro}
	\begin{cases}
		-\,{\cal L}_1 u = \lambda\, a(x)\, \vert v\vert^{\beta_1-1}v & \text{in } \Omega, \\[2pt]
		-\,{\cal L}_2 v = \mu\, b(x)\, \vert u\vert^{\beta_2-1}u & \text{in } \Omega, \\[2pt]
		u = v = 0 & \text{on } \partial\Omega,
	\end{cases}
\end{equation}
where $\Omega\subset\mathbb{R}^n$ is a bounded domain with $C^{1,1}$ boundary, 
$a$ and $b$ are positive weight functions, 
$\beta_1,\beta_2>0$ satisfy the homogeneity relation $\beta_1\beta_2=1$, 
and ${\cal L}_1$ and ${\cal L}_2$ are uniformly elliptic second-order operators. 
Using a constructive approach based on Leray–Schauder degree theory, 
Krasnoselskiĭ’s fixed point method, and the sub–supersolution technique, 
it was shown in~\cite{MR1765542} that the set of principal eigenvalues forms a smooth curve 
in the first quadrant of $\mathbb{R}^2$, 
exhibiting several qualitative properties such as simplicity, local isolation, continuity, 
asymptotic behaviour, and a min–max characterization.

We extend this theory to a broader nonlinear framework,  
where the system~\eqref{Montenegro} involves mixed power-type interactions 
$\vert u\vert^{\alpha_1-1}u\vert v\vert^{\beta_1-1}v$ and 
$\vert v\vert^{\alpha_2-1}v\vert u\vert^{\beta_2-1}u$, 
under the condition $\beta_1\beta_2 = (1-\alpha_1)(1-\alpha_2)$. 
This extension includes the Lane–Emden case as a special instance, 
but introduces an additional layer of asymmetry and nonlinearity 
that requires a refined application of the Krein–Rutman theorem.

We study several questions related to the nonlinear system
\begin{equation}\label{1.3}
	\begin{cases}
		-\,\mathcal{L}_1 u=\lambda\,a(x)\,\vert u\vert^{\alpha_1-1}u\vert v\vert^{\beta_1-1}v & \text{in }\Omega,\\[2pt]
		-\,\mathcal{L}_2 v=\mu\,b(x)\,\vert v\vert^{\alpha_2-1}v\vert u\vert^{\beta_2-1}u & \text{in }\Omega,\\[2pt]
		u=v=0 & \text{on }\partial\Omega,
	\end{cases}
\end{equation}
where $\Omega\subset\mathbb{R}^n$ is a bounded open set with $C^{1,1}$ boundary, $n \ge 1$, 
$0\le \alpha_1, \alpha_2 < 1$, and $\beta_1,\beta_2>0$ satisfy $\beta_1\beta_2=(1-\alpha_1)(1-\alpha_2)$, 
and $(\lambda, \mu)\in\mathbb{R}^2$.

Let $a, b \in L^p(\Omega)$ with $p>n$, satisfying
\[
\operatorname*{ess\,inf}_{x\in\Omega} a(x)>0
\quad\text{and}\quad
\operatorname*{ess\,inf}_{x\in\Omega} b(x)>0.
\]
Assume that $\mathcal{L}_1$ and $\mathcal{L}_2$ are uniformly elliptic second-order operators of nondivergence type, 
each of the form
\[
\mathcal{L}_l = a^l_{ij}(x) \partial_{ij} + b^l_i(x) \partial_i + c_l(x), \qquad l=1,2,
\]
and satisfy the structural conditions: there exist constants $c_0, C_0>0$ such that
\begin{equation}\label{H.1}
	c_0|\xi|^2 \le a^l_{ij}(x)\xi_i\xi_j \le C_0|\xi|^2,
	\quad\forall\,\xi\in\mathbb{R}^n,\ x\in\Omega,
\end{equation}
with $a^l_{ij}\in C(\overline{\Omega})$, $b^l_i, c_l\in L^\infty(\Omega)$, and a constant $T>0$ such that
\begin{equation}\label{H.2}
	|b^l_i(x)|,\,|c_l(x)|\le T,\quad\forall x\in\Omega,\ i=1,\dots,n.
\end{equation}


Let $\lambda_1(-{\cal L}_l-\lambda a(x))$ be the principal eigenvalue of the linear scalar problem
\[
\begin{cases}
-\,\mathcal{L}_l u=\lambda\,a(x) u & \text{in }\Omega,\\[2pt]
u=0 & \text{on }\partial\Omega.
\end{cases}
\]
We says that $\lambda_1(-{\cal L}_l-\lambda) > 0$ if and only if the operator ${\cal L}_l$ satisfies Scalar Weak Maximum Principle in $\Omega$, that is, for any function $u \in W^{2,p}(\Omega)$ such that ${\cal L}_l u \leq 0$ in $\Omega$ a.e. and $u \geq 0$ on $\partial \Omega$, one has $u \geq 0$ in $\Omega$. The positivity of $\lambda_1(-{\cal L}_l-\lambda)$ is also equivalent to the Scalar Strong Maximum Principle for ${\cal L}_l$ in $\Omega$, that is, for any function $u \in W^{2,p}(\Omega)$ such that ${\cal L}_l u \leq 0$ in $\Omega$ a.e. and $u \geq 0$ on $\partial \Omega$, one has either $u \equiv 0$ in $\Omega$ or $u > 0$ in $\Omega$ and $\partial_\nu u(x_0)<0$ for every $x_0\in\partial\Omega$ with $u(x_0)=0$, 
where \(\nu\equiv\nu(x_0)\) denotes the outward unit normal to $\partial\Omega$ at $x_0$. Alternatively, the operator ${\cal L}_l + \lambda a(x)$ satisfies Scalar Weak Maximum Principle (or Scalar Strong Maximum Principle) in $\Omega$ if and only if $\lambda < \lambda_1(-{\cal L}_l-\lambda a(x))$. These equivalences were proved in
\cite[Theorem~1.1]{MR1258192} and \cite[Theorem~2.1]{MR1255895}. We also refer to the monographs \cite{MR762825} and \cite{MR2597943},
where these results were previously established for self-adjoint operators. 

Unless otherwise stated, we assume that ${\cal L}_1$ and ${\cal L}_2$ satisfy the Scalar Weak Maximum Principle and Scalar Strong Maximum Principle in $\Omega$.

We denote \(V:=C^{1,\eta}_0(\overline{\Omega})\),
\(V_+:=\{u\in V:\,u(x)\ge0\text{ for all }x\in\Omega\}\),
and let $\text{int}_V(V_+)$ be the topological interior of \(V_+\) in \(V\).
It is well known (see Lemma~\ref{nonempty}) that \(\text{int}_V(V_+)\neq\emptyset\) and can be characterized by
\[
u\in \text{int}_V(V_+)
\iff
u\in V_+^\circ := \Big\{ u \in V :\ u > 0 \ \text{in } \Omega \ \text{and} \ \inf_{\partial\Omega}(-\partial_\nu u) > 0 \Big\}.
\]

We say that $(\lambda,\mu)\in\mathbb{R}^2$ is an \emph{eigenvalue} of system~\eqref{1.3} 
if the system admits a nontrivial strong solution $(\varphi,\psi)\in (W^{2,p}(\Omega))^2$, 
called an \emph{eigenfunction} corresponding to $(\lambda,\mu)$. 
We call $(\lambda,\mu)$ a \emph{principal eigenvalue} if it possesses a positive eigenfunction $(\varphi,\psi)$, 
that is, $\varphi,\psi>0$ in $\Omega$. 
Moreover, $(\lambda,\mu)$ is said to be \emph{simple in} \(V_+^\circ\times V_+^\circ\) if, for any eigenfunctions 
$(\varphi,\psi),(\tilde{\varphi},\tilde{\psi})\in V_+^\circ\times V_+^\circ$, there exists $\rho>0$ such that 
\[
\tilde{\varphi}=\rho\,\varphi
\quad\text{and}\quad
\tilde{\psi}=\rho^{\frac{1-\alpha_1}{\beta_1}}\psi
\quad\text{in }\Omega.
\]

As already mentioned, in this paper we extend the results of Montenegro~\cite{MR1765542} 
to the general case $0 \le \alpha_1, \alpha_2 < 1$ and $\beta_1, \beta_2 > 0$ satisfying 
$\beta_1 \beta_2 = (1 - \alpha_1)(1 - \alpha_2)$. 
More precisely, we study the nonlinear coupled system~\eqref{1.3}, 
which serves as the general framework for our analysis, 
as well as its nonhomogeneous counterpart where external source terms appear 
on the right-hand sides. 
We establish the existence of principal eigenvalues of system~\eqref{1.3} 
and describe several of their qualitative properties. 
In particular, we show that the set of all such principal eigenvalues 
is represented by the smooth curve
\[
\mathcal{C}_1
:= \Big\{ (\lambda, \mu) \in (\mathbb{R}_+^*)^2 : 
\lambda^{\frac{1}{\sqrt{\beta_1(1-\alpha_1)}}} \,
\mu^{\frac{1}{\sqrt{\beta_2(1-\alpha_2)}}}
= \Lambda_0
\Big\},
\]
for some constant $\Lambda_0 > 0$; see Section~\ref{sec:proof-teor2}. 
This curve will be refered as the principal curve associated to the system \eqref{1.3}.

Let ${\cal R}_1$ be the open region in the first quadrant below ${\cal C}_1$, $${\cal C}_2:=\Big\{ (\lambda, \mu)\in \R^2 : (\lambda,-\mu)\in {\cal C}_1 \Big\}$$ and ${\cal R}_2$ be the open region in the fourth quadrant above ${\cal C}_2$. We will study the following properties:

\begin{properties}
	\item \label{prop:uniqueness} (Uniqueness)
	$(\lambda, \mu) \in \mathbb{R}^2$ is a principal eigenvalue of the system~\eqref{1.3} 
	if and only if $(\lambda,\mu)\in \mathcal{C}_1$.
	
	\item \label{prop:simplicity} (Simplicity in $V_+^\circ\times V_+^\circ$)
	The curve $\mathcal{C}_1$ is simple in $V_+^\circ\times V_+^\circ$, that is, $(\lambda,\mu)$ is simple in $V_+^\circ\times V_+^\circ$ 
	for all $(\lambda,\mu)\in \mathcal{C}_1$.
	
	\item \label{prop:simplicity2} (Simplicity in $V\times V$)
Let $(\varphi,\psi)\in V\times V$ be an eigenfunction associated to $(\lambda,\mu)\in \mathcal{C}_1$. If $\alpha_1=0=\alpha_2$, then either
 $(\varphi,\psi)\in V_+^\circ\times V_+^\circ$ or  $-(\varphi,\psi)\in V_+^\circ\times V_+^\circ$. Now, if $0 \le \alpha_1, \alpha_2 < 1$, not both zero, then  $(\varphi,\psi)\in V_+^\circ\times V_+^\circ$.
	
	\item \label{prop:isolation} (Upper isolated) For each $(\lambda_1,\mu_1)\in\mathcal{C}_1$, there exists $\varepsilon>0$ 
	such that $B_{\varepsilon}(\lambda_1,\mu_1)\cap \overline{\mathcal{R}_1}^c$ does not contain any eigenvalue of~\eqref{1.3}.

\item \label{prop:monotonicity} (Monotonicity with respect to the weights)
If $a \leq \tilde{a}$ and $b \leq \tilde{b}$ in $\Omega$, then
\[
\Lambda_0(a, b,\Omega) \geq \Lambda_0(\tilde{a}, \tilde{b},\Omega),
\]
and if $(a,b)\not\equiv(\tilde{a},\tilde{b})$ the inequality is strict.

\item \label{LI} (Lower isolated) If $0 \le \alpha_1, \alpha_2 < 1$ are not both zero, then there are no eigenvalues of~\eqref{1.3} in $\overline{{\cal R}_1} \setminus {\cal C}_1\cup-(\overline{{\cal R}_1} \setminus {\cal C}_1)\cup \overline{{\cal R}_2} \setminus {\cal C}_2\cup-(\overline{{\cal R}_2} \setminus {\cal C}_2)$.
\end{properties}

The spectral curve associated with negative eigenfunctions of~\eqref{1.3} is given by

\begin{itemize}
    \item[{\rm (i)}]  $\mathcal{C}_1$ when $\alpha_1=0=\alpha_2$ (Lane-Emden System), 
	\item[{\rm (ii)}]  $-\mathcal{C}_1$ when $0<\alpha_1,\alpha_2<1$, 
	\item[{\rm (iii)}] $\mathcal{C}_2$ when $\alpha_1=0$ and $0<\alpha_2<1$, 
	\item[{\rm (iv)}] $-\mathcal{C}_2$ when $\alpha_2=0$ and $0<\alpha_1<1$.
\end{itemize}

In the same spirit as the results in~\cite{MR2742481}, we now state our first theorem:

\begin{theorem}[Principal curve]\label{teor2}
	Assume that
	\begin{equation}\label{cond1}
		\beta_1 \beta_2 = (1 - \alpha_1)(1 - \alpha_2).
	\end{equation}
	Then there exist a constant $\Lambda_0 > 0$ such that 
	system~\eqref{1.3} admits a positive solution $(u,v) \in V_+\times V_+$ corresponding to some 
	$(\lambda, \mu) \in (\mathbb{R}_+^*)^2$ if and only if
	\begin{equation}\label{cond2}
		\lambda^{\frac{1}{\sqrt{\beta_1(1-\alpha_1)}}}\,
		\mu^{\frac{1}{\sqrt{\beta_2(1-\alpha_2)}}}
		= \Lambda_0.
	\end{equation}
	Moreover, the curve determined by~\eqref{cond2} possesses 
	the qualitative properties stated in~\eqref{prop:uniqueness}--\eqref{prop:monotonicity}.
\end{theorem}

	By {\bf (WMP)} we mean that for any supersolution of problem~\eqref{1.3}, that is, for any pair of functions $(u,v)$ in $(W^{2,p}(\Omega))^2$ satisfying

\[
\left\{
\begin{array}{llll}
-{\cal L}_1u  \geq \lambda a(x)\vert u\vert^{\alpha_1-1}u\vert v\vert^{\beta_1-1}v & \text{in } \Omega, \\[2pt]
-{\cal L}_2v  \geq \mu b(x)\vert v\vert^{\alpha_2-1}v\vert u\vert^{\beta_2-1}u & \text{in } \Omega, \\[2pt]
\end{array}
\right. 
\]
and $u, v \geq 0$ on $ \partial\Omega$, one has $u, v \geq 0$ in $\Omega$. If, in addition, the set $\{x\in\Omega:u(x),v(x)>0\}$ is nonempty when $u$ and $v$ are nontrivial in $\Omega$, we say that the local {\bf (SMP)} associated to \eqref{1.3} holds in $\Omega$. Furthermore, in the case $\lambda,\mu\neq0$, we say that the {\bf (SMP)} associated to~\eqref{1.3} holds in $\Omega$ if either $u\equiv v\equiv0$ in $\Omega$ or $u,v>0$ in $\Omega$. When at least one of the parameters $\lambda$ and $\mu$ is zero, we say that the {\bf (SMP)} associated to~\eqref{1.3} holds in $\Omega$ if either $u\equiv v\equiv0$ in $\Omega$ or at least one of the functions $u$ and $v$ is positive in $\Omega$.

In Theorem~1.1 of~\cite{MR4072483}, the authors completely characterize the set of 
$(\lambda, \mu) \in \mathbb{R}^2$ for which {\bf (WMP)} and {\bf (SMP)} associated with~\eqref{1.3}, 
in the Lane--Emden case $(\alpha_1 = \alpha_2 = 0)$, hold in~$\Omega$. 
Precisely:
\[
(\lambda,\mu) \in \overline{{\cal R}_1} \setminus {\cal C}_1 \Leftrightarrow\ {\bf (WMP)}\text{ holds in}\ \Omega \Leftrightarrow\ {\bf (SMP)}\text{ associated to~\eqref{1.3} holds in}\ \Omega.
\]

Our main theorem fully characterizes the set of pairs $(\lambda, \mu) \in \R^2$ for which {\bf (WMP)} holds for system~\eqref{1.3} when $\alpha_1$ and $\alpha_2$ are not both zero.

Precisely,

\begin{theorem}[WMP] \label{MP}
 Suppose that \eqref{cond1} holds. Then:

\begin{itemize}
\item[{\rm (i)}] If $0<\alpha_1,\alpha_2<1$, then $(\lambda, \mu) \in -(\overline{{\cal R}_1} \setminus {\cal C}_1)$ if and only if {\bf (WMP)} associated to \eqref{1.3} holds in $\Omega$;

\item[{\rm (ii)}] If $\alpha_1=0$ and $0<\alpha_2<1$, then $(\lambda, \mu) \in \overline{{\cal R}_2} \setminus {\cal C}_2$ if and only if {\bf (WMP)} associated to \eqref{1.3} holds in $\Omega$;

\item[{\rm (iii)}] If $\alpha_2=0$ and $0<\alpha_1<1$, then $(\lambda, \mu) \in -(\overline{{\cal R}_2} \setminus {\cal C}_2)$ if and only if {\bf (WMP)} associated to \eqref{1.3} holds in $\Omega$.
\end{itemize}
\end{theorem}


\begin{figure}[H]
	\centering
	
	\begin{subfigure}{0.48\textwidth}
		\centering
		\begin{tikzpicture}
			\begin{axis}[
				axis lines=middle,
				xmin=-4, xmax=4,
				ymin=-4, ymax=4,
				xtick=\empty, ytick=\empty,
				xlabel={$\lambda$},
				ylabel={$\mu$},
				clip=false
				]
				
				\addplot[name path=curve,domain=0.5:4,samples=200,draw=none] {2/x};
				\path[name path=bottom] (0,0) -- (4,0);
				\path[name path=vert] (0,0) -- (0,4);
				
				\addplot[gray!35] fill between[of=curve and bottom,soft clip={domain=0.5:4}];
				\addplot[draw=none, fill=gray!35]
				coordinates {(0,0) (0.5,0) (0.5,4) (0,4)} -- cycle;
				
				\addplot[domain=0.5:4,samples=200,very thick,black,dashed] {2/x};
				
				\addplot[gray!80, very thick] coordinates {(0,0) (4,0)};
				\addplot[gray!80, very thick] coordinates {(0,0) (0,4)};
				
				\node[black!60!black] at (1.6,0.6) {$\overline{\mathcal R}_1 \setminus \mathcal C_1$};
				\node[black] at (1.6,2.0) {$\mathcal C_1$};
				
			\end{axis}
		\end{tikzpicture}
		\caption{$\overline{\mathcal R}_1 \setminus \mathcal C_1$, $\ \alpha_1=\alpha_2=0$ (Lane-Emden System)}
	\end{subfigure}
	\hfill
	\begin{subfigure}{0.48\textwidth}
		\centering
		\begin{tikzpicture}
			\begin{axis}[
				axis lines=middle,
				xmin=-4, xmax=4,
				ymin=-4, ymax=4,
				xtick=\empty, ytick=\empty,
				xlabel={$\lambda$},
				ylabel={$\mu$},
				clip=false
				]
				
				\addplot[name path=curve,domain=-4:-0.5,samples=250,draw=none] {2/x};
				\path[name path=top] (-4,0) -- (0,0);
				\path[name path=vert] (0,0) -- (0,-4);
				
				\addplot[gray!35] fill between[of=curve and top,soft clip={domain=-4:-0.5}];
				\addplot[gray!35] fill between[of=curve and vert,soft clip={domain=-0.5:0}];
				\addplot[draw=none, fill=gray!35]
				coordinates {(0,0) (-0.5,0) (-0.5,-4) (0,-4)} -- cycle;
				
				\addplot[domain=-4:-0.5,samples=250,very thick,black,dashed] {2/x};
				
				\addplot[gray!80, very thick] coordinates {(-4,0) (0,0)};
				\addplot[gray!80, very thick] coordinates {(0,-4) (0,0)};
				
				\node[black!60!black] at (-1.6,-0.6) {$-(\overline{\mathcal R}_1 \setminus \mathcal C_1)$};
				\node[black] at (-1.6,-2.0) {$-\mathcal C_1$};
				
			\end{axis}
		\end{tikzpicture}
		\caption{$-(\overline{\mathcal R}_1 \setminus \mathcal C_1)$, $\ 0<\alpha_1<1,\ 0<\alpha_2<1$}
	\end{subfigure}
	
	\caption{Admissible regions determined by $\mathcal C_1$ in the first and third quadrants.}
	\label{fig:C1pair}
\end{figure}
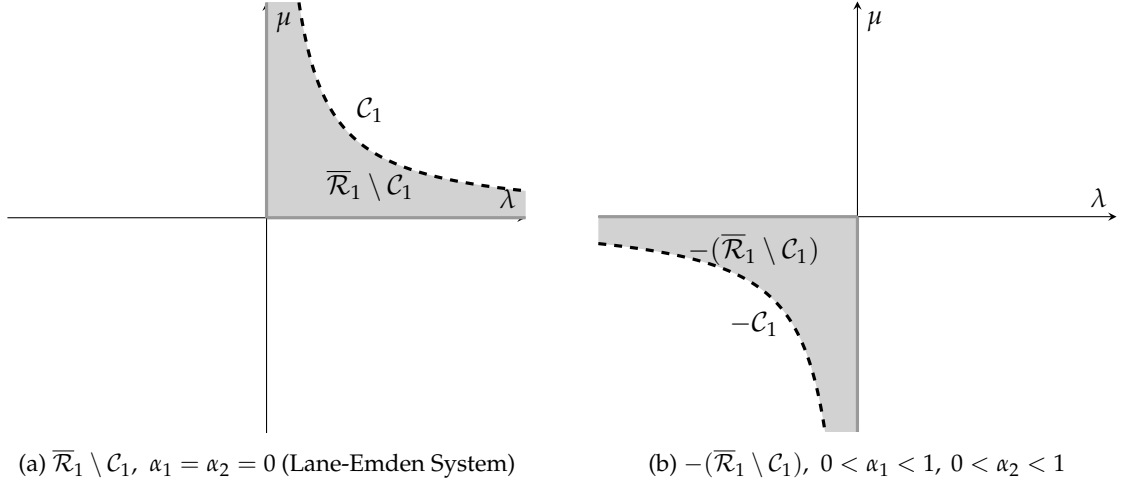


\begin{figure}[H]
	\centering
	
	\begin{subfigure}{0.48\textwidth}
		\centering
		\begin{tikzpicture}
			\begin{axis}[
				axis lines=middle,
				xmin=-4, xmax=4,
				ymin=-4, ymax=4,
				xtick=\empty, ytick=\empty,
				xlabel={$\lambda$},
				ylabel={$\mu$},
				clip=false
				]
				\addplot[name path=curve,domain=0.5:4,samples=200,draw=none] {-2/x};
				\path[name path=top] (0,0) -- (4,0);
				\path[name path=vert] (0,0) -- (0,-4);
				
				\addplot[gray!35] fill between[of=curve and top,soft clip={domain=0.5:4}];
				\addplot[gray!35] fill between[of=curve and vert,soft clip={domain=0:0.5}];
				\addplot[draw=none, fill=gray!35]
				coordinates {(0,0) (0.5,0) (0.5,-4) (0,-4)} -- cycle;
				
				\addplot[domain=0.5:4,samples=200,very thick,black,dashed] {-2/x};
				
				\addplot[gray!80, very thick] coordinates {(0,0) (4,0)};
				\addplot[gray!80, very thick] coordinates {(0,-4) (0,0)};
				
				\node[black!60!black] at (1.6,-0.6) {$\overline{\mathcal R}_2 \setminus \mathcal C_2$};
				\node[black] at (1.6,-2.0) {$\mathcal C_2$};
				
			\end{axis}
		\end{tikzpicture}
		\caption{$\overline{\mathcal R}_2 \setminus \mathcal C_2$, $\ \alpha_1=0,\ 0<\alpha_2<1$}
	\end{subfigure}
	\hfill
	\begin{subfigure}{0.48\textwidth}
		\centering
		\begin{tikzpicture}
			\begin{axis}[
				axis lines=middle,
				xmin=-4, xmax=4,
				ymin=-4, ymax=4,
				xtick=\empty, ytick=\empty,
				xlabel={$\lambda$},
				ylabel={$\mu$},
				clip=false
				]
				\def\k{2}
				
				\addplot[name path=curve,domain=-4:-0.5,samples=250,draw=none] {-\k/x};
				\path[name path=top] (-4,0) -- (0,0);
				\path[name path=vert] (0,0) -- (0,4);
				
				\addplot[gray!35] fill between[of=curve and top,soft clip={domain=-4:-0.5}];
				\addplot[draw=none, fill=gray!35]
				coordinates {(0,0) (-0.5,0) (-0.5,4) (0,4)} -- cycle;
				
				\addplot[domain=-4:-0.5,samples=250,very thick,black,dashed] {-\k/x};
				
				\addplot[gray!80, very thick] coordinates {(-4,0) (0,0)};
				\addplot[gray!80, very thick] coordinates {(0,0) (0,4)};
				
				\node[black!60!black] at (-1.6,0.6) {$-(\overline{\mathcal R}_2 \setminus \mathcal C_2)$};
				\node[black] at (-1.6,2.0) {$-\mathcal C_2$};
				
			\end{axis}
		\end{tikzpicture}
		\caption{$-(\overline{\mathcal R}_2 \setminus \mathcal C_2)$, $\ \alpha_2=0,\ 0<\alpha_1<1$}
	\end{subfigure}
	
	\caption{Admissible regions in the second and fourth quadrants.}

	\label{fig:twoRegions}
\end{figure}
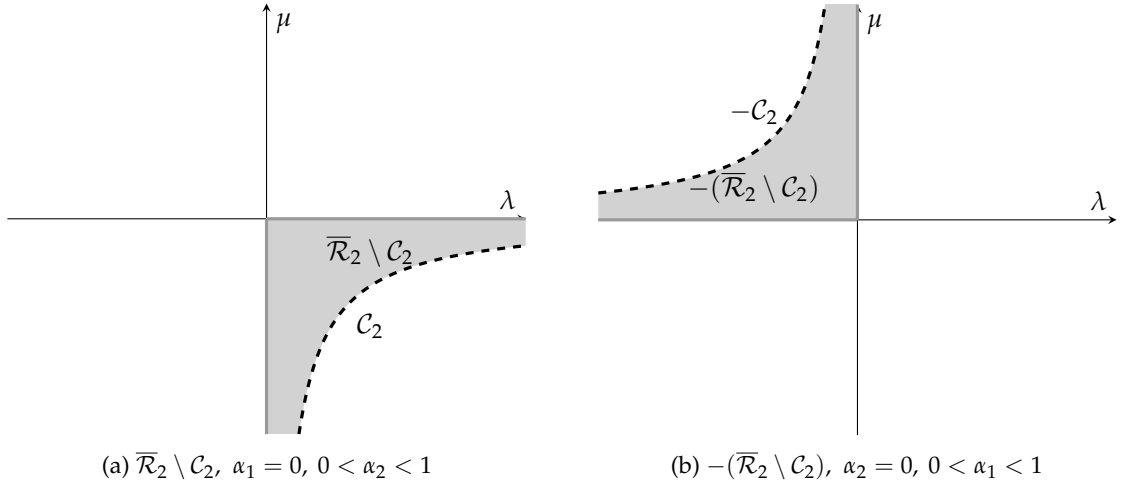

By the Weak Minimum Principle {\bf (WmP)} we mean that for any subsolution of problem \eqref{1.3}, that is, for any pair of functions $(u,v)$ in $(W^{2,p}(\Omega))^2$ satisfying

\[
\left\{
\begin{array}{llll}
-{\cal L}_1u  \leq \lambda a(x)\vert u\vert^{\alpha_1-1}u\vert v\vert^{\beta_1-1}v & \text{in } \Omega, \\[2pt]
-{\cal L}_2v  \leq \mu b(x)\vert v\vert^{\alpha_2-1}v\vert u\vert^{\beta_2-1}u & \text{in } \Omega, \\[2pt]
\end{array}
\right. 
\]
and $u, v \leq 0$ on $ \partial\Omega$, one has $u, v \leq 0$ in $\Omega$. 

As a consequence of Theorem~\ref{MP}, we obtain

\begin{corollary}[WmP] \label{mP}
 Suppose that $0\le \alpha_1,\alpha_2<1$, and $\beta_1,\beta_2 > 0$ obey \eqref{cond1}. Then $(\lambda, \mu) \in \overline{{\cal R}_1} \setminus {\cal C}_1$ if and only if {\bf (WmP)} associated to \eqref{1.3} holds in $\Omega$.
\end{corollary}

Note that {\bf (WMP)} is related to the spectral curve associated with negative eigenfunctions, while {\bf (WmP)} is associated with the curve ${\cal C}_1$, that is, the principal eigenvalues. 

Thus, in comparison to the Lane-Emden system, we discovered that the region where the {\bf (WMP)} holds is located below the principal curve if and only if (\ref{1.3}) is the Lane-Endem system, that is, $\alpha_1=0=\alpha_2$. However, the region of validity of the {\bf (WmP)} remains below the principal curve regardless of the values $\alpha_1$ and $\alpha_2$.

\begin{corollary}
The curve $\mathcal{C}_1$ satisfies property (\ref{LI}).
\end{corollary}

The claim follows from Theorem~\ref{MP}, the structure of problem~\eqref{1.3}, and the Scalar Weak Maximum Principle.

\begin{remark}\label{example}
We show below that the {\bf (SMP)} associated with~\eqref{1.3}, when $\alpha_1$ and $\alpha_2$ are not both zero, fails in $\Omega$ in the region where {\bf (WMP)} is valid.
 
 In the case $0<\alpha_1,\alpha_2<1$, by Theorem 7.5 of \cite{MR2025185}, the problem $-\Delta u=-u^{\alpha_1}$ admits a $C^2$ (nontrivial) nonnegative compactly supported solution in 
$\{x \in \R^n : \vert x\vert > 1\}$.

Thus, for any $\lambda<0$ and $\mu\leq 0$, there exists $R>1$ such that $u$ is nontrivial in $B_R(0)$ and
\[
	\begin{cases}
		-\,\Delta \gamma u=\lambda\,(\gamma u)^{\alpha_1}1^{\beta_1}& \text{in }B_R(0),\\[2pt]
		-\,\Delta 1=0\geq\mu\,b(x)\,1^{\alpha_2} (\gamma u)^{\beta_2} & \text{in }B_R(0),\\[2pt]
		\gamma u\geq 0 & \text{on }\partial B_R(0),
	\end{cases}
	\]
where $\gamma=(-\lambda)^{\frac{1}{1-\alpha_1}}>0$ and $B_R(0)=\{x \in \R^n : \vert x\vert < R\}$. Therefore, the {\bf (SMP)} fails in $B_R(0)$. Note that $(\lambda,\mu)\in -(\overline{{\cal R}_1} \setminus {\cal C}_1)$ when $(-\lambda)^{\frac{1}{\sqrt{\beta_1(1-\alpha_1)}}}(-\mu)^{\frac{1}{\sqrt{\beta_2(1-\alpha_2)}}}<\Lambda_0$.

In the case $\alpha_2=0<\alpha_1<1$, we take $\lambda<0$ and $\mu\geq 0$ and \(v(x) = E(2R - |x|^2)\), where $E=\max\{1,\sup\{\mu (\gamma u(x))^{\beta_2}:x\in B_R(0)\}\}$.
Thus,
\[
	\begin{cases}
		-\,\Delta \gamma u=\lambda\,(\gamma u)^{\alpha_1}\geq \lambda\,(\gamma u)^{\alpha_1}v^{\beta_1}& \text{in }B_R(0),\\[2pt]
		-\,\Delta v=2En\geq\mu\, (\gamma u)^{\beta_2} & \text{in }B_R(0),\\[2pt]
		\gamma u, v\geq 0 & \text{on }\partial B_R(0).
	\end{cases}
	\]
	Again, the {\bf (SMP)} fails in $B_R(0)$.
\end{remark}

However, by the Scalar Strong Maximum Principle and Theorem~\ref{MP}, we have the following result:

\begin{corollary} (Local {\bf (SMP)})
Let $(u,v)$ in $(W^{2,p}(\Omega))^2$ be a supersolution of problem \eqref{1.3}.
\begin{itemize}
\item[{\rm (i)}] If $0<\alpha_1,\alpha_2<1$ and $(\lambda, \mu) \in -(\overline{{\cal R}_1} \setminus {\cal C}_1)$, then $u,v\geq 0$ in $\Omega$ and the local {\bf (SMP)} associated to \eqref{1.3} holds in $\Omega$.
\item[{\rm (ii)}] If $\alpha_1=0<\alpha_2<1$ and $(\lambda, \mu) \in \overline{{\cal R}_2} \setminus {\cal C}_2$, then $u,v\geq 0$ in $\Omega$ and the local {\bf (SMP)} associated to \eqref{1.3} holds in $\Omega$.
\item[{\rm (iii)}] If $\alpha_2=0<\alpha_1<1$ and $(\lambda, \mu) \in -(\overline{{\cal R}_2} \setminus {\cal C}_2)$, then $u,v\geq 0$ in $\Omega$ and the local {\bf (SMP)} associated to \eqref{1.3} holds in $\Omega$.

\end{itemize}
\end{corollary}

Another topic of great interest is whether or not maximum principles always hold depending on the domain, more specifically, whether there exist counterparts to Theorem 2.6 of~\cite{MR1258192} and to Theorem 1.3 of~\cite{MR4072483} for the system~\eqref{1.3}. Indeed, assuming that ${\cal L}_1$ and ${\cal L}_2$ satisfy \eqref{H.1} and \eqref{H.2} and $(\lambda,\mu)$ is an arbitrary pair in $\R^2$, we characterize when the {\bf (WMP)} holds for domains $\Omega$ of small measure.

Precisely, we have:

\begin{theorem} \label{sm}
({\bf (WMP)} in small domains)
	Suppose that \eqref{cond1} holds. Assume that $n \geq 2$ and $\operatorname{diam}(\Omega) \leq d$ for a fixed number $d > 0$. Then:

\begin{itemize}
\item[{\rm (i)}] If $0<\alpha_1,\alpha_2<1$, we have $\lambda \leq 0$ and $\mu \leq 0$ if and only if there exist constants $\sigma, \kappa_1,\kappa_2 > 0$ depending only on $n$, $c_0$, $C_0$, $T$, $d$, $\lambda$ and $\mu$ such that {\bf (WMP)} associated to \eqref{1.3} holds in $\Omega$ provided that $|\Omega| < \sigma$, $\|a\|_{L^n(\Omega)}<\kappa_1$ and $\|b\|_{L^n(\Omega)} < \kappa_2$;
\item[{\rm (ii)}] If $\alpha_1=0<\alpha_2<1$, we have $\lambda \geq 0$ and $\mu \leq 0$ if and only if there exist constants $\sigma, \kappa_1,\kappa_2 > 0$ depending only on $n$, $c_0$, $C_0$, $T$, $d$, $\lambda$ and $\mu$ such that {\bf (WMP)} associated to \eqref{1.3} holds in $\Omega$ provided that $|\Omega| < \sigma$, $\|a\|_{L^n(\Omega)}<\kappa_1$ and $\|b\|_{L^n(\Omega)} < \kappa_2$;
\item[{\rm (iii)}] If $\alpha_2=0<\alpha_1<1$, we have $\lambda \leq 0$ and $\mu \geq 0$ if and only if there exist constants $\sigma, \kappa_1,\kappa_2 > 0$ depending only on $n$, $c_0$, $C_0$, $T$, $d$, $\lambda$ and $\mu$ such that {\bf (WMP)} associated to \eqref{1.3} holds in $\Omega$ provided that $|\Omega| < \sigma$, $\|a\|_{L^n(\Omega)}<\kappa_1$ and $\|b\|_{L^n(\Omega)} < \kappa_2$.
\end{itemize}
\end{theorem}


\begin{theorem}[A priori estimate for the system]\label{thm:apriori-general}
	Let $n\ge2$. 	Suppose that $0\leq\alpha_1,\alpha_2 < 1$ are not both zero. 	
	Given data
	\[
	f,g\in L^{p}(\Omega),\qquad \varphi,\psi\in W^{2,p}(\Omega),
	\]
	consider the nonhomogeneous system
	\begin{equation}\label{eq:system-general-en}
		\begin{cases}
			-\mathcal L_1 u=\lambda\,a(x)\,|u|^{\alpha_1-1}u\,|v|^{\beta_1-1}v+f(x), & \text{in }\Omega,\\[2pt]
			-\mathcal L_2 v=\mu\,b(x)\,|v|^{\alpha_2-1}v\,|u|^{\beta_2-1}u+g(x), & \text{in }\Omega,\\[2pt]
			u=\varphi,\quad v=\psi & \text{on }\partial\Omega.
		\end{cases}
	\end{equation}
	Under condition~\eqref{cond1}, set
	\begin{equation}\label{eq:def-r-en}
		s:=\frac{\beta_1}{1-\alpha_1}=\frac{1-\alpha_2}{\beta_2}.
	\end{equation}
	
Assume that $(\lambda,\mu)\in \left(\overline{{\cal R}_1} \setminus {\cal C}_1\right)\cup\left(-\left(\overline{{\cal R}_1} \setminus {\cal C}_1\right)\right)\cup\left(\overline{{\cal R}_2} \setminus {\cal C}_2\right)\cup\left(-\left(\overline{{\cal R}_2} \setminus {\cal C}_2\right)\right)$.	Then there exists a constant $A>0$, depending only on $\Omega$, the structural constants of 
$\mathcal L_1,\mathcal L_2$, the parameters $(\lambda,\mu)$, 
$\|a\|_{L^p(\Omega)}$, and $\|b\|_{L^p(\Omega)}$,
	such that any strong solution $(u,v)\in (W^{2,p}(\Omega))^2$ of \eqref{eq:system-general-en} satisfies
	\begin{equation}\label{eq:apriori-final-en}
		\|u\|_{L^\infty(\Omega)}+\|v\|_{L^\infty(\Omega)}^{\,s}
		\;\le\;
		A\Big(
		\|\varphi\|_{W^{2,p}(\Omega)}
		+\|\psi\|_{W^{2,p}(\Omega)}^{\,s}
		+\|f\|_{L^{p}(\Omega)}
		+\|g\|_{L^{p}(\Omega)}^{\,s}
		\Big).
	\end{equation}
\end{theorem}

\begin{theorem}[Existence for the nonhomogeneous problem outside the eigencurve]\label{thm:existence-outside-curve} 

	Under the hypotheses of Theorem \ref{thm:apriori-general}, the problem \eqref{eq:system-general-en} admits at least one strong solution \((u,v)\in (W^{2,p}(\Omega))^2\) for any pairs $(f,g)\in (L^{p}(\Omega))^2$ and $(\varphi,\psi)\in (W^{2,p}(\Omega))^2$ of nonnegative functions.

\begin{itemize}
\item[{\rm (i)}] If $0<\alpha_1,\alpha_2<1$ and $(\lambda, \mu) \in -(\overline{{\cal R}_1} \setminus {\cal C}_1)$, then $u,v\geq 0$ in $\Omega$. 
\item[{\rm (ii)}] If $\alpha_1=0<\alpha_2<1$ and $(\lambda, \mu) \in \overline{{\cal R}_2} \setminus {\cal C}_2$, then $u,v\geq 0$ in $\Omega$. 
\item[{\rm (iii)}] If $\alpha_2=0<\alpha_1<1$ and $(\lambda, \mu) \in -(\overline{{\cal R}_2} \setminus {\cal C}_2)$, then $u,v\geq 0$ in $\Omega$.

Moreover, if $f,g>0$ in $\Omega$, then $u,v$ are positive in $\Omega$.
\end{itemize}	
\end{theorem}

\begin{corollary}
Let $n\ge2$. 	Suppose that $0\leq\alpha_1,\alpha_2 < 1$ are not both zero and that \eqref{cond1} holds. Assume that $(\lambda, \mu) \in (\overline{{\cal R}_1} \setminus {\cal C}_1)$. Then the problem \eqref{eq:system-general-en} admits at least one nonpositive solution \((u,v)\in (W^{2,p}(\Omega))^2\) for any pair $(f,g)\in(L^p(\Omega))^2$ and any pair $(\varphi,\psi)\in(W^{2,p}(\Omega))^2$ of nonpositive functions. Moreover, if $f,g<0$ in $\Omega$, then $u,v$ are negative in $\Omega$.
\end{corollary}

Note that Theorem \ref{thm:existence-outside-curve} guarantees the existence of a nonnegative solution in the region where {\bf (WMP)} holds. However, such regions are outside the first quadrant when $\alpha_1\neq 0$ or $\alpha_2\neq 0$ (see Figures 1 and 2). In Section~\ref{s8} we will study the existence of nonnegative or positive solutions if $(\lambda,\mu)$ belongs to the first quadrant, that is, we will obtain the following result.

\begin{theorem}[Existence for the nonhomogeneous problem in $\mathcal{R}_1$] \label{E1} For each pair $ (\lambda, \mu) \in \mathcal{R}_1$ and any nonnegative pair $(f, g) \in (L^p(\Omega))^2$, the system (\ref{eq:system-general-en}) with $\varphi=0=\psi$ on $\partial\Omega$ possesses a unique nonnegative solution $(u,v)\in (W^{2,p}(\Omega))^2$. Moreover, if $f+g\not\equiv 0$ in $\Omega$, then $u, v > 0$ in $\Omega$.
\end{theorem} 

Under the hypotheses of previous theorem, it is important to observe that when $\alpha_1=0=\alpha_2$, problem (\ref{eq:system-general-en}) does not admit any other solution besides this nonnegative one. However, if $0\leq\alpha_1,\alpha_2<1$ are not both zero, system \eqref{eq:system-general-en} may have another solution that is not nonnegative (see, for instance, the initial part of the proof of Theorem \ref{MP} for the case $\lambda>0$). This is natural, since $\mathcal{R}_1$ does not belong to the region of validity of {\bf (WMP)} in the case $\alpha_1$ and $\alpha_2$ are not both zero.

In \cite{MR4072483} and \cite{MR4732371} Leite and Montenegro have obtained an ABP estimate for the system \eqref{1.3} with $\alpha_1=0=\alpha_2$. Here we obtain the following extension.

\begin{theorem}[ABP estimate] \label{ABP}
Let $n \geq 2$. Assume that $(u,v) \in (W^{2,p}(\Omega))^2$ is a nonnegative nontrivial solution of the problem

\begin{equation}\label{abp}
\left\{
\begin{array}{llll}
-{\cal L}_1u \leq \lambda a(x)u^{\alpha_1} v^{\beta_1} + f(x) & {\rm in} \ \ \Omega,\\
-{\cal L}_2v \leq \mu b(x)v^{\alpha_2} u^{\beta_2} + g(x) & {\rm in} \ \ \Omega,\\
u =0= v & {\rm on} \ \ \partial\Omega,
\end{array}
\right.
\end{equation}
where $(f,g) \in (L^p(\Omega))^2$, $f, g \geq 0$ and $f+g\not\equiv 0$ in $\Omega$. If $(\lambda, \mu) \in {\cal R}_1$, then the ABP estimate holds in $\Omega$:

\begin{equation}
\|u\|_{L^\infty(\Omega)}+\|v\|_{L^\infty(\Omega)}^{\,s}
		\;\le\;
		A\Big(
		\|f\|_{L^{p}(\Omega)}
		+\|g\|_{L^{p}(\Omega)}^{\,s}
		\Big), \label{estABP}
\end{equation}
where $s$ and $A$ are the constants in \eqref{eq:def-r-en} and \eqref{eq:apriori-final-en}, respectively.
\end{theorem}

Note that we obtain an ABP estimate in a region outside the validity of the {\bf (WMP)} in the case $\alpha_1\neq 0$ or $\alpha_2\neq 0$, but this region is essential for the development of the theory for nonlinear systems. For example, this estimate plays a role in obtaining principal eigenvalues for problems in general bounded domains. The basic idea consists in approximating $\Omega$ by a sequence of expanding smooth subdomains and then considering the principal eigenvalue associated to each subdomain and studying its convergence. To this end, it is necessary to obtain the ABP estimate in the region below the principal curve (see \cite{MR4732371} for the Lane--Emden system).

We refer to \cite{MR4072483} for readers interested in a more complete historical background on the above results in a linear scalar context or in the context of systems.

Finally, we show that an antimaximum principle occurs for the system \eqref{1.3} with $\alpha_1=0=\alpha_2$ when the parameters move slightly above the curve $\mathcal{C}_1$. This type of result is fundamental in resonance phenomena and bifurcation; see, e.g., \cite{MR1865948,MR2609541,MR2644914}. A classical reference in the scalar case is \cite{MR550042} for smooth domains and \cite{MR1340547} for general domains, while counterparts for systems involving the $p$-Laplacian and fully nonlinear operators can be found in \cite[Theorem 3.4]{MR2742481} and \cite{MR4525730}, respectively. Here we extend these results to the system \eqref{1.3} with $\alpha_1=0=\alpha_2$ as follows.

\begin{theorem}[Antimaximum principle: case $\alpha_1=0=\alpha_2$]
\label{AMP}
		Given $f,g\in L^p(\Omega)$ with $f\ge0$, $g\ge0$, and $f+g\not\equiv 0$, consider the boundary value problem
	\begin{equation}\label{eq:system}
		\begin{cases}
			-\mathcal L_1 u=\lambda\,a(x)\,|v|^{\beta_1-1}v+f(x) & \text{in }\Omega,\\[2pt]
			-\mathcal L_2 v=\mu\,b(x)\,|u|^{\beta_2-1}u+g(x) & \text{in }\Omega,\\[2pt]
			u=v=0 & \text{on }\partial\Omega.
		\end{cases}
	\end{equation}
	Fix $(\lambda_1,\mu_1)\in\mathcal C_1$. Then there exists $\delta>0$ such that, for every $(\lambda,\mu)$ satisfying
	\[
	\lambda_1<\lambda<\lambda_1+\delta,\qquad \mu_1<\mu<\mu_1+\delta,
	\]
	every strong solution $(u,v)\in (W^{2,p}(\Omega))^2$ of \eqref{eq:system} satisfies $u<0$ and $v<0$ in $\Omega$, and $\partial_\nu u>0$, $\partial_\nu v>0$ on $\partial\Omega$.
\end{theorem}

However, we do not have a result of this type for system (\ref{1.3}) with $0\leq\alpha_1,\alpha_2<1$ and $\alpha_1,\alpha_2$ not both zero. For example, when $0<\alpha_1,\alpha_2<1$, the region of validity of the {\bf (WMP)} is $ -(\overline{{\cal R}_1} \setminus {\cal C}_1)$. Note that the first example in Remark \ref{example} shows that we cannot have an antimaximum principle below $-{\cal C}_1$. We also cannot have an antimaximum principle in the other regions of $\R^2$, see the proof of Theorem \ref{MP}.

The paper is organized as follows. In Section 2 we recall some preliminary facts and definitions and also construct important tools for proving the existence of principal eigenvalues. Section 3 is devoted to the construction of the principal curve $\mathcal{C}_1$ stated in Theorem \ref{teor2}. In Section 4 we obtain a complete characterization of {\bf (WMP)} stated in Theorem \ref{MP}. In Section 5 we obtain a lower estimate and a {\bf (WMP)} in small domains for such systems stated in Theorems \ref{lower} and \ref{sm}, respectively. In Section 6, using Theorem \ref{MP}, we obtain the existence of nonnegative solutions for the nonhomogeneous elliptic system (\ref{eq:system-general-en}) in the region of validity of the {\bf (WMP)}, namely Theorem \ref{thm:existence-outside-curve}. In Section 7 we deal with the antimaximum principle, proving Theorem \ref{AMP}. Section 8 is dedicated to the study of the existence of nonnegative solutions for the system (\ref{eq:system-general-en}) in ${\cal R}_1$, from which we obtain Theorem \ref{E1} as a consequence. Finally, in Section 9, applying Theorem \ref{E1}, we derive an ABP estimate (Theorem \ref{ABP}) for nonhomogeneous systems.

	\section{Preliminaries: Framework, Ordered Spaces, and the Resolvent}
	
In this section, we collect several fundamental tools that will be instrumental in establishing the principal curve associated with system \eqref{1.3}. 
While many of the underlying arguments are classical and can be found elsewhere, we provide a self-contained exposition for the reader’s convenience.

\subsection{The ordered space and the interior of the cone}

Since $p>n$, the Sobolev embedding theorem ensures that $W^{2,p}(\Omega)\hookrightarrow C^{1,\eta}(\overline{\Omega})$. 
We shall work in the ordered Banach space
\[
V := C^{1,\eta}_0(\overline{\Omega}), \qquad 
\|u\|_{1,\eta} := \|u\|_{C^{1,\eta}(\overline{\Omega})}.
\]
Its positive cone is defined by
\[
V_+ := \{u \in V :\ u(x) \ge 0 \ \text{for all } x \in \Omega\},
\]
and we denote by $\nu$ the outward unit normal to $\partial\Omega$. 
A key object in our analysis is the interior of this cone, which will be shown to admit the following explicit characterization:
\[
V_+^\circ := \Big\{ u \in V :\ u > 0 \ \text{in } \Omega \ \text{and} \ \inf_{\partial\Omega}(-\partial_\nu u) > 0 \Big\}=\text{int}_V(V_+).
\]
We shall prove below (Lemma~\ref{nonempty}) that $V_+^\circ$ is nonempty and coincides with the topological interior of $V_+$ in $V$.

We use the following notation for the partial order in $V$: 
\[
u \preceq v \quad\Longleftrightarrow\quad v-u \in V_+, \qquad
u \ll v \quad\Longleftrightarrow\quad v-u \in V_+^\circ,
\]
and denote order intervals by \([u,v] := \{w \in V :\ u \preceq w \preceq v\}\).

A map \(T:V_+\to V_+\) is said to be
\emph{order-preserving} if \(u\preceq v \Rightarrow Tu\preceq Tv\);
\emph{strongly positive} if \(u\in V_+\setminus\{0\} \Rightarrow Tu\in V_+^\circ\);
\emph{positively \(k\)-homogeneous} if \(T(tu)=t^{\,k}Tu\) for all \(t>0\);
and \emph{compact} if \(T(B)\) is relatively compact in \(V\) whenever \(B\subset V_+\) is bounded.

For each \(i\in\{1,2\}\) and \(f\in C^{0,\eta}(\overline{\Omega})\), let \(S_i(f)\in V\) denote the unique solution of
\[
-\,\mathcal L_i w = f \quad \text{in } \Omega, \qquad w = 0 \quad \text{on } \partial\Omega.
\]
Then \(S_i : C^{0,\eta}(\overline{\Omega}) \to V\) is linear, continuous, positive
(\(f \ge 0 \Rightarrow S_i(f) \ge 0\)) and compact; moreover, \(f > 0\) implies
\(S_i(f) \in V_+^\circ\) (by the Scalar Strong Maximum Principle and Hopf’s lemma).
These properties will be essential for establishing the order, compactness, and strong positivity
of the operators \(J_i\) and \(K_i\) introduced later.

\begin{lemma}[Non-emptiness and characterization of $V_+^\circ$]\label{nonempty}
	The set $V_+^\circ$ is nonempty. Moreover, for $u\in V$ the following are equivalent:
	\[
	u\in \text{int}_V(V_+)
	\quad\Longleftrightarrow\quad
	u>0 \text{ in }\Omega \ \text{and}\ \inf_{\partial\Omega}(-\partial_\nu u)>0.
	\]
	Equivalently, $V_+^\circ$ coincides with the topological interior of $V_+$ in $V$.
\end{lemma}
	
	\begin{proof}
		To prove non-emptiness, consider the Dirichlet problem
		\[
		\begin{cases}
			-\,\mathcal{L} u = 1 & \text{in }\Omega,\\
			u=0 & \text{on }\partial\Omega,
		\end{cases}
		\]
		for any uniformly elliptic nondivergence operator $\mathcal{L}$ (e.g., $\mathcal{L}=\mathcal{L}_1$).
		Then, there exists a unique solution
		$\psi\in W^{2,p}(\Omega)\cap W^{1,p}_0(\Omega)\subset V$. Since the right-hand side is strictly positive, the Scalar Strong Maximum Principle yields $\psi>0$ in $\Omega$, and Hopf's lemma gives
		$-\partial_\nu\psi(x_0)>0$ for every $x_0\in\partial\Omega$. The continuity of $\partial_\nu\psi$ on
		$\partial\Omega$ and the compactness of $\partial\Omega$ imply
		\[
		\kappa:=\min_{\partial\Omega}(-\partial_\nu\psi)>0.
		\]
		Hence $\psi\in V_+^\circ$, showing $V_+^\circ\neq\varnothing$.
		
		We now prove the characterization. First, assume $u\in V$ satisfies $u>0$ in $\Omega$ and
		$\inf\limits_{\partial\Omega}(-\partial_\nu u)=\kappa>0$. Since $u\in C^{1,\eta}(\overline\Omega)$, there exists
		$\delta>0$ such that $u(x)\ge m>0$ on the inner strip $\{x\in\Omega:\ \operatorname{dist}(x,\partial\Omega)\ge\delta\}$.
		Moreover, by the $C^{1,\eta}$ regularity up to the boundary and a flattening-of-the-boundary argument along normal
		coordinates, one has the non-tangential lower bound
		\[
		u(x)\ \ge\ \frac{\kappa}{2}\,\operatorname{dist}(x,\partial\Omega)\quad\text{whenever }\operatorname{dist}(x,\partial\Omega)<\delta,
		\]
		upon decreasing $\delta$ if necessary. Consequently, there exists $\varepsilon>0$ such that every $w\in V$ with
		$\|w-u\|_{1,\eta}<\varepsilon$ satisfies $w\ge m/2>0$ on
		$\{x\in\Omega:\ \operatorname{dist}(x,\partial\Omega)\ge\delta\}$ and $-\partial_\nu w\ge \kappa/2>0$ on $\partial\Omega$, which implies $w>0$ also on
		$\{x\in\Omega:\ \operatorname{dist}(x,\partial\Omega)<\delta\}$ by the previous boundary estimate. Therefore $w\in V_+^\circ$, showing that $u$ is an
		interior point of $V_+$. This proves the inclusion
		\[
		\{u\in V:\ u>0 \text{ in }\Omega,\ \inf_{\partial\Omega}(-\partial_\nu u)>0\}\ \subset\ \operatorname{int}_V(V_+).
		\]
		
		Conversely, suppose $u\in \operatorname{int}_V(V_+)$ but $\inf\limits_{\partial\Omega}(-\partial_\nu u)=0$. Then there exists
		$x_0\in\partial\Omega$ such that $-\partial_\nu u(x_0)=0$. Let $\chi\in C^\infty_c(\mathbb{R}^N)$ be a cutoff supported
		in a small neighborhood of $x_0$ and equal to $1$ near $x_0$, and let $d(x)=\operatorname{dist}(x,\partial\Omega)$.
		For $\varepsilon>0$ small, set
		\[
		w_\varepsilon:=u+\varepsilon\,\chi\, d(\cdot).
		\]
		Since $d\in C^{1,1}$ near $\partial\Omega$, we have $w_\varepsilon\to u$ in $C^{1,\eta}(\overline\Omega)$ as
		$\varepsilon\downarrow0$. However, along the outward normal at $x_0$, one has
		\[
		\partial_\nu w_\varepsilon(x_0)=\partial_\nu u(x_0)+\varepsilon\,\partial_\nu d(x_0)
		=0-\varepsilon<0,
		\]
		and $w_\varepsilon=0$ on $\partial\Omega$, hence $w_\varepsilon<0$ in $\Omega$ near $x_0$ for $\varepsilon>0$ small.
		Thus $w_\varepsilon\notin V_+$ while $\|w_\varepsilon-u\|_{1,\eta}$ can be made arbitrarily small, contradicting the
		assumption that $u$ is an interior point of $V_+$. Therefore $\inf\limits_{\partial\Omega}(-\partial_\nu u)>0$.
		The equivalence and the identification of $V_+^\circ$ with the topological interior of $V_+$ in $V$ follow.
	\end{proof}

	\subsection{The auxiliary operators $J_1$ and $J_2$}

We now show that, for each fixed positive function, the auxiliary boundary value problems defining $J_1$ and $J_2$ admit a unique positive solution.

Let $v \in V_+$. We define $u = J_1(v)$ as the unique solution of
\begin{equation}\label{eq:J1_en}
	\begin{cases}
		-\,\mathcal{L}_1 u = a(x)\,u^{\alpha_1}\,v^{\beta_1}, & \text{in } \Omega, \\[2mm]
		u = 0, & \text{on } \partial\Omega,
	\end{cases}
\end{equation}
and, for $u \in V_+$, we define $v = J_2(u)$ as the unique solution of
\begin{equation}\label{eq:J2_en}
	\begin{cases}
		-\,\mathcal{L}_2 v = b(x)\,v^{\alpha_2}\,u^{\beta_2}, & \text{in } \Omega, \\[2mm]
		v = 0, & \text{on } \partial\Omega.
	\end{cases}
\end{equation}

\begin{lemma}[Existence, uniqueness, and strong positivity]\label{lem:existJ_en}
	For every $v\in V_+$ with $v\neq 0$, problem~\eqref{eq:J1_en} admits a unique solution $u\in V$, and moreover $u\in V_+^\circ$. The same conclusion holds for~\eqref{eq:J2_en}.
\end{lemma}

	\begin{proof} The solution is obtained by constructing ordered sub- and supersolutions and applying monotone iteration.
		
		Fix \(v\in V_+\) with \(v\neq 0\). Set $m:=a\,v^{\beta_1}\in L^{p}(\Omega)$.
		Let $\varphi>0$ be the principal eigenfunction of
		\[
		-\,\mathcal{L}_1\varphi=\lambda_1\,m\,\varphi,\qquad \varphi|_{\partial\Omega}=0,
		\]
		normalized so that $0<\varphi\le 1$ in $\Omega$ (existence and simplicity by Krein–Rutman and elliptic theory).
		For $c>0$ sufficiently small, 
		\[
		-\,\mathcal{L}_1(c\varphi)=c\,\lambda_1 m\,\varphi \;\le\; c^{\alpha_1} m\,\varphi^{\alpha_1}
		= m\,(c\varphi)^{\alpha_1},
		\]
		because $\alpha_1<1$ and $\varphi\le1$. Hence $c\varphi$ is a subsolution to~\eqref{eq:J1_en}.
		
		Let $S_1(m)$ denote the unique solution of $-\,\mathcal{L}_1 w=m$ in $\Omega$, $w=0$ on $\partial\Omega$.
		Choose $\tau>0$ such that $\tau^{1-\alpha_1}\ge \|S_1(m)\|_{L^\infty(\Omega)}^{\alpha_1}$, and set $z_0:=\tau S_1(m)\ge0$. Then
		\[
		-\,\mathcal{L}_1 z_0=\tau m \;\ge\; m\,z_0^{\alpha_1},
		\]
		so $z_0$ is a supersolution to~\eqref{eq:J1_en}. Moreover, taking $\tau\ge c\,\lambda_1$ yields
		$z_0\ge c\varphi$ in $\Omega$ (by the Scalar Weak Maximum Principle applied to $z_0-c\varphi$).

		Define recursively, for $k\ge1$,
		\[
		\begin{cases}
			-\,\mathcal{L}_1 z_k = m\,z_{k-1}^{\alpha_1} & \text{in }\Omega,\\
			z_k = 0 & \text{on }\partial\Omega.
		\end{cases}
		\]
		The linear comparison principle gives, for all $k\ge1$,
		\[
		c\varphi \ \le\ z_k \ \le\ z_{k-1} \ \le\ z_0 \quad \text{in }\Omega.
		\]
		Thus $(z_k)$ is pointwise decreasing and uniformly bounded below by $c\varphi>0$.

		Let $f_k:=m\,z_{k-1}^{\alpha_1}$. Applying the Calderón-Zygmund theory and the Sobolev embedding for the Dirichlet problem yields a constant
		$C_1>0$ such that
		\[
			\|z_k\|_{W^{2,p}(\Omega)} \;\le\; C_1\,\|f_k\|_{L^{p}(\Omega)}.
		\]
		Note that,
		\[
		\|f_k\|_{L^{p}(\Omega)}
		\;\le\;
		\|m\|_{L^{p}(\Omega)}\|z_{k-1}\|_{L^\infty(\Omega)}^{\alpha_1}.
		\]
		Since $0\le z_{k-1}\le z_0$, we have $\|z_{k-1}\|_{L^\infty(\Omega)}\le\|z_0\|_{L^\infty(\Omega)}$. 
		Thus,
		\[
			\sup_{k\ge0}\,\|z_k\|_{W^{2,p}(\Omega)} \;<\; \infty.
		\]
		The embedding $W^{2,p}(\Omega)\hookrightarrow C^{1}(\overline\Omega)$ is compact; hence, up to subsequences,
		$z_k\to u$ in $C^{1}(\overline\Omega)$. Since $(z_k)$ is decreasing in $C(\overline{\Omega})$ and bounded below by $c\varphi$,
		the limit is unique and the whole sequence converges to $u$. Passing to the limit in the iteration equation yields
		\[
		-\,\mathcal{L}_1 u = m\,u^{\alpha_1}\quad\text{in }\Omega,\qquad u|_{\partial\Omega}=0,
		\]
		i.e., $u$ solves~\eqref{eq:J1_en}. If $u_1,u_2$ are two solutions, the (nonlinear) comparison principle implies
		$u_1\equiv u_2$. Finally, since $u\ge c\varphi>0$ in $\Omega$ and $u=0$ on $\partial\Omega$, the Scalar Strong Maximum Principle
		and Hopf's lemma give $u\in V_+^\circ$.
	\end{proof}

	
	\begin{lemma}[Monotonicity]\label{lem:mono_en}
		If $0 \le v_1 \le v_2$ in $\Omega$ and $v_1 \neq 0$, then
		\[
		J_1(v_1) \le J_1(v_2) \quad \text{in } \Omega.
		\]
	\end{lemma}
	
	\begin{proof}
		Let $v_1, v_2 \in V_+$ with $0 \le v_1 \le v_2$ in $\Omega$ and $v_1 \neq 0$. For $l=1,2$, define
		\[
		m_l := a\,v_l^{\beta_1}.
		\]
		Since $a \ge 0$ and $v_1 \le v_2$, it follows that $m_1 \le m_2$ pointwise in $\Omega$.
		
		Let $w := S_1(m_2)$ be the unique solution of $-\mathcal{L}_1 w = m_2$ in $\Omega$, with $w=0$ on $\partial\Omega$. Choose $\tau > 0$ large enough so that $\tau^{1-\alpha_1} \ge \|w\|_{L^\infty(\Omega)}^{\alpha_1}$ and set $z_0 := \tau w \ge 0$. Then
		\[
		-\,\mathcal{L}_1 z_0 = \tau m_2 \ge m_2 z_0^{\alpha_1}.
		\]
		Since $m_1 \le m_2$ and $z_0 \ge 0$, we also have $-\,\mathcal{L}_1 z_0 \ge m_1 z_0^{\alpha_1}$. Hence $z_0$ is a supersolution for both problems with weights $m_1$ and $m_2$.
		
		Starting from $z_0$, for each $l=1,2$ and $k \ge 1$ define
		\[
		\begin{cases}
			-\,\mathcal{L}_1 z^{(l)}_k = m_l \,(z^{(l)}_{k-1})^{\alpha_1} & \text{in }\Omega,\\
			z^{(l)}_k = 0 & \text{on }\partial\Omega,
		\end{cases}
		\]
		with $z^{(l)}_0 = z_0$. By the linear comparison principle,
		\[
		0 \le z^{(l)}_k \le z^{(l)}_{k-1} \le z_0 \quad \text{for all } k \ge 1.
		\]
		We claim that $z^{(1)}_k \le z^{(2)}_k$ for all $k \ge 0$. Indeed, for $k=1$,
		\[
		-\,\mathcal{L}_1 z^{(1)}_1 = m_1 (z_0)^{\alpha_1} \le m_2 (z_0)^{\alpha_1} = -\,\mathcal{L}_1 z^{(2)}_1,
		\]
		so that $z^{(1)}_1 \le z^{(2)}_1$ by comparison. If $z^{(1)}_{k-1} \le z^{(2)}_{k-1}$, then the monotonicity of $s \mapsto s^{\alpha_1}$ gives
		\[
		m_1 (z^{(1)}_{k-1})^{\alpha_1} \le m_2 (z^{(2)}_{k-1})^{\alpha_1},
		\]
		and another application of the comparison principle yields $z^{(1)}_k \le z^{(2)}_k$.
		
		As argued in Lemma~\ref{lem:existJ_en}, there exist $c > 0$ and a positive eigenfunction $\varphi^{(1)}$ of
		\[
		-\,\mathcal{L}_1 \varphi^{(1)} = \lambda_1 m_1 \varphi^{(1)}, \quad \varphi^{(1)}|_{\partial\Omega} = 0,
		\]
		normalized so that $0 < \varphi^{(1)} \le 1$, such that
		\[
		0 < c\,\varphi^{(1)} \le z^{(1)}_k \le z^{(2)}_k \quad \forall\, k \ge 0.
		\]
		
		The sequences $(z^{(l)}_k)$ are monotone decreasing and uniformly bounded in $W^{2,p}(\Omega)$.  
		Then, we may extract convergent subsequences in $C^{1}(\overline{\Omega})$ such that
		\[
		z^{(l)}_k \longrightarrow u_l \quad \text{in } C^{1}(\overline{\Omega}).
		\]
		The limit $u_l$ satisfies
		\[
		-\,\mathcal{L}_1 u_l = m_l\,u_l^{\alpha_1}, \quad u_l|_{\partial\Omega} = 0,
		\]
		that is, $u_l = J_1(v_l)$.  
		Passing to the limit in $z^{(1)}_k \le z^{(2)}_k$ yields $u_1 \le u_2$ in $\Omega$, completing the proof.
	\end{proof}
	
	\begin{lemma}[Homogeneity]\label{lem:homo_en}
		For $c>0$,
		\[
		J_1(c v)=c^{\frac{\beta_1}{1-\alpha_1}}\,J_1(v),
		\qquad
		J_2(c u)=c^{\frac{\beta_2}{1-\alpha_2}}\,J_2(u).
		\]
	\end{lemma}
	
	\begin{proof}
		Fix $c>0$ and \(v\in V_+\) with \(v\neq 0\). Let $u=J_1(v)$ be the unique solution of
		\[
		-\,\mathcal{L}_1 u = a\,u^{\alpha_1} v^{\beta_1}\quad \text{in }\Omega,
		\qquad
		u=0\quad \text{on }\partial\Omega
		\]
		(see Lemma~\ref{lem:existJ_en}). For $k>0$, define $u_k:=k\,u$. By the linearity of $\mathcal{L}_1$ we have
		\[
		-\,\mathcal{L}_1(u_k)=k\,(-\mathcal{L}_1 u)=k\,a\,u^{\alpha_1} v^{\beta_1}.
		\]
		Since $u=u_k/k$, it follows that
		\[
		u^{\alpha_1}=\Big(\tfrac{u_k}{k}\Big)^{\alpha_1}=k^{-\alpha_1} u_k^{\alpha_1},
		\]
		and hence
		\[
		-\,\mathcal{L}_1(u_k)=k^{1-\alpha_1} a\,u_k^{\alpha_1} v^{\beta_1}\quad \text{in }\Omega,
		\qquad
		u_k=0\quad \text{on }\partial\Omega.
		\]
		
		Consider now the problem with data $cv$:
		\[
		-\,\mathcal{L}_1 w = a\,w^{\alpha_1} (cv)^{\beta_1}
		= c^{\beta_1} a\,w^{\alpha_1} v^{\beta_1}\quad \text{in }\Omega,
		\qquad
		w=0\quad \text{on }\partial\Omega.
		\]
		Choose
		\[
		k=c^{\frac{\beta_1}{1-\alpha_1}}>0.
		\]
		Then $k^{1-\alpha_1}=c^{\beta_1}$, so the function $u_k$ satisfies
		\[
		-\,\mathcal{L}_1(u_k)=c^{\beta_1} a\,u_k^{\alpha_1} v^{\beta_1}
		= a\,u_k^{\alpha_1} (cv)^{\beta_1}\quad \text{in }\Omega,
		\qquad
		u_k=0\quad \text{on }\partial\Omega.
		\]
		By the existence and uniqueness result (Lemma~\ref{lem:existJ_en}), the solution of the problem with datum $cv$ is unique; therefore
		\[
		J_1(cv)=u_k=k\,u = c^{\frac{\beta_1}{1-\alpha_1}}\,J_1(v).
		\]
		
		The proof of the identity for $J_2$ is identical, replacing $(\alpha_1,\beta_1,a,\mathcal{L}_1,v)$ by $(\alpha_2,\beta_2,b,\mathcal{L}_2,u)$.
	\end{proof}

	\subsection{Composition and Spectral Problem}
	
	Define the return map $T: V_+ \times V_+ \to V_+ \times V_+$ by
	\[
	T(u,v) := \big(J_1(v),\, J_2(u)\big),
	\]
	and the associated compositions
	\[
	K_1 := J_1 \circ J_2: V_+ \to V_+,
	\qquad
	K_2 := J_2 \circ J_1: V_+ \to V_+.
	\]
	
	\begin{lemma}[Degree of homogeneity]\label{lem:gamma}
		Both $K_1$ and $K_2$ are homogeneous of degree
		\[
		\gamma \;=\; \frac{\beta_1\beta_2}{(1-\alpha_1)(1-\alpha_2)}.
		\]
		In particular, they are $1$-homogeneous if and only if
		\[
		\beta_1\beta_2 \;=\; (1-\alpha_1)(1-\alpha_2).
		\]
	\end{lemma}
	
	\begin{proof}
		Fix $u \in V_+$ with \(u\neq 0\) and $c>0$. By Lemma~\ref{lem:homo_en},
		\[
		J_2(cu) \;=\; c^{\frac{\beta_2}{1-\alpha_2}}\, J_2(u).
		\]
		Applying $J_1$ and using Lemma~\ref{lem:homo_en} again,
		\[
		K_1(cu)
		= J_1\big(J_2(cu)\big)
		= J_1\!\left(c^{\frac{\beta_2}{1-\alpha_2}}\,J_2(u)\right)
		= \left(c^{\frac{\beta_2}{1-\alpha_2}}\right)^{\frac{\beta_1}{1-\alpha_1}} J_1(J_2(u))
		= c^{\frac{\beta_1\beta_2}{(1-\alpha_1)(1-\alpha_2)}}\, K_1(u).
		\]
		Hence $K_1$ is homogeneous of degree $\gamma$ as claimed. The proof for $K_2$ is identical, exchanging the roles of $(J_1,\alpha_1,\beta_1)$ and $(J_2,\alpha_2,\beta_2)$.  
		The characterization of $1$-homogeneity follows from $\gamma=1$, i.e., $\beta_1\beta_2=(1-\alpha_1)(1-\alpha_2)$.
	\end{proof}
	
	\begin{lemma}[Characterization via $J_1,J_2$]\label{lem:carac}
		A pair $(u,v)\in V_+^\circ \times V_+^\circ$ solves
		\[
	\begin{cases}
		-\,{\cal L}_1 u = \lambda\, a(x)\, u^{\alpha_1}v^{\beta_1} & \text{in } \Omega, \\[2pt]
		-\,{\cal L}_2 v = \mu\, b(x)\, v^{\alpha_2}u^{\beta_2} & \text{in } \Omega, \\[2pt]
		u = v = 0 & \text{on } \partial\Omega,
	\end{cases}
\]
		if and only if
		\[
		J_1(v)=\lambda^{-1/(1-\alpha_1)}\,u,
		\qquad
		J_2(u)=\mu^{-1/(1-\alpha_2)}\,v.
		\]
	\end{lemma}
	
	\begin{proof}
		Let $w := J_1(v)$ be the unique solution of
		\[
		-\,\mathcal{L}_1 w = a\,w^{\alpha_1} v^{\beta_1}\quad \text{in }\Omega,
		\qquad
		w=0 \text{ on } \partial\Omega.
		\]
		For any $t>0$, by linearity of $\,\mathcal{L}_1$,
		\[
		-\,\mathcal{L}_1(tw)=t(-\mathcal{L}_1 w)=t\,a\,w^{\alpha_1}v^{\beta_1}
		= t^{1-\alpha_1}\, a\, (tw)^{\alpha_1} v^{\beta_1},
		\]
		since $w^{\alpha_1}=t^{-\alpha_1}(tw)^{\alpha_1}$. Choosing $t=\lambda^{1/(1-\alpha_1)}$ gives
		\[
		-\,\mathcal{L}_1(tw)=\lambda\, a\, (tw)^{\alpha_1} v^{\beta_1}.
		\]
		By uniqueness (Lemma~\ref{lem:existJ_en}), the solution of the first equation is $u=tw=\lambda^{1/(1-\alpha_1)} J_1(v)$, i.e.,
		\[
		J_1(v)=\lambda^{-1/(1-\alpha_1)}\,u.
		\]
		The second identity follows analogously by setting $z:=J_2(u)$ and taking $s=\mu^{1/(1-\alpha_2)}$.
		
		Conversely, assume $J_1(v)=\lambda^{-1/(1-\alpha_1)} u$ and $J_2(u)=\mu^{-1/(1-\alpha_2)} v$. Then, using Lemma~\ref{lem:homo_en} and the defining equation for $J_1(v)$,
		\[
		-\,\mathcal{L}_1 u
		= -\,\mathcal{L}_1\!\left(\lambda^{1/(1-\alpha_1)} J_1(v)\right)
		= \lambda^{1/(1-\alpha_1)} \big(-\mathcal{L}_1 J_1(v)\big)
		= \lambda^{1/(1-\alpha_1)} a\, J_1(v)^{\alpha_1} v^{\beta_1}.
		\]
		Since $J_1(v)=\lambda^{-1/(1-\alpha_1)} u$, we have $J_1(v)^{\alpha_1}=\lambda^{-\alpha_1/(1-\alpha_1)} u^{\alpha_1}$, hence
		\[
		-\,\mathcal{L}_1 u = \lambda^{\frac{1-\alpha_1}{1-\alpha_1}} a\, u^{\alpha_1} v^{\beta_1}
		= \lambda\, a\, u^{\alpha_1} v^{\beta_1}.
		\]
		The boundary condition is preserved by the scaling. The second equation follows identically from $J_2(u)=\mu^{-1/(1-\alpha_2)} v$. Thus $(u,v)$ solves the system.
	\end{proof}
	
	\section{Proof of Theorem \ref{teor2}}\label{sec:proof-teor2}
	
We begin this section by introducing the quantities that will play a central role in the analysis. 
For parameters $(\lambda,\mu)\in(0,\infty)^2$, define
\[
\Lambda_1(\lambda,\mu):=\lambda^{\frac{1}{1-\alpha_1}}\mu^{\frac{1}{\beta_2}},
\qquad
\Lambda_2(\lambda,\mu):=\mu^{\frac{1}{1-\alpha_2}}\lambda^{\frac{1}{\beta_1}}.
\]
Associated with these quantities, we define the admissible sets
\[
\mathcal{D}_1:=\Big\{(\lambda,\mu)\in(0,\infty)^2:\ \Lambda_1(\lambda,\mu)=\rho_1^{-1}\Big\},
\qquad
\mathcal{D}_2:=\Big\{(\lambda,\mu)\in(0,\infty)^2:\ \Lambda_2(\lambda,\mu)=\rho_2^{-1}\Big\},
\]
where $\rho_1$ and $\rho_2$ are fixed positive constants, to be determined later via the nonlinear Krein-Rutman theorem.

To establish Theorem~\ref{teor2}, we shall prove that the operators $K_1$ and $K_2$ are compact, order-preserving, strongly positive, and $1$-homogeneous. 
Moreover, we show that $\mathcal{D}_1=\mathcal{D}_2$; denoting this common set by $\mathcal{C}$, 
the coupled system (\ref{1.3}) admits a positive solution $(u,v)\in V_+^\circ\times V_+^\circ$ if and only if $(\lambda,\mu)\in\mathcal{C}$. 
Equivalently, this condition can be expressed as
\begin{equation}\label{eq:Lambda-conds}
	\lambda^{\frac{1}{1-\alpha_1}}\mu^{\frac{1}{\beta_2}}=\rho_1^{-1}
	\qquad\Longleftrightarrow\qquad
	\lambda^{\frac{1}{\beta_1}}\mu^{\frac{1}{1-\alpha_2}}=\rho_2^{-1}.
\end{equation}
In that case, all positive solutions belong to the scaling family
\[
(u,v)=\big(s\,u',\ s^{\frac{\beta_2}{\,1-\alpha_2\,}}\,v'\big),\qquad s>0.
\]

Finally, it is convenient to introduce the constant
\[
\Lambda_0
:= \rho_1^{-\sqrt{\frac{\beta_2}{\,1-\alpha_2\,}}}
= \rho_2^{-\sqrt{\frac{\beta_1}{\,1-\alpha_1\,}}},
\]
which coincides with the quantity appearing in the statement of Theorem~\ref{teor2}. 
In what follows, these constructions will be combined with Propositions~\ref{psim2},~\ref{prop - 1} and~\ref{propos2.15} 
to complete the proof of the theorem.

	\begin{proof}[Proof of Theorem \ref{teor2}]
		By Lemma~\ref{lem:existJ_en}, the return maps $J_1,J_2:V_+\to V_+^\circ$ are
		well-defined, strongly positive, continuous and compact.
		Lemma~\ref{lem:homo_en} implies that $J_1$ and $J_2$ are homogeneous of degrees
		$\frac{\beta_1}{1-\alpha_1}$ and $\frac{\beta_2}{1-\alpha_2}$, respectively.
		Hence, by Lemma~\ref{lem:gamma}, the compositions
		$K_1:=J_1\circ J_2$ and $K_2:=J_2\circ J_1$ are homogeneous of degree
		\[
		\gamma=\frac{\beta_1\beta_2}{(1-\alpha_1)(1-\alpha_2)}=1,
		\]
		and therefore $1$-homogeneous. Their strong positivity and compactness follow
		from those of $J_1,J_2$.
		
		By the nonlinear Krein–Rutman theorem for continuous, compact,
		order-preserving, strongly positive, $1$-homogeneous maps on a Banach
		space with a solid cone, there exist principal eigenpairs
		\[
		K_1u'=\rho_1 u',\qquad u'\in V_+^\circ,\ \rho_1>0,
		\qquad\text{and}\qquad
		K_2v'=\rho_2 v',\qquad v'\in V_+^\circ,\ \rho_2>0,
		\]
		each eigenvector being unique up to positive multiples.
		Moreover, by definition, the intertwining identities
		\[
		J_1\circ K_2=K_1\circ J_1,\qquad J_2\circ K_1=K_2\circ J_2
		\]
		hold. Using the homogeneity of $J_1,J_2$ this yields the principal eigenvalue
		relations
		\begin{equation}\label{eq:rho-rel-main}
			\rho_1=\rho_2^{\frac{\beta_1}{1-\alpha_1}},
			\qquad
			\rho_2=\rho_1^{\frac{\beta_2}{1-\alpha_2}}.
		\end{equation}

		By Lemma~\ref{lem:carac}, $(u,v)\in V_+^\circ\times V_+^\circ$ solves the
		coupled system for $(\lambda,\mu)\in(0,\infty)^2$ if
		\begin{equation}\label{eq:carac-scaled-main}
			J_1(v)=\lambda^{-1/(1-\alpha_1)}u,\qquad
			J_2(u)=\mu^{-1/(1-\alpha_2)}v.
		\end{equation}
		Composing and using the homogeneity of $J_1,J_2$,
		\[
		\begin{aligned}
			K_1u
			&=J_1\!\big(J_2(u)\big)
			=J_1\!\big(\mu^{-1/(1-\alpha_2)}v\big)
			=\mu^{-\frac{\beta_1}{(1-\alpha_1)(1-\alpha_2)}}\,J_1(v)\\
			&=\mu^{-\frac{1}{\beta_2}}\ \lambda^{-\frac{1}{1-\alpha_1}}\,u
			= \Lambda_1(\lambda,\mu)^{-1}\,u,
		\end{aligned}
		\]
		where the condition (\ref{cond1}) was used.
		Similarly,
		\[
		K_2v
		=J_2\!\big(J_1(v)\big)
		=\lambda^{-\frac{1}{\beta_1}}\ \mu^{-\frac{1}{1-\alpha_2}}\,v
		=\Lambda_2(\lambda,\mu)^{-1}\,v.
		\]
		By uniqueness of the positive eigenpair for $K_i$ (up to scaling),
		\[
		\Lambda_1(\lambda,\mu)=\rho_1^{-1},
		\qquad
		\Lambda_2(\lambda,\mu)=\rho_2^{-1},
		\]
		which is exactly the condition \eqref{eq:Lambda-conds}. This proves the
		necessity.

		Assume \eqref{eq:Lambda-conds} holds. Take principal eigenvectors $u',v'$ as
		above and set
		\[
		u:=u',\qquad v:=\mu^{\frac{1}{1-\alpha_2}}\,J_2(u').
		\]
		Then $J_2(u)=\mu^{-1/(1-\alpha_2)}v$ by construction. Moreover,
		\[
		J_1(v)
		=J_1\!\Big(\mu^{\frac{1}{1-\alpha_2}}J_2(u')\Big)
		=\mu^{\frac{\beta_1}{(1-\alpha_1)(1-\alpha_2)}}\,K_1u'
		=\mu^{\frac{1}{\beta_2}}\rho_1\,u'
		=\lambda^{-1/(1-\alpha_1)}u,
		\]
		where we used the condition (\ref{cond1}) and $\Lambda_1^{-1}=\rho_1$.
		Hence \eqref{eq:carac-scaled-main} holds and $(u,v)$ solves the system.

		Let $(u,v)$ be any positive solution for $(\lambda,\mu)$ and write it as
		\[
		u=s\,u',\qquad v=\tilde s\,v',\qquad s,\tilde s>0,
		\]
		where $K_1u'=\rho_1u'$, $K_2v'=\rho_2v'$. From
		\eqref{eq:carac-scaled-main} we have $J_2(u)=\mu^{-1/(1-\alpha_2)}v$.
		Using the homogeneity of $J_2$,
		\begin{equation}\label{eq:rel1-main}
			J_2(su')=s^{\frac{\beta_2}{1-\alpha_2}}\,J_2(u')
			=\mu^{-\frac{1}{1-\alpha_2}}\,\tilde s\,v'.
		\end{equation}
		Now $J_2(u')$ is an eigenvector of $K_2$:
		\[
		K_2\big(J_2u'\big)=J_2(K_1u')=J_2(\rho_1u')
		=\rho_1^{\frac{\beta_2}{1-\alpha_2}}\,J_2(u').
		\]
		By \eqref{eq:rho-rel-main}, $\rho_1^{\frac{\beta_2}{1-\alpha_2}}=\rho_2$, so
		$J_2(u')$ and $v'$ are eigenvectors of $K_2$ for the same eigenvalue; by
		uniqueness in the cone, there exists $\theta>0$ with $J_2(u')=\theta v'$.
		Plugging this into \eqref{eq:rel1-main} gives
		\[
		s^{\frac{\beta_2}{1-\alpha_2}}\theta=\mu^{-\frac{1}{1-\alpha_2}}\,\tilde s.
		\]
		Since $v'$ is defined up to a positive multiple, renormalize it by
		\[
		\widehat v':=\mu^{\frac{1}{1-\alpha_2}}\theta\,v'
		\quad\Longrightarrow\quad
		J_2(u')=\mu^{-\frac{1}{1-\alpha_2}}\widehat v'.
		\]
		With this normalization the relation simplifies to
		\[
		\tilde s=s^{\frac{\beta_2}{1-\alpha_2}},
		\]
		so every positive solution has the form
		\[
		(u,v)=\big(s\,u',\ s^{\frac{\beta_2}{\,1-\alpha_2\,}}\,\widehat v'\big),\qquad s>0.
		\]
		Since $\widehat v'$ is still a principal eigenvector of $K_2$ (eigenvectors are
		unique up to positive multiples), this yields the claimed scaling family.
		
		Thus, Property~\eqref{prop:simplicity}, has already been established. Moreover, we obtain a part of~\eqref{prop:uniqueness}, that is,	$(\lambda, \mu) \in (\mathbb{R}_+^*)^2$ is a principal eigenvalue of the system~\eqref{1.3} 
	if and only if $(\lambda,\mu)\in \mathcal{C}_1$. To conclude~\eqref{prop:uniqueness}, we need to verify that there is no principal eigenvalue in $\mathbb{R}^2$ outside of $\mathcal{C}_1$.

		The properties \eqref{prop:simplicity2}, \eqref{prop:isolation} and \eqref{prop:monotonicity} rely on Propositions~\ref{psim2},~\ref{prop - 1} and~\ref{propos2.15}, respectively.
			\end{proof}

\begin{proposition}\label{psim2}
	Let $(\varphi,\psi)\in V\times V$ be an eigenfunction associated to $(\lambda,\mu)\in \mathcal{C}_1$. If $\alpha_1=0=\alpha_2$, then either
	$(\varphi,\psi)\in V_+^\circ\times V_+^\circ$ or  $-(\varphi,\psi)\in V_+^\circ\times V_+^\circ$. Now, if $0 \le \alpha_1, \alpha_2 < 1$, not both zero, then  $(\varphi,\psi)\in V_+^\circ\times V_+^\circ$.
\end{proposition}

\begin{proof}
	For the case $\alpha_1=0=\alpha_2$ see \cite{MR1765542}.
	We now prove the case $0<\alpha_1,\alpha_2<1$.
	Let $(u,v)$ be an eigenfunction of problem~\eqref{1.3} associated with $(\lambda_1,\mu_1)\in\mathcal C_1$.
	Proceeding by contradiction, assume that $u$ or $v$ is negative somewhere in $\Omega$.
	
	If $uv\le0$ in $\Omega$, then
	\[
	\left\{
	\begin{array}{lll}
		-{\cal L}_1(-u)\ge0,\\
		-{\cal L}_2(-v)\ge0,
	\end{array}
	\right.
	\]
	in $\Omega$ a.e. and $u=v=0$ on $\partial\Omega$.
	By the Scalar Strong Maximum Principle, we obtain $u<0$ and $v<0$ in $\Omega$, contradicting the assumption $uv\le0$ in $\Omega$.
	
	If $uv\ge0$ in $\Omega$, then
	\[
	\left\{
	\begin{array}{lll}
		-{\cal L}_1u\ge0,\\
		-{\cal L}_2v\ge0,
	\end{array}
	\right.
	\]
	in $\Omega$ a.e. and $u=v=0$ on $\partial\Omega$.
	By the Scalar Strong Maximum Principle, we obtain $u>0$ and $v>0$ in $\Omega$, contradicting the hypothesis that $u$ or $v$ is negative somewhere in $\Omega$.
	
	Therefore, neither $uv\le0$ nor $uv\ge0$ holds in $\Omega$. Hence,
	\[
	H_1:=\{x\in\Omega:uv>0\}
	\qquad\text{and}\qquad
	H_2:=\{x\in\Omega:uv<0\}
	\]
	are both nonempty.
	
	Let $(\varphi,\psi)$ be a positive eigenfunction corresponding to $(\lambda_1,\mu_1)$ and consider
	\(
	\Gamma
	:=
	\{\gamma>0:\ \varphi>-\gamma u
	\ \text{and}\
	\psi>-\gamma^r v
	\ \text{in }\Omega\}.
	\)
	By Hopf's lemma, $\Gamma$ is nonempty and bounded above.
	Set
	\(
	\gamma^\ast:=\sup\Gamma.
	\)
	
	First, suppose that
	\(
	\varphi\ge\gamma^\ast u \mbox{ and }
	\psi\ge(\gamma^\ast)^r v \mbox{ in }H_2.
	\)
	Then
	\[
	\left\{
	\begin{array}{lll}
		-{\cal L}_1(\varphi+\gamma^\ast u)
		=
		\lambda_1a(x)\varphi^{\alpha_1}\psi^{\beta_1}
		+
		\lambda_1a(x)\gamma^\ast
		|u|^{\alpha_1-1}u
		|v|^{\beta_1-1}v
		\ge0,
		\\[3pt]
		-{\cal L}_2(\psi+(\gamma^\ast)^rv)
		=
		\mu_1b(x)\psi^{\alpha_2}\varphi^{\beta_2}
		+
		\mu_1b(x)(\gamma^\ast)^r
		|v|^{\alpha_2-1}v
		|u|^{\beta_2-1}u
		\ge0,
	\end{array}
	\right.
	\]
	in $\Omega$ a.e., and \(
	\varphi+\gamma^\ast u
	=
	\psi+(\gamma^\ast)^r v
	=
	0
	\mbox{ on }\partial\Omega.
	\)
	Applying the Scalar Strong Maximum Principle and Hopf's lemma to each equation, we obtain
	\[
	\varphi+\gamma^\ast u\in V_+^\circ,
	\qquad
	\psi+(\gamma^\ast)^r v\in V_+^\circ.
	\]
	Hence, for $\varepsilon>0$ sufficiently small,
	\[
	\varphi>-(\gamma^\ast+\varepsilon)u
	\qquad\text{and}\qquad
	\psi>-(\gamma^\ast+\varepsilon)^rv
	\quad\text{in }\Omega,
	\]
	contradicting the definition of $\gamma^\ast$.
	Thus, $u,v\ge0$ in $\Omega$.
	
	Otherwise, there exists a nonempty region in $H_2$ such that
	\[
	\varphi<\gamma^\ast u
	\qquad\text{or}\qquad
	\psi<(\gamma^\ast)^rv.
	\]
	Consider
	\(
	\Gamma_0
	:=
	\{\gamma>0:\ 
	\varphi>\gamma\gamma^\ast u
	\ \text{and}\
	\psi>\gamma^r(\gamma^\ast)^r v
	\ \text{in }\Omega
	\}.
	\)
	Again, by Hopf's lemma, $\Gamma_0$ is nonempty and bounded above.
	Set
	\(
	\gamma_0:=\sup\Gamma_0<1.
	\)
	Then
	\(
	\varphi\ge -\gamma_0\gamma^\ast u,
	\psi\ge -\gamma_0^r(\gamma^\ast)^r v,
	\)
	and
	\(
	\varphi\ge\gamma_0\gamma^\ast u,
	\psi\ge\gamma_0^r(\gamma^\ast)^r v
	 \mbox{ in }\Omega.
	\)
	Hence,
	\[
	\left\{
	\begin{array}{lll}
		-{\cal L}_1(\varphi-\gamma_0\gamma^\ast u)
		=
		\lambda_1a(x)\varphi^{\alpha_1}\psi^{\beta_1}
		-
		\lambda_1a(x)\gamma_0\gamma^\ast
		|u|^{\alpha_1-1}u
		|v|^{\beta_1-1}v
		\ge0,
		\\[3pt]
		-{\cal L}_2(\psi-\gamma_0^r(\gamma^\ast)^rv)
		=
		\mu_1b(x)\psi^{\alpha_2}\varphi^{\beta_2}
		-
		\mu_1b(x)\gamma_0^r(\gamma^\ast)^r
		|v|^{\alpha_2-1}v
		|u|^{\beta_2-1}u
		\ge0,
	\end{array}
	\right.
	\]
	in $\Omega$ a.e., and
	\(
	\varphi-\gamma_0\gamma^\ast u
	=\psi-\gamma_0^r(\gamma^\ast)^rv
	=0\mbox{ on }\partial\Omega.
	\)
	Applying the Scalar Strong Maximum Principle and Hopf's lemma, we obtain
	\[
	\varphi-\gamma_0\gamma^\ast u\in V_+^\circ,
	\qquad
	\psi-\gamma_0^r(\gamma^\ast)^rv\in V_+^\circ.
	\]
	Therefore, for $\varepsilon>0$ sufficiently small,
	\(
	\varphi>(\gamma_0+\varepsilon)\gamma^\ast u
	\mbox{ and }
	\psi>(\gamma_0+\varepsilon)^r(\gamma^\ast)^rv
	\text{ in }\Omega,
	\)
	contradicting the definition of $\gamma_0$.
	
	Hence $u,v\ge0$ in $\Omega$ and, by the Scalar Strong Maximum Principle,
	\(
	(u,v)\in V_+^\circ\times V_+^\circ.
	\)
	
	The cases $\alpha_1=0$ or $\alpha_2=0$ follow similarly.
\end{proof}

From the proof of Proposition~\ref{psim2}, for $0\le\alpha_1,\alpha_2<1$, not both zero, we observe that if $(u,v)$ is an eigenfunction of problem~\eqref{1.3} associated with $(\lambda,\mu)$ in the first quadrant above the curve $\mathcal{C}_1$, then the sets  $\{x\in \Omega:uv> 0\}$ and $\{x\in \Omega:uv< 0\}$ are nonempty.

Thus, there is no principal eigenvalue in  $\mathbb{R}^2$ outside $\mathcal{C}_1$. This completes the proof of property~\eqref{prop:uniqueness}.

	\begin{proposition} \label{prop - 1} The principal curve $\mathcal{C}_1$ is locally isolated above.
\end{proposition}
\begin{proof} Suppose that the claim is false. Then, there are $(\lambda_1,\mu_1)\in\mathcal{C}_1$ and a sequence of eigenvalues pairs $((\lambda_k,\mu_k))_{k \geq 1}$ contained in $B_{\varepsilon_k}(\lambda_1,\mu_1)\cap\overline{{\cal R}_1}^c$, where $(\varepsilon_k)_{k \geq 1}$ is a sequence of positive numbers converging to zero. Take an eigenfunction pair $(\varphi_k,\psi_k)$ corresponding to $(\lambda_k,\mu_k)$ such that $\| \varphi_k \|_{1,\eta} + \| \psi_k \|_{1,\eta} = 1$, that is, a strong solution of the system
\[
	\begin{cases}
		-\,\mathcal{L}_1 \varphi_k=\lambda_k\,a(x)\,\vert \varphi_k\vert^{\alpha_1-1}\varphi_k\vert \psi_k\vert^{\beta_1-1}\psi_k & \text{in }\Omega,\\[2pt]
		-\,\mathcal{L}_2 \psi_k=\mu_k\,b(x)\,\vert \psi_k\vert^{\alpha_2-1}\psi_k\vert \varphi_k\vert^{\beta_2-1}\varphi_k & \text{in }\Omega,\\[2pt]
		\varphi_k=\psi_k=0 & \text{on }\partial\Omega.
	\end{cases}
	\]
Therefore, by the compact embedding $C^{1,\eta}(\overline\Omega)\hookrightarrow C^1(\overline\Omega)$, up to a subsequence, we get
\begin{equation}\label{conv}
\varphi_k \rightarrow \varphi\ \text{and}\ \psi_k \rightarrow\psi\ \text{in}\ C^1_0(\overline{\Omega})\ \text{as}\ k\rightarrow\infty.
\end{equation}
Then, $(\varphi, \psi)\in (W^{2,p}(\Omega))^2$ is a strong solution of the system
\[
\left\{
\begin{array}{llll}
-\,\mathcal{L}_1 \varphi = \lambda_1 a(x) \vert\varphi\vert^{\alpha_1-1}\varphi \vert \psi\vert^{\beta_1-1}\psi & {\rm in} \ \ \Omega,\\
-\,\mathcal{L}_2 \psi = \mu_1 b(x) \vert\psi\vert^{\alpha_2-1}\psi\vert\varphi\vert^{\beta_2-1}\varphi & {\rm in} \ \ \Omega,\\
\varphi =\psi =0 & {\rm on} \ \ \partial\Omega.
\end{array}
\right.
\]
By Proposition \ref{psim2}, we have $(\varphi,\psi)\in V^\circ_+\times V_+^\circ$, and so, from the convergence in (\ref{conv}), we get $(\varphi_k,\psi_k)\in\ V^\circ_+\times V_+^\circ$ for $k$ sufficiently large. So, by property (\ref{prop:uniqueness}), we derive $(\lambda_k,\mu_k) \in {\cal C}_1$ for $k$ large enough, contradicting that $(\lambda_k,\mu_k) \in \overline{{\cal R}_1}^c$ for all $k\in \mathbb{N}$.
\end{proof}

\begin{proposition} \label{propos2.15}
Let $a$, $b$, $\tilde{a}$ and $\tilde{b}$ be functions in $L^p(\Omega)$ such that $0 < a \leq \tilde{a}$ and $0 < b \leq \tilde{b}$ in $\Omega$. Then, $\Lambda_0(a,b) \geq \Lambda_0( \tilde{a}, \tilde{b})$.
\end{proposition}

\begin{proof} Suppose by contradiction that $\Lambda_0(a,b) < \Lambda_0( \tilde{a}, \tilde{b})$. Let $(\varphi, \psi)$ and $(\tilde{\varphi}, \tilde{\psi})$ be positive eigenfunctions corresponding to the eigenvalues $(\lambda_1(a, b),\mu_1(a, b))$ and $(\lambda_1( \tilde{a}, \tilde{b}), \mu_1( \tilde{a}, \tilde{b}))$, respectively, with $\frac{\mu_1(a,b)}{\lambda_1(a,b)}=\frac{\mu_1( \tilde{a}, \tilde{b})}{\lambda_1( \tilde{a}, \tilde{b})}$.

Consider $\Gamma =\{\gamma>0 : \tilde{\varphi} > \gamma \varphi\ {\rm and}\ \tilde{\psi} > \gamma^{\frac{1-\alpha_1}{\beta_1}} \psi\ {\rm in}\ \Omega\}$ and set $\gamma^{\ast}=\sup \Gamma$. Since $\lambda_1( \tilde{a}, \tilde{b}) > \lambda_1(a, b)$ and $\mu_1( \tilde{a}, \tilde{b}) > \mu_1(a, b)$, we obtain
\begin{eqnarray*}
-\mathcal{L}_1 (\tilde{\varphi}  - \gamma^{\ast} \varphi) &=& \lambda_1( \tilde{a}, \tilde{b}) \tilde{a}(x) \tilde{\varphi}^{\alpha_1}\tilde{\psi}^{\beta_1} - \gamma^{\ast} \lambda_1(a, b) a(x) \varphi^{\alpha_1}\psi^{\beta_1} \\
&\geq& (\lambda_1( \tilde{a}, \tilde{b}) - \lambda_1(a, b)) \gamma^{\ast} a(x) \varphi^{\alpha_1}\psi^{\beta_1} > 0, \\
-\mathcal{L}_2 (\tilde{\psi}  - {\gamma^{\ast}}^{\frac{1-\alpha_1}{\beta_1}}\psi) &=& \mu_1( \tilde{a}, \tilde{b}) \tilde{b}(x)\tilde{\psi}^{\alpha_2} \tilde{\varphi}^{\beta_2} - {\gamma^{\ast}}^{\frac{1-\alpha_1}{\beta_1}}\mu_1(a,b) b(x)\psi^{\alpha_2} \varphi^{\beta_2} \\
&\geq& (\mu_1( \tilde{a}, \tilde{b}) - \mu_1(a, b))  {\gamma^{\ast}}^{\frac{1-\alpha_1}{\beta_1}} b(x) \psi^{\alpha_2}\varphi^{\beta_2} > 0,
\end{eqnarray*}
in $\Omega$ a.e. and $\tilde{\varphi}  - \gamma^{\ast} \varphi=0=\tilde{\psi}  - {\gamma^{\ast}}^{\frac{1-\alpha_1}{\beta_1}}\psi$ in $\partial\Omega$. Then, by the Scalar Strong Maximum Principle, 
$\tilde{\varphi} > (\gamma^{\ast} + \varepsilon)\varphi$
and
$\tilde{\psi} > (\gamma^{\ast} + \varepsilon)^{\frac{1-\alpha_1}{\beta_1}}\psi$
in $\Omega$ for sufficiently small $\varepsilon>0$, yielding a contradiction.
\end{proof}
	
	\section{Proof of Theorem \ref{MP}: {\bf (WMP)}}

\begin{proof}[Proof of Theorem \ref{MP}]
	We first prove item (i), namely, the case $0<\alpha_1,\alpha_2<1$. Let $(\lambda, \mu) \in -\mathcal{C}_1$ and $(\tilde{\varphi},\tilde{\psi})$ be a positive eigenfunction corresponding to $(-\lambda, -\mu)$. Thus, $(-\tilde{\varphi},-\tilde{\psi})$ is a negative eigenfunction corresponding to $(\lambda, \mu)$. Therefore, {\bf (WMP)} fails in $\Omega$.

	Suppose now that $(\lambda, \mu) \in \mathbb{R}^2$ is a fixed pair outside $-\overline{\mathcal{R}_1}$. If $(\lambda, \mu) \in (\mathbb{R}_-^*)^2$, we obtain $-\lambda > \lambda_1$ and $-\mu > \mu_1$, where $(\lambda_1, \mu_1)$ is a principal eigenvalue of \eqref{1.3} with $\frac{\mu}{\lambda}=\frac{\mu_1}{\lambda_1}$. We denote by $(\varphi,\psi)$ a positive eigenfunction associated to $(\lambda_1, \mu_1)$. Then, $(-\varphi,-\psi)$ satisfies
	\[
	\left\{
	\begin{array}{llll}
		-\mathcal{L}_1 (-\varphi) - \lambda a(x)\vert-\varphi\vert^{\alpha_1-1}(-\varphi)\vert-\psi\vert^{\beta_1-1}(-\psi) &=& - \lambda_1 a(x)\varphi^{\alpha_1} \psi^{\beta_1} - \lambda a(x)\varphi^{\alpha_1} \psi^{\beta_1}\\
		&=& (-\lambda - \lambda_1)a(x)\varphi^{\alpha_1} \psi^{\beta_1} \geq (\not\equiv)\ 0  \ \ & {\rm in}\ \Omega,\\
		-\mathcal{L}_2 (-\psi) - \mu b(x)\vert-\psi\vert^{\alpha_2-1}(-\psi)\vert -\varphi\vert^{\beta_2-1}(-\varphi) &=& -\mu_1 b(x)\psi^{\alpha_2} \varphi^{\beta_2} - \mu b(x)\psi^{\alpha_2} \varphi^{\beta_2} \\
		&=& (-\mu - \mu_1)b(x)\psi^{\alpha_2} \varphi^{\beta_2} \geq (\not\equiv)\ 0 \ \ & {\rm in}\ \Omega
	\end{array}
	\right.
	\]
	and $-\varphi = 0 =-\psi$ on $\partial\Omega$. Since $-\varphi,-\psi<0$ in $\Omega$, {\bf (WMP)} fails in $\Omega$ for $(\lambda, \mu)$ in the third quadrant below $-\mathcal{C}_1$.
	
	Now, assume that $\lambda > 0$. Thus, there is $(\lambda_1, \mu_1)\in \mathcal{C}_1$ with $\lambda_1 > 0$ small enough (and so $\mu_1>0$ large enough) so that $-\lambda < - \lambda_1$ and $\mu > - \mu_1$. Then, $(-\varphi,\psi)$ satisfies
	\[
	\left\{
	\begin{array}{llll}
		-\mathcal{L}_1 (-\varphi) - \lambda a(x)\vert-\varphi\vert^{\alpha_1-1}(-\varphi) \psi^{\beta_1} &=& - \lambda_1 a(x)\varphi^{\alpha_1} \psi^{\beta_1} + \lambda a(x)\varphi^{\alpha_1} \psi^{\beta_1}\\
		&=& (\lambda - \lambda_1)a(x)\varphi^{\alpha_1} \psi^{\beta_1} \geq (\not\equiv)\ 0  \ \ & {\rm in}\ \Omega,\\
		-\mathcal{L}_2 (\psi) - \mu b(x)\psi^{\alpha_2}\vert -\varphi\vert^{\beta_2-1}(-\varphi) &=& \mu_1 b(x)\psi^{\alpha_2} \varphi^{\beta_2} + \mu b(x)\psi^{\alpha_2} \varphi^{\beta_2} \\
		&=& (\mu + \mu_1)b(x)\psi^{\alpha_2} \varphi^{\beta_2} \geq (\not\equiv)\ 0 \ \ & {\rm in}\ \Omega
	\end{array}
	\right.
	\]
	and $-\varphi = 0 =\psi$ on $\partial\Omega$. But $-\varphi<0$ in $\Omega$ and so {\bf (WMP)} fails in $\Omega$.
	
	In the case $\mu > 0$, there is $(\lambda_1, \mu_1)\in \mathcal{C}_1$ with $\lambda_1 > 0$ large enough (and so $\mu_1>0$ small enough) so that $\lambda > -\lambda_1$ and $-\mu < -\mu_1$. Therefore, $(\varphi,-\psi)$ satisfies
	\[
	\left\{
	\begin{array}{llll}
		-\mathcal{L}_1 (\varphi) - \lambda a(x)\varphi^{\alpha_1} \vert -\psi\vert^{\beta_1-1}(-\psi) &=&  \lambda_1 a(x)\varphi^{\alpha_1} \psi^{\beta_1} + \lambda a(x)\varphi^{\alpha_1} \psi^{\beta_1}\\
		&=& (\lambda + \lambda_1)a(x)\varphi^{\alpha_1} \psi^{\beta_1} \geq (\not\equiv)\ 0  \ \ & {\rm in}\ \Omega,\\
		-\mathcal{L}_2 (-\psi) - \mu b(x)\vert-\psi\vert^{\alpha_2-1}(-\psi)\varphi^{\beta_2} &=& -\mu_1 b(x)\psi^{\alpha_2} \varphi^{\beta_2} + \mu b(x)\psi^{\alpha_2} \varphi^{\beta_2} \\
		&=& (\mu - \mu_1)b(x)\psi^{\alpha_2} \varphi^{\beta_2} \geq (\not\equiv)\ 0 \ \ & {\rm in}\ \Omega
	\end{array}
	\right.
	\]
	and $\varphi = 0 =-\psi$ on $\partial\Omega$. Note that, $-\psi<0$ in $\Omega$ and again {\bf (WMP)} fails in $\Omega$.
	
	Finally, we show the sufficiency of the condition
	\(
	(\lambda,\mu)\in -(\overline{{\cal R}_1}\setminus{\cal C}_1).
	\)
	
	Let $(u,v)$ be a supersolution of problem~\eqref{1.3}.
	
	It is clear that {\bf (WMP)} holds in $\Omega$ when $\lambda=\mu=0$, or whenever $u$ or $v$ is trivial in $\Omega$. Therefore, we may assume that both $u$ and $v$ are nontrivial in $\Omega$.
	
	Let us first consider the case $(\lambda,\mu)\in-{\cal R}_1$. Proceeding by contradiction, assume that $u$ or $v$ is negative somewhere in $\Omega$.
	
	If $uv\leq0$ throughout $\Omega$, then
	\[
	-\mathcal L_1u\ge0,
	\qquad
	-\mathcal L_2v\ge0
	\]
	in $\Omega$ a.e. and $u,v\ge0$ on $\partial\Omega$. Hence, by the Scalar Strong Maximum Principle, we derive the contradiction
	\[
	u,v>0
	\quad
	\text{in }\Omega.
	\]
	
	Therefore, there exists $y_0\in\Omega$ such that
	\(
	u(y_0)
	v(y_0)>0.
	\)
	Set
	\(
	M_1:=\{x\in\Omega:\ u(x)>0 \mbox{ and } v(x)>0\},
	\)
	and
	\(
	M_2:=\{x\in\Omega:\ u(x)\le0 \mbox{ or } v(x)\le0\}.
	\)
	Then
	\[
	\Omega=M_1\cup M_2.
	\]
	
	Let $(\varphi,\psi)$ be a positive eigenfunction corresponding to $(\lambda_1,\mu_1)$, with
	\(
	\mu/\lambda
	=\mu_1/\lambda_1.
	\)
	Consider the set
	\(
	\Gamma
	:=
	\left\{
	\gamma>0:
	\varphi>-\gamma u
	\text{ and }
	\psi>-\gamma^rv
	\text{ in }\Omega
	\right\},
	\)
	where \( r=(1-\alpha_1)/\beta_1=\beta_2/(1-\alpha_2). \)
	By Hopf's lemma, $\Gamma$ is nonempty and clearly bounded above. Set
	\[
	\gamma^\ast
	=
	\sup\Gamma.
	\]
	
	We first claim that
	\(
	\varphi>-(\gamma^\ast+\varepsilon)u, \mbox{ }
	\psi>-(\gamma^\ast+\varepsilon)^rv
	\)
	in $\Omega$ for $\varepsilon\sim0$.
	
	Suppose first that $M_1=\emptyset$. Then, since
	\(
	-\lambda<\lambda_1,
	-\mu<\mu_1,
	\)
	we obtain
	\begin{equation}\label{p01}
		\left\{
		\begin{array}{lll}
			-{\cal L}_1 (\varphi + \gamma^{\ast} u)
			\geq
			(\lambda_1+\lambda)a(x)\varphi^{\alpha_1}\psi^{\beta_1}
			>0,
			\\[3pt]
			-{\cal L}_2 (\psi + (\gamma^\ast)^r v)
			\geq
			(\mu_1+\mu)b(x)\psi^{\alpha_2}\varphi^{\beta_2}
			>0,
		\end{array}
		\right.
	\end{equation}
	in $\Omega$ a.e. Moreover, \(\varphi+\gamma^\ast u\ge0\) and \(\psi+(\gamma^\ast)^rv\ge0\) on \(\partial\Omega\). Applying the Scalar Strong Maximum Principle to each equation, we obtain
	\[
	\varphi+\gamma^\ast u>0,
	\qquad
	\psi+(\gamma^\ast)^rv>0
	\]
	in $\Omega$. Hence, \(\varphi>-(\gamma^\ast+\varepsilon)u\) and \(\psi>-(\gamma^\ast+\varepsilon)^rv\) in \(\Omega\) for \(\varepsilon\sim0\), contradicting the definition of \(\gamma^\ast\). Therefore,
	\[
	u,v\ge0
	\quad
	\text{in }\Omega.
	\]
	
	We now assume that $M_1\neq\emptyset$. Since \(\varphi>-(\gamma^\ast+\varepsilon)u\) and \(\psi>-(\gamma^\ast+\varepsilon)^rv\) hold trivially in $M_1$ for every $\varepsilon>0$, it suffices to prove these inequalities in $M_2$.
	
	Suppose, by contradiction, that there exists \(
	y\in M_2
	\)
	such that \(\varphi(y)=-\gamma^\ast u(y)\) or \(\psi(y)=-(\gamma^\ast)^rv(y)\).
	Then
	\(
	u(y)<0
	\mbox{ or }
	v(y)<0.
	\)
	Hence
	\(
	y\notin\overline{M_1}\cup\partial M_2,
	\)
	and so there exists an open ball
	\(
	B\subset M_2
	\)
	containing $y$. Since inequalities~\eqref{p01} hold in $M_2$, we have
	\[
	-{\cal L}_1 (\varphi+\gamma^\ast u)>0,
	\qquad
	-{\cal L}_2 (\psi+(\gamma^\ast)^rv)>0
	\]
	in $B$ a.e. Moreover,
	\[
	\varphi+\gamma^\ast u\ge0,
	\qquad
	\psi+(\gamma^\ast)^rv\ge0
	\]
	in $\partial B$. Applying the Scalar Strong Maximum Principle, we obtain
	\[
	\varphi+\gamma^\ast u>0,
	\qquad
	\psi+(\gamma^\ast)^rv>0
	\]
	in $B$, contradicting the choice of $y$. Thus,
	\[
	\varphi>-\gamma^\ast u,
	\qquad
	\psi>-(\gamma^\ast)^rv
	\]
	in $M_2$.
	
	If \(
	\partial\Omega\cap\partial M_2
	=
	\emptyset,
	\)
	then the conclusion follows.
	
	Now suppose that
	\(
	\partial\Omega\cap\partial M_2
	\neq
	\emptyset.
	\)
	Let
	\(
	x_0
	\in
	\partial\Omega\cap\partial M_2.
	\)
	If there exists $\varepsilon_0>0$ such that
	\(
	u,v\ge0
	\)
	in
	\(
	B(x_0,\varepsilon_0)\cap\Omega,
	\)
	then
	\[
	\varphi>-(\gamma^\ast+\varepsilon)u,
	\qquad
	\psi>-(\gamma^\ast+\varepsilon)^rv
	\]
	hold there for all $\varepsilon>0$.
	
	Assume therefore that, for every $\varepsilon_0>0$, there exists	\(
	x\in B(x_0,\varepsilon_0)\cap\Omega
	\)
	such that
	\(
	u(x)<0
	\mbox{ or }
	v(x)<0.
	\)
	
	If
	\(
	u(x_0)=0
	\mbox{ and }
	\partial_\nu u(x_0)\le0,
	\)
	then, since $\varphi\in V_+^\circ$,
	\(
	\partial_\nu\varphi(x_0)<0,
	\)
	and therefore
	\[
	\partial_\nu(\varphi+\gamma^\ast u)(x_0)
	<
	0.
	\]
	The analogous conclusion holds for $v$ and $\psi$.
	
	Finally, if
	\(
	u(x_0)=0
	\mbox{ and }
	\partial_\nu u(x_0)>0,
	\)
	or
	\(
	v(x_0)=0
	\mbox{ and }
	\partial_\nu v(x_0)>0,
	\)
	then
	\(
	x_0
	\notin
	\partial M_1.
	\)
	Hence there exists an open set $\Omega'$, with boundary of class $C^{1,1}$, such that
	\[
	\Omega'
	\subset
	M_2,
	\]
	and a neighborhood $W_0$ of $x_0$ in $\partial\Omega$ satisfying
	\(
	W_0
	\subset
	\partial\Omega'.
	\)
	Moreover, inequalities~\eqref{p01} hold in $\Omega'$. By Hopf's lemma,
	\[
	\partial_\nu(\varphi+\gamma^\ast u)(x_0)<0
	\]
	or
	\[
	\partial_\nu(\psi+(\gamma^\ast)^rv)(x_0)<0.
	\]
	
	Consequently,
	\[
	\varphi>-(\gamma^\ast+\varepsilon)u,
	\qquad
	\psi>-(\gamma^\ast+\varepsilon)^rv
	\]
	in $\Omega$ for $\varepsilon\sim0$, contradicting again the definition of $\gamma^\ast$. Therefore,
	\[
	u,v\ge0
	\quad
	\text{in }\Omega.
	\]
	
Now, if $\lambda=0$ and $\mu<0$, then $u\ge0$ in $\Omega$. Since $v$ is nontrivial, if $uv\le0$ throughout $\Omega$, then applying the Scalar Strong Maximum Principle to the second equation, we derive the contradiction $v>0$ in $\Omega$. Therefore, proceeding by contradiction, assume that $v$ changes sign in $\Omega$. Arguing as in the previous case, we again obtain a contradiction. Hence,
\(
u,v\ge0
\)
in $\Omega$.

The case $\mu=0$ and $\lambda<0$ can be treated analogously. This completes the proof of item~(i).

Items~(ii) and~(iii) follow similarly to item~(i), with the necessary modifications.
\end{proof}

\section{Lower Estimate and Proof of Theorem \ref{sm}}

As a key ingredient in the proof of the theorem and of independent interest, we present the following lower bound for principal eigenvalues corresponding to the problem \eqref{1.3}:

\begin{theorem} \label{lower} Let $a,b\in L^p(\Omega)$ and 
\[
\mathcal{C}_1=\left\{(\lambda,\mu)\in (\R_+^*)^2:\lambda^{\frac{1}{\sqrt{\beta_1(1-\alpha_1)}}}\mu^{\frac{1}{\sqrt{\beta_2(1-\alpha_2)}}}=\Lambda_0 \right\},
\] be the principal curve associated to \eqref{1.3}. Assume that $T B |\Omega|^\frac{1}{n} < 1$ with $B=\max\{B_1,B_2\}$. Then

\begin{eqnarray}
&& \Lambda_0 \geq \frac{(1-B_1T\vert\Omega\vert^{\frac{1}{n}})^{\frac{1}{\theta}}}{B_1^{\frac{1}{\theta}}\|a\|^{\frac{1}{\theta}}_{L^n(\Omega)}}\frac{(1-B_2T\vert\Omega\vert^{\frac{1}{n}})^{\frac{1}{\zeta}}}{B_2^{\frac{1}{\zeta}}\|b\|^{\frac{1}{\zeta}}_{L^n(\Omega)}},  \label{lb1} 
\end{eqnarray}
where $\theta:=\sqrt{\beta_1(1-\alpha_1)}$, $\zeta:=\sqrt{\beta_2(1-\alpha_2)}$.
\end{theorem}

\begin{proof}

Let $(\lambda_1, \mu_1)\in \mathcal{C}_1$ and $(\varphi,\psi) \in V_+^\circ\times V_+^\circ$ be a positive eigenfunction of the system \eqref{1.3} associated to $(\lambda_1, \mu_1)$. Denote $\tilde{{\cal L}}_l = {\cal L}_l - c_l^+(x)$ for $l=1,2$. Thus, $(\varphi,\psi)$ satisfies

\[
\left\{
\begin{array}{llll}
-\tilde{{\cal L}}_1 \varphi = \lambda_1 a(x) \varphi^{\alpha_1}\psi^{\beta_1} + c_1^+(x) \varphi & {\rm in} \ \ \Omega,\\
-\tilde{{\cal L}}_2 \psi = \mu_1 b(x) \psi^{\alpha_2} \varphi^{\beta_2} + c_2^+(x) \psi & {\rm in} \ \ \Omega,\\
\varphi= \psi=0 & {\rm on} \ \ \partial\Omega.
\end{array}
\right.
\]
Using the ABP estimate for elliptic operators (see \cite[Theorem~9.1]{MR737190}) to the first above equation, we get
\begin{eqnarray*}
||\varphi||_{L^\infty(\Omega)} &=& \sup_\Omega \varphi \leq B_1 ||\lambda_1 a(x)\varphi^{\alpha_1} \psi^{\beta_1} + c_1^+(x) \varphi||_{L^n(\Omega)}\\
&\leq& \lambda_1 B_1 ||a||_{L^n(\Omega)}||\varphi||^{\alpha_1}_{L^\infty(\Omega)} ||\psi||^{\beta_1}_{L^\infty(\Omega)} + B_1 ||c_1^+||_{L^p(\Omega)} ||\varphi||_{L^\infty(\Omega)}\\
&\leq& \lambda_1 B_1 ||a||_{L^n(\Omega)} ||\varphi||^{\alpha_1}_{L^\infty(\Omega)} ||\psi||^{\beta_1}_{L^\infty(\Omega)}  + T B_1 |\Omega|^\frac{1}{n} ||\varphi||_{L^\infty(\Omega)}.
\end{eqnarray*}
Assuming that $T B_1 |\Omega|^\frac{1}{n} < 1$, we then get

\begin{equation}\label{i1}
||\varphi||_{L^\infty(\Omega)} \leq \frac{\lambda_1 B_1 ||a||_{L^p(\Omega)}}{ 1 - T B_1 |\Omega|^\frac{1}{n}} ||\varphi||^{\alpha_1}_{L^\infty(\Omega)} ||\psi||^{\beta_1}_{L^\infty(\Omega)} .
\end{equation}
Proceeding in a similar way with the second equation of the above system, one also has

\begin{equation}\label{i2}
||\psi||_{L^\infty(\Omega)} \leq \frac{\mu_1 B_2 ||b||_{L^n(\Omega)}}{ 1 - T B_2 |\Omega|^\frac{1}{n}} ||\varphi||^{\beta_2}_{L^\infty(\Omega)} ||\psi||^{\alpha_2}_{L^\infty(\Omega)} 
\end{equation}
assuming that $T B_2 |\Omega|^\frac{1}{n} < 1$. Thus, combining inequalities \eqref{i1} and \eqref{i2} and applying the hypothesis 
\[
\beta_1 \beta_2 = (1-\alpha_1)(1-\alpha_2)\ \text{ and }\ \lambda_1^{\frac{1}{\theta}}\mu_1^{\frac{1}{\zeta}} = \Lambda_0,
\]
we obtain \eqref{lb1}. The proof is complete. 
\end{proof}	

\begin{proof}[ Proof of Theorem \ref{sm}] We first prove item (i), namely, the case $0<\alpha_1,\alpha_2<1$. By Theorem \ref{MP} (i), {\bf (WMP)} implies $(\lambda, \mu) \in-( \overline{\mathcal{R}_1} \setminus \mathcal{C}_1)$, that is, $\lambda,\mu\leq 0$. 

Conversely, we assume that $\lambda \leq 0$ and $\mu \leq 0$. By Theorem 2.6 of \cite{MR1258192}, there exists a constant $\sigma_1 > 0$ depending only on $n$, $c_0$, $C_0$ and $T$ such that ${\cal L}_1$ and ${\cal L}_2$ satisfy Scalar Strong Maximum Principle in $\Omega$ provided that $|\Omega| < \sigma_1$. If $\lambda = 0$ or $\mu = 0$, then by Theorem \ref{MP}, the desired {\bf (WMP)} follows.

Finally, suppose that $\lambda < 0$ and $\mu < 0$. In this case, Theorem \ref{lower} plays an important role in the proof.

Consider the constants $\sigma_2 > 0$ and $\sigma_3 > 0$ given by

\[
\sigma_2 = \frac{1}{2^n T^n B_1^n}\ \  \sigma_3 = \frac{1}{2^n T^n B_2^n}
\]
and set $\sigma = \min\{\sigma_1,\sigma_2,\sigma_3\}$. In addition, we define

\[
\kappa_1 = \frac{1}{-2 \lambda B_1},\ \ \kappa_2 = \frac{1}{-2 \mu B_2}.
\]
Then, applying the estimate \eqref{lb1} of Theorem \ref{lower}, we get
\begin{eqnarray*}
&& \Lambda_0 > \frac{(1-B_1T\eta_2^{\frac{1}{n}})^{\frac{1}{\theta}}}{(B_1\kappa_1)^{\frac{1}{\theta}}}\frac{(1-B_2T\eta_3^{\frac{1}{n}})^{\frac{1}{\zeta}}}{(B_2\kappa_2)^{\frac{1}{\zeta}}}= \frac{\frac{1}{2^{\frac{1}{\theta}}}}{\frac{1}{(-2 \lambda)^{\frac{1}{\theta}}}}\frac{\frac{1}{2^{\frac{1}{\zeta}}}}{\frac{1}{(-2 \mu)^{\frac{1}{\zeta}}}} = (-\lambda)^{\frac{1}{\theta}}(-\mu)^{\frac{1}{\zeta}}
\end{eqnarray*}
whenever $|\Omega| < \sigma$, $\|a\|_{L^n(\Omega)}<\kappa_1$ and $\|b\|_{L^n(\Omega)} < \kappa_2$. Consequently, we get $(\lambda, \mu) \in -(\overline{{\cal R}_1} \setminus \mathcal{C}_1)$ for such domains and hence, by Theorem \ref{MP}, the desired {\bf (WMP)} follows.

Items (ii) and (iii) follow similarly to item (i), with standard modifications.
\end{proof}

\section{The Nonhomogeneous Problem: A Priori Bounds and Existence}

\begin{proof}[ Proof of Theorem \ref{thm:apriori-general}]
	The argument proceeds by contradiction, through a normalization procedure that balances the two equations of the system simultaneously. Suppose that \eqref{eq:apriori-final-en} fails. Then there exist sequences
	\[
	(f_k,g_k)\in (L^{p}(\Omega))^2,\qquad (\varphi_k,\psi_k)\in (W^{2,p}(\Omega))^2,
	\]
	and corresponding strong solutions $(u_k,v_k)\in (W^{2,p}(\Omega))^2$ of \eqref{eq:system-general-en}, with data $(f_k,g_k,\varphi_k,\psi_k)$, such that
	\begin{equation}\label{eq:contradiction-ineq-nosteps}
		\|u_k\|_{L^\infty(\Omega)}+\|v_k\|_{L^\infty(\Omega)}^{\,s}
		>
		k\Big(
		\|\varphi_k\|_{W^{2,p}(\Omega)}
		+\|\psi_k\|_{W^{2,p}(\Omega)}^{\,s}
		+\|f_k\|_{L^p(\Omega)}
		+\|g_k\|_{L^p(\Omega)}^{\,s}
		\Big).
	\end{equation}
	Define the scaling factor
	\[
		S_k:=\|u_k\|_{L^\infty(\Omega)}+\|v_k\|_{L^\infty(\Omega)}^{\,s}\,>\,0,
	\]
	and set the normalized unknowns and data as
	\begin{equation}\label{eq:tilde-unknowns-nosteps}
		\tilde u_k:=\frac{u_k}{S_k},\qquad
		\tilde v_k:=\frac{v_k}{S_k^{\,r}},\qquad
		\tilde\varphi_k:=\frac{\varphi_k}{S_k},\quad
		\tilde\psi_k:=\frac{\psi_k}{S_k^{\,r}},\quad
		\tilde f_k:=\frac{f_k}{S_k},\quad
		\tilde g_k:=\frac{g_k}{S_k^{\,r}},
	\end{equation}
	where $r=\dfrac{1}{s}$. By construction,
	\begin{equation}\label{eq:Linf-normalization-nosteps}
		\|\tilde u_k\|_{L^\infty(\Omega)}+\|\tilde v_k\|_{L^\infty(\Omega)}^{\,s}=1,
	\end{equation}
	while from \eqref{eq:contradiction-ineq-nosteps} one obtains
	\begin{equation}\label{eq:data-go-to-0-nosteps}
		\|\tilde\varphi_k\|_{W^{2,p}(\Omega)}
		+\|\tilde\psi_k\|_{W^{2,p}(\Omega)}^{\,s}
		+\|\tilde f_k\|_{L^p(\Omega)}
		+\|\tilde g_k\|_{L^p(\Omega)}^{\,s}
		\;\le\;\frac1k\;\xrightarrow[k\to\infty]{}\;0.
	\end{equation}
	
	Dividing the first equation of \eqref{eq:system-general-en} by $S_k$ and using \eqref{eq:tilde-unknowns-nosteps}, one finds
	\begin{eqnarray*}
		-\mathcal L_1\tilde u_k
		&=&\lambda\,a(x)\,
		\frac{|v_k|^{\beta_1-1}v_k\,|u_k|^{\alpha_1-1}u_k}{S_k}
		+\tilde f_k(x)\nonumber\\
		&=&\lambda\,a(x)\,S_k^{\,E_1}\,
		|\tilde v_k|^{\beta_1-1}\tilde v_k\,|\tilde u_k|^{\alpha_1-1}\tilde u_k
		+\tilde f_k(x),
		\label{eq:tilde-eq-1-nosteps}
	\end{eqnarray*}
	where
	\[
	E_1:=\beta_1 r+\alpha_1-1.
	\]
	Similarly, dividing the second equation by $S_k^{\,r}$ yields
	\begin{eqnarray*}
		-\mathcal L_2\tilde v_k
		&=&\mu\,b(x)\,
		\frac{|u_k|^{\beta_2-1}u_k\,|v_k|^{\alpha_2-1}v_k}{S_k^{\,r}}
		+\tilde g_k(x)\nonumber\\
		&=&\mu\,b(x)\,S_k^{\,E_2}\,
		|\tilde u_k|^{\beta_2-1}\tilde u_k\,|\tilde v_k|^{\alpha_2-1}\tilde v_k
		+\tilde g_k(x),
		\label{eq:tilde-eq-2-nosteps}
	\end{eqnarray*}
	with
	\[
	E_2:=\beta_2+\alpha_2 r-r.
	\]
	By the definition of $r$ in \eqref{eq:def-r-en}, both exponents vanish, namely
	\[
	E_1=\beta_1\Big(\tfrac{1-\alpha_1}{\beta_1}\Big)+\alpha_1-1=0,\qquad
	E_2=\beta_2+(\alpha_2-1)\frac{\beta_2}{1-\alpha_2}=0.
	\]
	Consequently, the normalized pair satisfies
	\begin{equation}\label{eq:tilde-system-final-nosteps}
		\begin{cases}
			-\mathcal L_1\tilde u_k=\lambda\,a(x)\,|\tilde v_k|^{\beta_1-1}\tilde v_k\,|\tilde u_k|^{\alpha_1-1}\tilde u_k+\tilde f_k(x)& \text{in } \Omega, \\[2pt]
			-\mathcal L_2\tilde v_k=\mu\,b(x)\,|\tilde u_k|^{\beta_2-1}\tilde u_k\,|\tilde v_k|^{\alpha_2-1}\tilde v_k+\tilde g_k(x)& \text{in } \Omega, \\[2pt]
			\tilde u_k=\tilde\varphi_k,\quad \tilde v_k=\tilde\psi_k & \text{on }\partial\Omega,
		\end{cases}
	\end{equation}
	with the normalization \eqref{eq:Linf-normalization-nosteps} and vanishing data \eqref{eq:data-go-to-0-nosteps}.
	
	Since $a,b\in L^p(\Omega)$ and $0\le |\tilde u_k|,|\tilde v_k|\le 1$, the nonlinear right-hand sides of \eqref{eq:tilde-system-final-nosteps} belong to $L^p(\Omega)$ with norms uniformly bounded by $C(\|a\|_{L^p(\Omega)}+\|b\|_{L^p(\Omega)})$. Therefore, applying standard linear elliptic estimates, one obtains
	\[
		\|\tilde u_k\|_{W^{2,p}(\Omega)}+\|\tilde v_k\|_{W^{2,p}(\Omega)}
		\;\le\;C\Big(1+\|\tilde f_k\|_{L^p(\Omega)}+\|\tilde g_k\|_{L^p(\Omega)}+\|\tilde\varphi_k\|_{W^{2,p}(\Omega)}+\|\tilde\psi_k\|_{W^{2,p}(\Omega)}\Big),
	\]
	for some constant $C$ independent of $k$. In view of \eqref{eq:data-go-to-0-nosteps}, the sequence $\{(\tilde u_k,\tilde v_k)\}$ is thus bounded in $(W^{2,p}(\Omega))^2$. By compact embedding, there exists $\eta\in(0,1)$ such that, up to a subsequence,
	\begin{equation}\label{eq:compactness-nosteps}
		\tilde u_k\to \tilde u,\qquad \tilde v_k\to \tilde v
		\quad\text{in } W^{2,p}(\Omega)\ \text{ and in }C^{1,\eta}(\overline\Omega).
	\end{equation}
	In particular, passing to the limit in \eqref{eq:Linf-normalization-nosteps}, one obtains
	\begin{equation}\label{eq:nontrivial-limit-nosteps}
		\|\tilde u\|_{L^\infty(\Omega)}+\|\tilde v\|_{L^\infty(\Omega)}^{\,s}=1.
	\end{equation}
	
	Passing to the limit in \eqref{eq:tilde-system-final-nosteps}, using \eqref{eq:compactness-nosteps} and \eqref{eq:data-go-to-0-nosteps}, one concludes that $(\tilde u,\tilde v)$ is a strong solution of the homogeneous coupled system
	\[
	\begin{cases}
		-\mathcal L_1\tilde u=\lambda\,a(x)\,|\tilde v|^{\beta_1-1}\tilde v\,|\tilde u|^{\alpha_1-1}\tilde u& \text{in } \Omega, \\[2pt]
		-\mathcal L_2\tilde v=\mu\,b(x)\,|\tilde u|^{\beta_2-1}\tilde u\,|\tilde v|^{\alpha_2-1}\tilde v& \text{in } \Omega, \\[2pt]
		\tilde u=\tilde v=0&\text{on }\partial\Omega,
	\end{cases}
	\]
	which is nontrivial thanks to \eqref{eq:nontrivial-limit-nosteps}. This contradicts the assumption that $(\lambda,\mu)$ does not belong to the spectrum of the associated homogeneous problem. The contradiction proves that the a priori estimate \eqref{eq:apriori-final-en} must hold with a constant $A>0$ depending only on the structural data of the problem, that is, a constant independent of $(f,g), (\varphi,\psi)$ and $(u,v)$.
\end{proof}

\begin{remark}
	For every compact set $K\subset\mathbb{R}^2$ contained in the admissible region and disjoint from the eigencurve, there exists a constant $A=A(K)>0$ such that, for all $(\lambda,\mu)\in K$ and all admissible data, any strong solution $(u,v)$ of \eqref{eq:system-general-en} satisfies the uniform bound
	\begin{equation}\label{eq:apriori-Linf}
		\|u\|_{L^\infty(\Omega)}+\|v\|_{L^\infty(\Omega)}^{\,s}
		\;\le\;
		A\Big(
		\|\varphi\|_{W^{2,p}(\Omega)}
		+\|\psi\|_{W^{2,p}(\Omega)}^{\,s}
		+\|f\|_{L^{p}(\Omega)}
		+\|g\|_{L^{p}(\Omega)}^{\,s}
		\Big).
	\end{equation}
	In particular, the constant $A$ depends only on $K$, the domain $\Omega$, and the structural constants associated with the operators.
\end{remark}

\begin{proof}[Proof of Theorem~\ref{thm:existence-outside-curve}]
	The argument relies on the a priori estimate \eqref{eq:apriori-Linf} (proved in Theorem~\ref{thm:apriori-general} via the normalization procedure $S=\|u\|_{L^\infty(\Omega)}+\|v\|_{L^\infty(\Omega)}^s$) and a Leray-Schauder continuation principle based on the topological degree.

	Let $X:=C(\overline{\Omega})\times C(\overline{\Omega})$, endowed with the product norm
	\[
	\|(\xi,\zeta)\|_X:=\|\xi\|_{L^\infty(\Omega)}+\|\zeta\|_{L^\infty(\Omega)}.
	\]
	For $t\in[0,1]$ and $(\xi,\zeta)\in X$, define $H(t,\xi,\zeta)=(U,V)$, where $U,V\in W^{2,p}(\Omega)$ solve the linear boundary problems
	\[
	\begin{cases}
		-\mathcal L_1 U = t\lambda\,a(x)\,|\xi|^{\alpha_1-1}\xi\,|\zeta|^{\beta_1-1}\zeta + t f(x)& \text{in } \Omega, \\[2pt]
		U = t\varphi & \text{on }\partial\Omega,
	\end{cases}
	\]
	\[
	\begin{cases}
		-\mathcal L_2 V = t\mu\,b(x)\,|\zeta|^{\alpha_2-1}\zeta\,|\xi|^{\beta_2-1}\xi + t g(x)& \text{in } \Omega, \\[2pt]
		V = t\psi & \text{on }\partial\Omega.
	\end{cases}
	\]
	Standard elliptic estimates together with the compact embedding $W^{2,p}(\Omega)\hookrightarrow C(\overline{\Omega})$ ensure that $H:[0,1]\times X\to X$ is well defined, continuous, and compact.

	Assume that $H(t,\xi,\zeta)=(\xi,\zeta)$ for some $t\in[0,1]$. Then $(u,v)=(\xi,\zeta)$ solves the nonlinear system with parameters $(t\lambda,t\mu)$ and data $(t f,t g,t\varphi,t\psi)$. Since $(\lambda,\mu)$ is outside the eigencurve, the a priori estimate \eqref{eq:apriori-Linf} yields
	\[
	\|u\|_{L^\infty(\Omega)}+\|v\|_{L^\infty(\Omega)}^s
	\;\le\; A\Big(\|\varphi\|_{W^{2,p}(\Omega)}+\|\psi\|_{W^{2,p}(\Omega)}^s+\|f\|_{L^p(\Omega)}+\|g\|_{L^p(\Omega)}^s\Big) =: M_0,
	\]
	with $A$ independent of $t\in[0,1]$. In particular,
	\[
	\|(\xi,\zeta)\|_X \;\le\; M_0+M_0^{r}.
	\]
	Choose $R>M_0+M_0^{r}$ and let
	\[
	B_R := \{ (\xi,\zeta)\in X : \|(\xi,\zeta)\|_X < R \}.
	\]
	It follows that no fixed point of $H(t,\cdot)$ can lie on $\partial B_R$ for any $t\in[0,1]$.

	Define $F_t:=I-H(t,\cdot):\overline{B_R}\to X$. By the previous step, $0\notin F_t(\partial B_R)$ for all $t\in[0,1]$, so the Leray--Schauder degree $\deg(F_t,B_R,0)$ is well defined. The fundamental properties of the degree (existence, normalization, and homotopy invariance) guarantee that $\deg(F_t,B_R,0)$ remains constant along $t$. At $t=0$, one has $H(0,\cdot)\equiv (0,0)$, hence $F_0=I$ and $\deg(F_0,B_R,0)=1$. Therefore,
	\[
	\deg(F_1,B_R,0)=\deg(F_0,B_R,0)=1\neq 0,
	\]
	which implies the existence of $(u,v)\in B_R$ such that $F_1(u,v)=0$, i.e. $H(1,u,v)=(u,v)$. By construction, this pair $(u,v)\in (W^{2,p}(\Omega))^2$ is a strong solution of the nonhomogeneous problem \eqref{eq:system-general-en}. By Theorem \ref{MP}, we obtain the properties stated in items (i), (ii) and (ii) of theorem.

	Finally, inserting $(u,v)$ into the definition of $H$ and applying linear elliptic estimates yields
	\[
	\|u\|_{W^{2,p}(\Omega)}+\|v\|_{W^{2,p}(\Omega)}
	\;\le\; C\Big(1+\|\varphi\|_{W^{2,p}(\Omega)}+\|\psi\|_{W^{2,p}(\Omega)}+\|f\|_{L^p(\Omega)}+\|g\|_{L^p(\Omega)}\Big),
	\]
	with $C$ depending only on the domain, the structural constants of $\mathcal L_1,\mathcal L_2$, and the chosen compact set of parameters. This completes the proof.
\end{proof}

\section{Antimaximum Principle}


\begin{proof}[ Proof of Theorem \ref{AMP}] The proof will be divided into steps.
	
	\noindent\textbf{Step 1.}
	Assume $(\lambda,\mu)=(\lambda_1,\mu_1)\in\mathcal C_1$ and let $f,g\in L^p(\Omega)$ satisfy $f,g\ge0$ and $f+g\not\equiv0$. Then there exists no strong solution $(u,v)\in V\times V$ of (\ref{eq:system}).
	
	Suppose, by contradiction, that \eqref{eq:system} admits a strong solution $(u,v)$.  
	For $i=1,2$, denote by $T_i:L^p(\Omega)\to C^{1,\eta}(\overline\Omega)$ the resolvent operator associated with $-\mathcal L_i$, defined by
	\[
	T_i h := w \quad\text{such that}\quad 
	-\mathcal L_i w = h\text{ in }\Omega,\quad w=0\text{ on }\partial\Omega.
	\]
	The Scalar Strong Maximum Principle and Hopf’s lemma imply that $T_i$ is linear, compact, and strictly positive, namely
	\[
	h\ge0,\ h\not\equiv0 \ \Longrightarrow\ T_i h\in V_+^\circ.
	\]
	
	Define the nonlinear operators
	\[
	\tilde{J}_1(v):=T_1\!\big(a\,|v|^{\beta_1-1}v\big),\qquad 
	\tilde{J}_2(u):=T_2\!\big(b\,|u|^{\beta_2-1}u\big).
	\]
	Each of these maps $\tilde{J}_i:V\to V$ is continuous, compact, and order-preserving.

	Hence, system \eqref{eq:system} can be rewritten as
	\[
		u = \lambda_1 \tilde{J}_1(v) + T_1 f,\qquad 
		v = \mu_1 \tilde{J}_2(u) + T_2 g.
	\]
	Substituting the second identity into the first gives
	\begin{equation}\label{eq:operator-reduction}
		u = (\lambda_1\mu_1)\,(\tilde{J}_1\circ \tilde{J}_2)(u) + \tilde{J}_1(T_2 g) + T_1 f.
	\end{equation}
	Since $(\lambda_1,\mu_1)\in\mathcal C_1$, the product $\lambda_1\mu_1$ coincides with the inverse of the spectral radius of $\tilde{J}_1\circ \tilde{J}_2$, which corresponds to the principal eigenvalue of this compact and strongly positive operator.
	
	Because $f,g\ge0$ and $f+g\not\equiv0$, we have either $T_1 f\in V_+^\circ$ or $T_2 g\in V_+^\circ$. 
	Consequently, $\tilde{J}_1(T_2 g)\ge0$ and the function
	\[
	\phi = \tilde{J}_1(T_2 g)+T_1 f
	\]
	belongs to $V_+$ and satisfies $\phi\not\equiv0$. 
	Equation \eqref{eq:operator-reduction} can thus be written as
	\[
	u - (\tilde{J}_1\circ \tilde{J}_2)u = \phi\ge0,\qquad \phi\not\equiv0,
	\]
	which means that $u\ge (\tilde{J}_1\circ \tilde{J}_2)u$, with strict inequality somewhere in $\Omega$.
	
	By the Kreĭn–Rutman theorem, the operator $\tilde{J}_1\circ \tilde{J}_2$ admits a unique (up to scalar multiples) positive eigenvector $\varphi_1\in V_+^\circ$ corresponding to its spectral radius $\Lambda_1>0$, and moreover, if $u\in V_+\setminus\{0\}$ satisfies $u\ge (\tilde{J}_1\circ \tilde{J}_2)u$, then necessarily $u$ is a positive multiple of $\varphi_1$ and equality must hold:
	\[
	u = (\tilde{J}_1\circ \tilde{J}_2)u.
	\]
	This contradicts the strict inequality derived above from $\phi\not\equiv0$. 
	Hence, no strong solution $(u,v)\in V\times V$ of \eqref{eq:system} exists when $(\lambda,\mu)=(\lambda_1,\mu_1)\in\mathcal C_1$ and $f,g\ge0$ with $f+g\not\equiv0$. This finishes the proof of Step~1.
	
	Now, arguing by contradiction, assume that there exists a sequence $(\lambda_k',\mu_k')\to(\lambda_1,\mu_1)$ with $\lambda_k'>\lambda_1$ and $\mu_k'>\mu_1$, such that for each $k$ the corresponding system \eqref{eq:system} (with $(\lambda_k',\mu_k')$ in place of $(\lambda,\mu)$) admits a solution $(u_k,v_k)$ for which at least one of $-u_k$ or $-v_k$ does \emph{not} belong to $V_+^\circ$. We shall derive a contradiction.
	
	\medskip
	\emph{Case (a):} Suppose that the sequence $(u_k,v_k)$ is uniformly bounded in $L^\infty(\Omega)$, i.e.
	\[
	\|u_k\|_{L^\infty(\Omega)}+\|v_k\|_{L^\infty(\Omega)}\le C
	\]
	for some $C>0$.  
	By elliptic regularity and the compact embedding $C^{1,\eta}(\overline\Omega)\hookrightarrow C^1(\overline\Omega)$, we may extract a subsequence (not relabeled) such that
	\[
	(u_k,v_k)\to (u,v)\quad\text{in }C^1(\overline\Omega)\times C^1(\overline\Omega).
	\]
	Passing to the limit, we find that $(u,v)$ satisfies \eqref{eq:system} with $(\lambda,\mu)=(\lambda_1,\mu_1)$ and with the same data $f,g\ge0$, contradicting the nonexistence result established in Step~1.
	
	\medskip
	\emph{Case (b):} $\|u_k\|_{L^\infty(\Omega)}+\|v_k\|_{L^\infty(\Omega)}\to\infty$.  
	Define the normalization factors
	\[
	A_k := \|u_k\|_{L^\infty(\Omega)},\qquad B_k := \|v_k\|_{L^\infty(\Omega)}.
	\]
	We first claim that both $A_k$ and $B_k$ diverge to $+\infty$.
	
	\medskip
	\noindent\textbf{Claim 1.} If $A_k + B_k \to \infty$, then necessarily $A_k \to \infty$ and $B_k \to \infty$.
	
	\smallskip
	Indeed, from \eqref{eq:system} and estimate
	$\|T_i h\|_{L^\infty(\Omega)}\le C_i\,\|h\|_{L^p(\Omega)}$, we obtain
	\[
	A_k \le C_1\big(\lambda_k'\|a\|_{L^p(\Omega)}\,B_k^{\beta_1}+\|f\|_{L^p(\Omega)}\big),\qquad
	B_k \le C_2\big(\mu_k'\|b\|_{L^p(\Omega)}\,A_k^{\beta_2}+\|g\|_{L^p(\Omega)}\big).
	\]
	If, for instance, $(B_k)$ were bounded while $(A_k)$ were unbounded, 
	the first inequality would contradict $A_k\to\infty$. 
	The symmetric argument applies if $(A_k)$ were bounded while $(B_k)$ were unbounded. 
	Hence both sequences must diverge, which proves the claim.
	
	\medskip
	Now define the normalized functions
	\[
	\tilde u_k:=\frac{u_k}{A_k},\qquad \tilde v_k:=\frac{v_k}{B_k},\qquad
	\|\tilde u_k\|_{L^\infty(\Omega)}=\|\tilde v_k\|_{L^\infty(\Omega)}=1,
	\]
	and the ratios
	\[
	s_k:=\frac{B_k^{\beta_1}}{A_k}>0,\qquad t_k:=\frac{A_k^{\beta_2}}{B_k}>0.
	\]
	Note that $t_k=s_k^{-\beta_2}$ because $\beta_1\beta_2=1$.
	Then \eqref{eq:system} becomes
	\[
	\begin{cases}
		-\mathcal L_1 \tilde u_k = \lambda_k'\,a\,s_k\,|\tilde v_k|^{\beta_1-1}\tilde v_k + \frac{f}{A_k}& \text{in }\Omega,\\
		-\mathcal L_2 \tilde v_k = \mu_k'\,b\,t_k\,|\tilde u_k|^{\beta_2-1}\tilde u_k + \frac{g}{B_k}& \text{in }\Omega,\\
		\tilde u_k=\tilde v_k=0 & \text{on }\partial\Omega.
	\end{cases}
	\]
	
	\noindent\textbf{Claim 2.} There exist constants $0<c\le C<\infty$ and a subsequence such that $c\le s_k,t_k\le C$ and
	\[
	s_k\to s\in(0,\infty),\qquad t_k\to t=s^{-\beta_2}\in(0,\infty).
	\]
	
	Indeed, dividing the first resolvent inequality by $A_k$ and the second by $B_k$, we get
	\[
	1\le C_1\lambda_k'\|a\|_{L^p(\Omega)}\,s_k + C_1\frac{\|f\|_{L^p(\Omega)}}{A_k},\qquad
	1\le C_2\mu_k'\|b\|_{L^p(\Omega)}\,t_k + C_2\frac{\|g\|_{L^p(\Omega)}}{B_k}.
	\]
	Since $A_k,B_k\to\infty$, these inequalities yield uniform positive lower bounds on $s_k,t_k$ for large $k$. 
	Uniform upper bounds follow similarly. 
	Compactness then ensures the existence of limits, and the relation $t=s^{-\beta_2}$ holds by definition.
	
	Thus, it follows from 
	\begin{equation}\label{eq:norm-det}
		\begin{cases}
			-\mathcal L_1 \tilde u_k=\lambda_k'\,a(x)\,s_k\,|\tilde v_k|^{\beta_1-1}\tilde v_k+\dfrac{f}{A_k}& \text{in } \Omega, \\[2pt]
			-\mathcal L_2 \tilde v_k=\mu_k'\,b(x)\,t_k\,|\tilde u_k|^{\beta_2-1}\tilde u_k+\dfrac{g}{B_k}& \text{in } \Omega, \\[2pt]
			\tilde u_k=\tilde v_k=0 & \text{ on }\partial\Omega
		\end{cases}
	\end{equation}
	that, by Sobolev compactness and the Scalar Weak Maximum Principle, up to a subsequence,
	\[
	\tilde u_k\to u_\ast,\qquad \tilde v_k\to v_\ast\quad\text{in }C^1(\overline\Omega),
	\]
	with $\|u_\ast\|_{L^\infty(\Omega)}=\|v_\ast\|_{L^\infty(\Omega)}=1$ and $u_\ast,v_\ast$ not changing sign. 
	Passing to the limit and using the above claim, we obtain the \emph{homogeneous} limit system
	\[
		\begin{cases}
			-\mathcal L_1 u_\ast=(\lambda_1 s)\,a(x)\,|v_\ast|^{\beta_1-1}v_\ast& \text{in } \Omega, \\[2pt]
			-\mathcal L_2 v_\ast=(\mu_1 t)\,b(x)\,|u_\ast|^{\beta_2-1}u_\ast& \text{in } \Omega, \\[2pt]
			u_\ast=v_\ast=0& \text{ on }\partial\Omega.
		\end{cases}
	\]
	
	Thus, after a common change of sign if necessary, $(u_\ast,v_\ast)\in V_+^\circ\times V_+^\circ$, hence $(\lambda_1 s,\mu_1 t)\in\mathcal C_1$ with principal eigenfunction $(u_\ast,v_\ast)$.
	
	\medskip
	\textbf{Step 2.} Either
	\[
	\text{(I)}\quad (u_k,v_k)\in V_+^\circ\times V_+^\circ\ \text{ for all }k\ge k_0,
	\qquad\text{or}\qquad
	\text{(II)}\quad (-u_k,-v_k)\in V_+^\circ\times V_+^\circ\ \text{ for all }k\ge k_0.
	\]
	We show that each alternative contradicts our standing hypotheses and thereby yields the antimaximum conclusion.
	
	\medskip
	\noindent Case (I): $(u_k,v_k)\in V_+^\circ\times V_+^\circ$ eventually.
	Rewrite the unscaled system at $(\lambda_k',\mu_k')$ as
	\begin{equation}\label{eq:split}
		\begin{cases}
			-\mathcal L_1 u_k=\lambda_1\,a\,|v_k|^{\beta_1-1}v_k + (\lambda_k'-\lambda_1)\,a\,|v_k|^{\beta_1-1}v_k + f& \text{in } \Omega, \\[2pt]
			-\mathcal L_2 v_k=\mu_1\,b\,|u_k|^{\beta_2-1}u_k + (\mu_k'-\mu_1)\,b\,|u_k|^{\beta_2-1}u_k + g& \text{in } \Omega, \\[2pt]
			u_k=v_k=0& \text{ on }\partial\Omega.
		\end{cases}
	\end{equation}
	Since $(u_k,v_k)\in V_+^\circ\times V_+^\circ$, the additional terms
	\[
	(\lambda_k'-\lambda_1)\,a\,|v_k|^{\beta_1-1}v_k\ge0,\qquad
	(\mu_k'-\mu_1)\,b\,|u_k|^{\beta_2-1}u_k\ge0
	\]
	are nonnegative and not identically zero. Applying $T_1,T_2$ to \eqref{eq:split} and using positivity/monotonicity,
	\[
	u_k \ \ge\ T_1\!\big(\lambda_1 a\,|v_k|^{\beta_1-1}v_k+f\big),\qquad
	v_k \ \ge\ T_2\!\big(\mu_1 b\,|u_k|^{\beta_2-1}u_k+g\big).
	\]
	Hence $(u_k,v_k)$ is a \emph{strict supersolution} of the nonhomogeneous system at the \emph{curve point} $(\lambda_1,\mu_1)$.
	Taking $(0,0)$ as a subsolution (since $f,g\ge0$ and $T_i$ are positive), the Lemma 2.1 of \cite{MR1765542} would produce a genuine solution at $(\lambda_1,\mu_1)$—which contradicts the Step~1. Therefore, Case (I) is impossible.
	
	\medskip
	\noindent Case (II): $(-u_k,-v_k)\in V_+^\circ\times V_+^\circ$ eventually.
	This is precisely the antimaximum alternative: for $k\ge k_0$,
	\[
	-u_k>0,\ \ -v_k>0\ \ \text{in }\Omega,\qquad
	\partial_\nu u_k>0,\ \ \partial_\nu v_k>0\ \ \text{on }\partial\Omega,
	\]
	i.e., $-u_k,-v_k\in V_+^\circ$. But this contradicts our contradictory assumption (used to launch the blow-up branch) that \emph{for every $k$ at least one of $-u_k$ or $-v_k$ fails to belong to $V_+^\circ$}. Hence the assumption leading to \eqref{eq:norm-det} must be false, and the antimaximum conclusion holds.
\end{proof}

\section{The Nonhomogeneous Problem in $\mathcal{R}_1$}\label{s8}

This section is important because the region of validity of {\bf (WMP)} in the case $\alpha_1\neq0$ or $\alpha_2\neq0$ lies outside the first quadrant, while $\mathcal R_1$ is an essential region for the development of the theory. Thus, here we investigate qualitative properties of the sets of parameters
$(\lambda,\mu)\in\mathbb R_+^2$ for which the coupled boundary value problem
\begin{equation}\label{eq:new-system-ab}
\begin{cases}
-\,\mathcal L_1 u = \lambda\,a(x)\,u^{\alpha_1}v^{\beta_1} + f_0(x) & \text{in }\Omega,\\[2pt]
-\,\mathcal L_2 v = \mu\,b(x)\,v^{\alpha_2}u^{\beta_2}+ g_0(x) & \text{in }\Omega,\\[2pt]
u=0=v & \text{on }\partial\Omega,
\end{cases}
\end{equation}
admits nonnegative or positive solutions, depending on the regularity and
sign of the source terms $f_0$ and $g_0$.

We introduce the following admissible sets of parameters:
\[
\begin{aligned}
G^{a,b}_{\alpha_1,\beta_1,\alpha_2,\beta_2}(\Omega)
:= \Big\{(\lambda,\mu)\in\mathbb R_+^2:\;&
\text{for any } f_0,g_0\in L^p(\Omega)\ \mbox{with } f_0,g_0\ge 0\ \mbox{a.e. in }\Omega,\\
&\eqref{eq:new-system-ab}\ \text{admits a nonnegative solution}\Big\},
\end{aligned}
\]
\[
\begin{aligned}
O^{a,b}_{\alpha_1,\beta_1,\alpha_2,\beta_2}(\Omega)
:= \Big\{(\lambda,\mu)\in\mathbb R_+^2:\;&
\text{for any } f_0,g_0\in C(\overline\Omega)\ \text{with } f_0,g_0\ge 0\ \text{in }\Omega,\\
&\eqref{eq:new-system-ab}\ \text{admits a nonnegative solution}\Big\},
\end{aligned}
\]
\[
\begin{aligned}
\widetilde O^{a,b}_{\alpha_1,\beta_1,\alpha_2,\beta_2}(\Omega)
:= \Big\{(\lambda,\mu)\in\mathbb R_+^2:\;&
\text{for some } f_0,g_0\in C(\overline\Omega)\ \text{with } f_0,g_0>0\ \text{in }\Omega,\\
&\eqref{eq:new-system-ab}\ \text{admits a positive solution}\Big\}.
\end{aligned}
\]

By definition, these sets satisfy the inclusions
\[
G^{a,b}_{\alpha_1,\beta_1,\alpha_2,\beta_2}(\Omega)\subset
O^{a,b}_{\alpha_1,\beta_1,\alpha_2,\beta_2}(\Omega)\subset
\widetilde O^{a,b}_{\alpha_1,\beta_1,\alpha_2,\beta_2}(\Omega).
\]

\begin{lemma}\label{lem:small-eps-existence}
Let $k\in L^p(\Omega)$ be a nonnegative function. Assume that $0\le \alpha_1,\alpha_2<1$, $\beta_1,\beta_2>0$, and (\ref{cond1}). There exists a positive constant $\varepsilon_*>0$ such that for every
$\varepsilon_0\in(0,\varepsilon_*)$ and every pair of nonnegative
Carath\'eodory functions
$f,g:\Omega\times[0,\infty)\times[0,\infty)\to[0,\infty)$ satisfying
\begin{equation}\label{eq:growth-new}
f(x,t,s)\le \varepsilon_0\,a(x)\,t^{\alpha_1}s^{\beta_1}+k(x),
\qquad
g(x,t,s)\le \varepsilon_0\,b(x)\,s^{\alpha_2}t^{\beta_2}+k(x),
\end{equation}
for all $x\in\Omega$ and $t,s\ge0$, the boundary value problem
\begin{equation}\label{eq:coupled-system}
\begin{cases}
-\mathcal L_1 u = f(x,u,v) & \text{in } \Omega,\\[2pt]
-\mathcal L_2 v = g(x,u,v) & \text{in } \Omega,\\[2pt]
u = v = 0 & \text{on } \partial\Omega,
\end{cases}
\end{equation}
admits a nonnegative solution $(u,v)$. Thus $O^{a,b}_{\alpha_1,\beta_1,\alpha_2,\beta_2}(\Omega)$ is nonempty. 
\end{lemma}

\begin{proof}
For $s\in[0,1]$ and $(u,v)\in C(\overline{\Omega})$, consider the operator $H_{f,g}(s,u,v)=(z,w)$,
where $(z,w)$ is the unique solution of
\[
\begin{cases}
-\,\mathcal L_1 z = s\,f(x,u,v)& \text{in } \Omega, \\[2pt]
-\,\mathcal L_2 w = s\,g(x,u,v)& \text{in } \Omega, \\[2pt]
z=w=0 & \text{on }\partial\Omega.
\end{cases}
\]
By the $L^p$--theory and Sobolev embedding, $H_{f,g}:[0,1]\times C(\overline{\Omega})\to C(\overline{\Omega})$ is well-defined, continuous, and compact.

Assume that $(u,v)\in C(\overline{\Omega})$ is a fixed point of $H_{f,g}(s,\cdot)$ for some
$s\in[0,1]$, that is, $H_{f,g}(s,u,v)=(u,v)$. Then $(u,v)$ satisfies
\[
\begin{cases}
-\,\mathcal L_1 u = s\,f(x,u,v)& \text{in } \Omega, \\[2pt]
-\,\mathcal L_2 v = s\,g(x,u,v)& \text{in } \Omega, \\[2pt]
u=v=0 & \text{on }\partial\Omega.
\end{cases}
\]
Using the $L^p$ estimates and absorbing $s$ into the constant, there exists
$c>0$ such that
\[
\|u\|_{L^\infty(\Omega)} \le c\,\|f(x,u,v)\|_{L^p(\Omega)}, \qquad
\|v\|_{L^\infty(\Omega)} \le c\,\|g(x,u,v)\|_{L^p(\Omega)}.
\]
By \eqref{eq:growth-new}, we obtain
\[
\|u\|_{L^\infty(\Omega)}
\le c\,\varepsilon_0\|a\|_{L^p(\Omega)}\,\|u\|_{L^\infty(\Omega)}^{\alpha_1}
      \|v\|_{L^\infty(\Omega)}^{\beta_1}
    + c\|k\|_{L^p(\Omega)},
\]
\[
\|v\|_{L^\infty(\Omega)}
\le c\,\varepsilon_0\|b\|_{L^p(\Omega)}\,\|u\|_{L^\infty(\Omega)}^{\beta_2}
      \|v\|_{L^\infty(\Omega)}^{\alpha_2}
    + c\|k\|_{L^p(\Omega)}.
\]
Setting
\[
u_0:=\|u\|_{L^\infty(\Omega)}, \qquad
v_0:=\|v\|_{L^\infty(\Omega)}, \qquad
x_0:=\max\{u_0,v_0\},
\]
and
\[
a_0:=c\|a\|_{L^p(\Omega)}, \qquad b_0:=c\|b\|_{L^p(\Omega)}, \qquad k_0:=c\|k\|_{L^p(\Omega)},
\]
we can rewrite the above inequalities as
\begin{equation}\label{eq:ineq-UV}
u_0 \le \varepsilon_0 a_0 u_0^{\alpha_1}v_0^{\beta_1}+k_0,
\qquad
v_0 \le \varepsilon_0 b_0 v_0^{\alpha_2}u_0^{\beta_2}+k_0.
\end{equation}

If $x_0\le 2k_0$, then $u_0\le 2k_0$ and $v_0\le 2k_0$, and hence
$\|(u,v)\|_X\le 2k_0$.

Assume next that $u_0\ge 2k_0$ and $v_0<2k_0$. Since $k_0-0\le u_0/2$, we infer
\[
\frac{u_0}{2}\le \varepsilon_0 a_0 u_0^{\alpha_1}v_0^{\beta_1}
           \le \varepsilon_0 a_0 u_0^{\alpha_1}(2k_0)^{\beta_1},
\]
which yields
\[
u_0^{1-\alpha_1}\le 2\varepsilon_0 a_0 (2k_0)^{\beta_1}.
\]
For $\varepsilon_0>0$ sufficiently small, this implies $u_0<2k_0$, a contradiction.
The case $v_0\ge 2k_0$ and $u_0<2k_0$ is treated analogously (using the second
inequality in \eqref{eq:ineq-UV}) and leads to the same contradiction.

Finally, suppose that $u_0\ge 2k_0$ and $v_0\ge 2k_0$. Then $k_0\le u_0/2$ and $k_0\le v_0/2$,
and we obtain from \eqref{eq:ineq-UV}
\[
u_0^{1-\alpha_1}\le 2\varepsilon_0 a_0 v_0^{\beta_1},
\qquad
v_0^{1-\alpha_2}\le 2\varepsilon_0 b_0 u_0^{\beta_2}.
\]
From the second inequality,
\[
v_0 \le (2\varepsilon_0 b_0)^{\frac{1}{1-\alpha_2}}
     u_0^{\frac{\beta_2}{1-\alpha_2}},
\]
hence
\[
v_0^{\beta_1}
\le (2\varepsilon_0 b_0)^{\frac{\beta_1}{1-\alpha_2}}
    u_0^{\frac{\beta_1\beta_2}{1-\alpha_2}}.
\]
Substituting into the first inequality gives
\[
u_0^{1-\alpha_1}
\le 2\varepsilon_0 a_0\,
(2\varepsilon_0 b_0)^{\frac{\beta_1}{1-\alpha_2}}
u_0^{\frac{\beta_1\beta_2}{1-\alpha_2}}
= C\,\varepsilon_0^{\,1+\frac{\beta_1}{1-\alpha_2}}
u_0^{\frac{\beta_1\beta_2}{1-\alpha_2}},
\]
for some constant $C>0$ depending only on $a_0$ and $b_0$. Therefore,
\[
u_0^{\,1-\alpha_1-\frac{\beta_1\beta_2}{1-\alpha_2}}
\le C\,\varepsilon_0^{\,1+\frac{\beta_1}{1-\alpha_2}}.
\]
Using the condition \eqref{cond1},
\[
1-\alpha_1-\frac{\beta_1\beta_2}{1-\alpha_2}
=1-\alpha_1-\frac{(1-\alpha_1)(1-\alpha_2)}{1-\alpha_2}
=0,
\]
and thus we obtain
\[
1 \le C\,\varepsilon_0^{\,1+\frac{\beta_1}{1-\alpha_2}}.
\]
Choosing $\varepsilon_0>0$ sufficiently small, this is impossible; hence this
case cannot occur.

We conclude that there exists
\[
M = M\!\left(k,\|a\|_{L^p(\Omega)},\|b\|_{L^p(\Omega)}\right) > 0
\]
such that
\[
\|(u,v)\|_X < M
\]
for every fixed point of $H_{f,g}(s,\cdot)$, uniformly for $s\in[0,1]$.
The homotopy invariance of the Leray-Schauder degree in cones implies the
existence of a nonnegative solution to system \eqref{eq:coupled-system}.
\end{proof}

\begin{lemma}\label{lem:OequalsOtilde_new}
	Assume the condition (\ref{cond1}) with $0\le \alpha_1,\alpha_2<1$.
	Then the sets
	\[
	O^{a,b}_{\alpha_1,\beta_1,\alpha_2,\beta_2}(\Omega)
	\quad\text{and}\quad
	\widetilde O^{a,b}_{\alpha_1,\beta_1,\alpha_2,\beta_2}(\Omega)
	\]
	are equal.
\end{lemma}

\begin{proof}
	Let $(\lambda,\mu)\in \widetilde O^{a,b}_{\alpha_1,\beta_1,\alpha_2,\beta_2}(\Omega)$.
	By definition, there exist functions $f_1,g_1\in C(\overline\Omega)$ with
	$f_1,g_1>0$ in $\Omega$ such that the boundary value problem
	\[
	\begin{cases}
		-\,\mathcal L_1 u = \lambda\,a(x)\,u^{\alpha_1}v^{\beta_1} + f_1(x)
		& \text{in }\Omega,\\[2pt]
		-\,\mathcal L_2 v = \mu\,b(x)\,v^{\alpha_2}u^{\beta_2} + g_1(x)
		& \text{in }\Omega,\\[2pt]
		u=v=0 & \text{on }\partial\Omega
	\end{cases}
	\]
	admits a positive solution $(u_1,v_1)$.
	
	Let $f_0,g_0\in C(\overline\Omega)$ satisfy $0\le f_0\le f_1$ and
	$0\le g_0\le g_1$ in $\Omega$. Then $(0,0)$ is a subsolution and $(u_1,v_1)$
	is a positive supersolution for the problem with data $(f_0,g_0)$.
	By the Lemma 2.1 of \cite{MR1765542}, the boundary value problem
	\[
	\begin{cases}
		-\,\mathcal L_1 u = \lambda\,a(x)\,u^{\alpha_1}v^{\beta_1} + f_0(x)
		& \text{in }\Omega,\\[2pt]
		-\,\mathcal L_2 v = \mu\,b(x)\,v^{\alpha_2}u^{\beta_2} + g_0(x)
		& \text{in }\Omega,\\[2pt]
		u=v=0 & \text{on }\partial\Omega
	\end{cases}
	\]
	admits a nonnegative solution.
	
	Now let $f_0,g_0\in C(\overline\Omega)$ be arbitrary nonnegative functions. Consider the auxiliary problem
	\begin{equation}\label{eq:eps-problem-new}
		\begin{cases}
			-\,\mathcal L_1 u_\varepsilon
			= \lambda\,a(x)\,u_\varepsilon^{\alpha_1}v_\varepsilon^{\beta_1}
			+ \varepsilon\,f_0(x)
			& \text{in }\Omega,\\[2pt]
			-\,\mathcal L_2 v_\varepsilon
			= \mu\,b(x)\,v_\varepsilon^{\alpha_2}u_\varepsilon^{\beta_2}
			+ \varepsilon^{\frac{1-\alpha_1}{\beta_1}}\,g_0(x)
			& \text{in }\Omega,\\[2pt]
			u_\varepsilon=v_\varepsilon=0
			& \text{on }\partial\Omega.
		\end{cases}
	\end{equation}
	Choose $\varepsilon>0$ sufficiently small so that
	\[
	\varepsilon\,f_0\le f_1
	\qquad\text{and}\qquad
	\varepsilon^{\frac{1-\alpha_1}{\beta_1}}\,g_0\le g_1
	\qquad\text{in }\Omega.
	\]
	By the previous step, \eqref{eq:eps-problem-new} admits a nonnegative solution
	$(u_\varepsilon,v_\varepsilon)$.
	
	Define
	\[
	u:=\varepsilon^{-1}u_\varepsilon,
	\qquad
	v:=\varepsilon^{-\frac{1-\alpha_1}{\beta_1}}v_\varepsilon.
	\]
	Using the linearity of $\mathcal L_1$ and the first equation in
	\eqref{eq:eps-problem-new}, we obtain in $\Omega$
	\begin{align*}
		-\,\mathcal L_1 u
		&=\varepsilon^{-1}(-\,\mathcal L_1 u_\varepsilon)\\
		&=\lambda\,a(x)u^{\alpha_1}v^{\beta_1}
		+ f_0(x),
	\end{align*}
since $u_\varepsilon=\varepsilon u$ and
	$v_\varepsilon=\varepsilon^{\frac{1-\alpha_1}{\beta_1}}v$.
	
	Similarly, using the linearity of $\mathcal L_2$ and the second equation in
	\eqref{eq:eps-problem-new}, we get
	\begin{align*}
		-\,\mathcal L_2 v
		&=\varepsilon^{-\frac{1-\alpha_1}{\beta_1}}(-\,\mathcal L_2 v_\varepsilon)\\
		&=\mu\,b(x)v^{\alpha_2}u^{\beta_2}
		+ g_0(x).
	\end{align*}
	
	Therefore $(u,v)$ solves the original problem with data $(f_0,g_0)$, and so
	$(\lambda,\mu)\in O^{a,b}_{\alpha_1,\beta_1,\alpha_2,\beta_2}(\Omega)$.
	
	The reverse inclusion
	\[
	O^{a,b}_{\alpha_1,\beta_1,\alpha_2,\beta_2}(\Omega)
	\subset
	\widetilde O^{a,b}_{\alpha_1,\beta_1,\alpha_2,\beta_2}(\Omega)
	\]
	is immediate from the definitions. This concludes the proof.
\end{proof}

\begin{lemma}\label{lem:O-open}
	The set
	\[
	O^{a,b}_{\alpha_1,\beta_1,\alpha_2,\beta_2}(\Omega)\subset\mathbb R_+^2
	\]
	is open.
\end{lemma}

\begin{proof}
	Take $(\lambda_0,\mu_0)\in
	O^{a,b}_{\alpha_1,\beta_1,\alpha_2,\beta_2}(\Omega)$.
	Fix arbitrary functions $f_0,g_0\in C(\overline\Omega)$ such that
	$f_0,g_0>0$ in $\Omega$.
	By definition of
	$O^{a,b}_{\alpha_1,\beta_1,\alpha_2,\beta_2}(\Omega)$,
	the boundary value problem
	\begin{equation}\label{eq:open-prob0-new}
		\begin{cases}
			-\,\mathcal L_1 u
			= \lambda_0\,a(x)\,u^{\alpha_1}v^{\beta_1} + f_0(x)
			& \text{in }\Omega,\\[2pt]
			-\,\mathcal L_2 v
			= \mu_0\,b(x)\,v^{\alpha_2}u^{\beta_2} + g_0(x)
			& \text{in }\Omega,\\[2pt]
			u=v=0 & \text{on }\partial\Omega
		\end{cases}
	\end{equation}
	admits a nonnegative solution $(u_0,v_0)$.
	
	Choose $\eta>0$ so small that
	\[
	f_0-\eta>0
	\quad\text{and}\quad
	g_0-\eta>0
	\qquad \text{in }\Omega.
	\]
	
	We claim that there exists $\varepsilon>0$ such that, for every
	$(\lambda,\mu)\in\mathbb R_+^2$ satisfying
	$|\lambda-\lambda_0|<\varepsilon$ and $|\mu-\mu_0|<\varepsilon$,
	the pair $(u_0,v_0)$ is a positive supersolution of
	\begin{equation}\label{eq:open-prob1-new}
		\begin{cases}
			-\,\mathcal L_1 u
			= \lambda\,a(x)\,u^{\alpha_1}v^{\beta_1} + f_0(x)-\eta
			& \text{in }\Omega,\\[2pt]
			-\,\mathcal L_2 v
			= \mu\,b(x)\,v^{\alpha_2}u^{\beta_2} + g_0(x)-\eta
			& \text{in }\Omega,\\[2pt]
			u=v=0 & \text{on }\partial\Omega.
		\end{cases}
	\end{equation}
	
	Indeed, since $(u_0,v_0)$ solves \eqref{eq:open-prob0-new}, we have in $\Omega$
	\[
	-\,\mathcal L_1 u_0
	=\lambda_0\,a(x)\,u_0^{\alpha_1}v_0^{\beta_1}+f_0(x)
	\ge \lambda\,a(x)\,u_0^{\alpha_1}v_0^{\beta_1}+f_0(x)-\eta
	\]
	provided that
	\[
	(\lambda-\lambda_0)\,a(x)\,u_0^{\alpha_1}v_0^{\beta_1}\le \eta
	\qquad\text{in }\Omega.
	\]
	
	Similarly,
	\[
	-\,\mathcal L_2 v_0
	=\mu_0\,b(x)\,v_0^{\alpha_2}u_0^{\beta_2}+g_0(x)
	\ge \mu\,b(x)\,v_0^{\alpha_2}u_0^{\beta_2}+g_0(x)-\eta
	\]
	provided that
	\[
	(\mu-\mu_0)\,b(x)\,v_0^{\alpha_2}u_0^{\beta_2}\le \eta
	\qquad\text{in }\Omega.
	\]
	
	Since $a,b$ are fixed and $(u_0,v_0)\in C(\overline\Omega)\times C(\overline\Omega)$,
	the functions
	\[
	a\,u_0^{\alpha_1}v_0^{\beta_1},
	\qquad
	b\,v_0^{\alpha_2}u_0^{\beta_2}
	\]
	are bounded in $\Omega$.
	Hence we may choose $\varepsilon>0$ so small that
	\[
	|\lambda-\lambda_0|\,
	\|a\,u_0^{\alpha_1}v_0^{\beta_1}\|_{L^\infty(\Omega)}\le \eta,
	\qquad
	|\mu-\mu_0|\,
	\|b\,v_0^{\alpha_2}u_0^{\beta_2}\|_{L^\infty(\Omega)}\le \eta.
	\]
	
	On the other hand, $(0,0)$ is a subsolution of \eqref{eq:open-prob1-new} since
	$f_0-\eta>0$ and $g_0-\eta>0$ in $\Omega$.
	By Lemma 2.1 of \cite{MR1765542},
	problem \eqref{eq:open-prob1-new} admits a positive solution for every
	$(\lambda,\mu)$ with
	$|\lambda-\lambda_0|<\varepsilon$ and $|\mu-\mu_0|<\varepsilon$.
	
	Therefore,
	\[
	(\lambda,\mu)\in
	\widetilde O^{a,b}_{\alpha_1,\beta_1,\alpha_2,\beta_2}(\Omega).
	\]
	Using the assumed equality
	\[
	O^{a,b}_{\alpha_1,\beta_1,\alpha_2,\beta_2}(\Omega)
	=
	\widetilde O^{a,b}_{\alpha_1,\beta_1,\alpha_2,\beta_2}(\Omega),
	\]
	we conclude that
	$(\lambda,\mu)\in
	O^{a,b}_{\alpha_1,\beta_1,\alpha_2,\beta_2}(\Omega)$.
	
	Since $f_0,g_0>0$ were arbitrary, it follows that
	$O^{a,b}_{\alpha_1,\beta_1,\alpha_2,\beta_2}(\Omega)$
	is open in $\mathbb R_+^2$.
\end{proof}

\begin{proposition}\label{prop:ray-cutoff}
	Fix $c>0$. Then there exists a constant $t_0=t_0(c)>0$ such that
	\[
	(t,ct)\notin O^{a,b}_{\alpha_1,\beta_1,\alpha_2,\beta_2}(\Omega)
	\qquad \text{for every } t\ge t_0.
	\]
	
	Equivalently, there exist functions $f_0,g_0\in C(\overline\Omega)$ with
	$f_0,g_0>0$ in $\Omega$ such that the boundary value problem
	\begin{equation}\label{eq:Pt}
		\begin{cases}
			-\,\mathcal L_1 u = t\,a(x)\,u^{\alpha_1}v^{\beta_1}+f_0(x)
			& \text{in }\Omega,\\[2pt]
			-\,\mathcal L_2 v = ct\,b(x)\,v^{\alpha_2}u^{\beta_2}+g_0(x)
			& \text{in }\Omega,\\[2pt]
			u=v=0 & \text{on }\partial\Omega,
		\end{cases}
	\end{equation}
	admits no nonnegative solution.
\end{proposition}

\begin{proof}
	Fix $c>0$ and assume that there exist functions
	\[
	\phi,\psi\in W^{2,p}(\Omega)
	\]
	and constants $C_1=C_1(c)>0$, $C_2=C_2(c)>0$ such that
	\begin{equation}\label{8.9}
		\phi>0,\ \psi>0 \quad \text{in }\Omega,
		\qquad
		\phi=\psi=0 \quad \text{on }\partial\Omega,
	\end{equation}
	and
	\begin{equation}\label{eq:barrier-ineq-2}
		-\mathcal L_1\phi \le C_1\,a(x)\,\phi^{\alpha_1}\psi^{\beta_1}
		\quad \text{in }\Omega,
		\qquad
		-\mathcal L_2\psi \le C_2\,b(x)\,\psi^{\alpha_2}\phi^{\beta_2}
		\quad \text{in }\Omega.
	\end{equation}

For instance, one may take
$C_1=\lambda_1$ and $C_2=\mu_1$,
with $c=\frac{\mu_1}{\lambda_1}$,
$(\lambda_1,\mu_1)\in\mathcal C_1$,
and $(\phi,\psi)$ the eigenfunction pair associated with
$(\lambda_1,\mu_1)$.
	
	Fix once and for all functions $f_0,g_0\in C(\overline\Omega)$ such that
	\begin{equation}\label{eq:fg-positive}
		f_0>0,\qquad g_0>0\qquad\text{in }\Omega.
	\end{equation}
	
	Assume by contradiction that there exists a sequence $t_k\to+\infty$
	such that $(t_k,ct_k)\in O^{a,b}_{\alpha_1,\beta_1,\alpha_2,\beta_2}(\Omega)$.
	By the definition of $O^{a,b}_{\alpha_1,\beta_1,\alpha_2,\beta_2}(\Omega)$
	and the choice \eqref{eq:fg-positive}, for each $k$ the problem
	\eqref{eq:Pt} with $t=t_k$ admits a nonnegative solution $(u_k,v_k)$.
	Since $f_0,g_0>0$ in $\Omega$, the Scalar Strong Maximum Principle yields
	\[
		u_k>0,\quad v_k>0\quad\text{in }\Omega,\qquad u_k=v_k=0\quad\text{on }\partial\Omega,
	\]
	and Hopf's lemma implies
	\[
		\partial_\nu u_k<0,\qquad \partial_\nu v_k<0\qquad\text{on }\partial\Omega.
	\]

	Since
	\[
		r=\frac{1-\alpha_1}{\beta_1}=\frac{\beta_2}{1-\alpha_2},
	\]
	we have
	\begin{equation}\label{eq:theta-identities}
		\alpha_1+r\beta_1=1,
		\qquad
		\beta_2+r\alpha_2=r.
	\end{equation}
	
	Now, for each $k$, set
	\[
		S_k:=\Big\{\,s>0:\ u_k\ge s\,\phi\ \text{and}\ v_k\ge s^{r}\,\psi
		\ \text{in }\Omega\Big\}.
	\]
	By Hopf's lemma, the set $S_k$ is nonempty.
	Define
	\[
		s_k^*:=\sup S_k\in(0,+\infty).
	\]
	By definition, for every $0<s<s_k^*$ we have
	\[
		u_k\ge s\,\phi,\qquad v_k\ge s^r\,\psi\quad\text{in }\Omega.
	\]
	
	Letting $s\uparrow s_k^*$ and using continuity, we obtain
	\begin{equation}\label{eq:ineq-at-snstar}
		u_k\ge s_k^*\,\phi,\qquad v_k\ge (s_k^*)^r\,\psi\quad\text{in }\Omega.
	\end{equation}
	
	From \eqref{eq:ineq-at-snstar} and monotonicity of $t\mapsto t^\gamma$ on
	$[0,+\infty)$ we deduce
	\[
	u_k^{\alpha_1}v_k^{\beta_1}
	\ge
	(s_k^*)^{\alpha_1}\phi^{\alpha_1}\big((s_k^*)^r\psi\big)^{\beta_1}
	=
	(s_k^*)^{\alpha_1+r\beta_1}\phi^{\alpha_1}\psi^{\beta_1}.
	\]
	Using $\alpha_1+r\beta_1=1$ from \eqref{eq:theta-identities}, we obtain
	\begin{equation}\label{eq:lower-1}
		u_k^{\alpha_1}v_k^{\beta_1}\ge s_k^*\,\phi^{\alpha_1}\psi^{\beta_1}
		\quad\text{in }\Omega.
	\end{equation}
	Similarly,
	\[
	v_k^{\alpha_2}u_k^{\beta_2}
	\ge
	\big((s_k^*)^r\psi\big)^{\alpha_2}(s_k^*)^{\beta_2}\phi^{\beta_2}
	=
	(s_k^*)^{r\alpha_2+\beta_2}\psi^{\alpha_2}\phi^{\beta_2},
	\]
	and using $r\alpha_2+\beta_2=r$ from \eqref{eq:theta-identities} we get
	\begin{equation}\label{eq:lower-2}
		v_k^{\alpha_2}u_k^{\beta_2}\ge (s_k^*)^r\,\psi^{\alpha_2}\phi^{\beta_2}
		\quad\text{in }\Omega.
	\end{equation}
	
	Now, define
	\[
		U_k:=u_k-s_k^*\phi,\qquad V_k:=v_k-(s_k^*)^r\psi.
	\]
	By \eqref{eq:ineq-at-snstar} we have $U_k\ge0$ and $V_k\ge0$ in $\Omega$, and
	$U_k=V_k=0$ on $\partial\Omega$.
	Using \eqref{eq:Pt} (with $t=t_k$) and linearity of $\mathcal L_1,\mathcal L_2$,
	we compute in $\Omega$,
	\begin{align}
		-\mathcal L_1 U_k
		&=
		t_k a(x)u_k^{\alpha_1}v_k^{\beta_1}+f_0(x) - s_k^*\big(-\mathcal L_1\phi\big),
		\label{eq:L1Un}\\
		-\mathcal L_2 V_k
		&=
		ct_k b(x)v_k^{\alpha_2}u_k^{\beta_2}+g_0(x) - (s_k^*)^r\big(-\mathcal L_2\psi\big).
		\label{eq:L2Vn}
	\end{align}
	Using \eqref{eq:barrier-ineq-2} in \eqref{eq:L1Un} and then \eqref{eq:lower-1},
	we obtain
	\begin{align}
		-\mathcal L_1 U_k
		&\ge
		t_k a(x)u_k^{\alpha_1}v_k^{\beta_1}+f_0(x)
		- s_k^* C_1\,a(x)\phi^{\alpha_1}\psi^{\beta_1}\notag\\
		&\ge
		t_k a(x)\,s_k^*\,\phi^{\alpha_1}\psi^{\beta_1}+f_0(x)
		- s_k^* C_1\,a(x)\phi^{\alpha_1}\psi^{\beta_1}\notag\\
		&=
		s_k^*\,a(x)\phi^{\alpha_1}\psi^{\beta_1}\,(t_k-C_1)+f_0(x).
		\label{eq:Un-positive}
	\end{align}
	Similarly, using \eqref{eq:barrier-ineq-2} in \eqref{eq:L2Vn} and then
	\eqref{eq:lower-2}, we get
	\begin{align}
		-\mathcal L_2 V_k
		&\ge
		ct_k b(x)v_k^{\alpha_2}u_k^{\beta_2}+g_0(x)
		- (s_k^*)^r C_2\,b(x)\psi^{\alpha_2}\phi^{\beta_2}\notag\\
		&\ge
		ct_k b(x)\,(s_k^*)^r\,\psi^{\alpha_2}\phi^{\beta_2}+g_0(x)
		- (s_k^*)^r C_2\,b(x)\psi^{\alpha_2}\phi^{\beta_2}\notag\\
		&=
		(s_k^*)^r\,b(x)\psi^{\alpha_2}\phi^{\beta_2}\,(ct_k-C_2)+g_0(x).
		\label{eq:Vn-positive}
	\end{align}
	Since $t_k\to+\infty$, we may choose $k_0$ such that for all $k\ge k_0$,
	\[
		t_k>C_1\qquad\text{and}\qquad ct_k>C_2.
	\]
	Then \eqref{eq:Un-positive}--\eqref{eq:Vn-positive} and \eqref{eq:fg-positive}
	imply that for every $k\ge k_0$,
	\[
		-\mathcal L_1 U_k>0,\qquad -\mathcal L_2 V_k>0\quad\text{in }\Omega,
		\qquad
		U_k=V_k=0\quad\text{on }\partial\Omega.
	\]
	By the Scalar Strong Maximum Principle and Hopf's lemma, for $k\ge k_0$ we have
	\[
		U_k>0,\quad V_k>0\quad\text{in }\Omega,
		\qquad
		\partial_\nu U_k<0,\quad \partial_\nu V_k<0\quad\text{on }\partial\Omega.
	\]
	
	Fix $k\ge k_0$. Since $U_k>0$ in $\Omega$, $U_k=0$ on $\partial\Omega$, and
	$\partial_\nu U_k<0$ on $\partial\Omega$, while $\phi>0$ in $\Omega$,
	$\phi=0$ on $\partial\Omega$, and $\partial_\nu\phi<0$ on $\partial\Omega$,
	it follows from Hopf's lemma and the boundary regularity that there exists
	$m_k>0$ such that
	\[
	U_k\ge m_k\,\phi \quad \text{in }\Omega.
	\]
	In particular,
	\[
	u_k=s_k^*\phi+U_k\ge (s_k^*+m_k)\phi \qquad \text{in }\Omega.
	\]
	
	Similarly, since $V_k>0$ in $\Omega$, $V_k=0$ on $\partial\Omega$, and
	$\partial_\nu V_k<0$ on $\partial\Omega$, while $\psi>0$ in $\Omega$,
	$\psi=0$ on $\partial\Omega$, and $\partial_\nu\psi<0$ on $\partial\Omega$,
	there exists $\ell_k>0$ such that
	\[
	V_k\ge \ell_k\,\psi \quad \text{in }\Omega,
	\]
	and hence
	\[
	v_k=(s_k^*)^r\psi+V_k
	\ge\big((s_k^*)^r+\ell_k\big)\psi
	\quad\text{in }\Omega.
	\]
	
	Since the function $t\mapsto t^r$ is continuous and strictly increasing
	on $[0,+\infty)$, there exists $\varepsilon_k\in(0,m_k)$ such that
	\[
	(s_k^*+\varepsilon_k)^r\le (s_k^*)^r+\ell_k.
	\]
	Then
	\[
	u_k\ge (s_k^*+\varepsilon_k)\phi \quad \text{in }\Omega,
	\]
	and
	\[
	v_k\ge (s_k^*+\varepsilon_k)^r\psi \quad\text{in }\Omega.
	\]
	In particular, $s_k^*+\varepsilon_k\in S_k$, contradicting
	$s_k^*=\sup S_k$. This contradiction proves the result.
	
	Finally, observe that $C_1=\lambda_1$ and $C_2=\mu_1$ with $c=\frac{\mu_1}{\lambda_1}$ and $(\lambda_1,\mu_1)\in \mathcal{C}_1$ are the best constants, because by Theorem \ref{MP}, there exists no pair of functions satisfying \eqref{8.9} and \eqref{eq:barrier-ineq-2} in $\mathcal{R}_1$.  
\end{proof}

Let $(\lambda_1, \mu_1) \in \mathcal{C}_1$. We denote by $(\varphi,\psi)$ a positive eigenfunction pair associated to $(\lambda_1, \mu_1)$. Set
\[
\varphi_0=\varphi,\ \ \psi_0=\left(\frac{\lambda_1}{\lambda_0}\right)^{\frac{1}{\beta_1}}\psi,\ \ \ \lambda_0=\lambda_1^{\frac{1-\alpha_2}{1-\alpha_2+\beta_1}}\mu_1^{\frac{\beta_1}{1-\alpha_2+\beta_1}}.
\]

Thus $(\varphi_0,\psi_0)$ is a solution of the problem
\[
	\begin{cases}
		-\,\mathcal{L}_1 u=\lambda_0\,a(x)\,\vert u\vert^{\alpha_1-1}u\vert v\vert^{\beta_1-1}v & \text{in }\Omega,\\[2pt]
		-\,\mathcal{L}_2 v=\lambda_0\,b(x)\,\vert v\vert^{\alpha_2-1}v\vert u\vert^{\beta_2-1}u & \text{in }\Omega,\\[2pt]
		u=v=0 & \text{on }\partial\Omega.
	\end{cases}
	\]

This produces a spectral curve $\Lambda_{\lambda_0}$ parametrized by
\[
\Lambda_{\lambda_0}(\lambda)=(\lambda,\mu(\lambda)),
\]
where $\mu(\lambda)=\frac{\lambda_0^{\frac{1-\alpha_2+\beta_1}{\beta_1}}}{\lambda^{\frac{1-\alpha_2}{\beta_1}}}$, $\lambda>0$.

Moreover, the curve may be reparametrized as
\[
\Lambda_{\lambda_0}(c)=(\lambda_1(c),\mu_1(c)),
\]
where $\mu_1(c) = c\lambda_1(c)$, for $c=\frac{\mu_1(c)}{\lambda_1(c)}= \left(\frac{\lambda_0}{\lambda_1(c)}\right)^{\frac{1-\alpha_2+\beta_1}{\beta_1}}$.

Thus, by (\ref{eq:Lambda-conds}) and (\ref{eq:rho-rel-main}), we have

\[
\tilde{t}_c:=\lambda_1(c)=\frac{\lambda_0}{c^{\frac{\beta_1}{1-\alpha_2+\beta_1}}}=\frac{\rho_1^{-\frac{\beta_1\beta_2}{1-\alpha_2+\beta_1}}}{c^{\frac{\beta_1}{1-\alpha_2+\beta_1}}}=\frac{\rho_2^{-\frac{\beta_1(1-\alpha_2)}{1-\alpha_2+\beta_1}}}{c^{\frac{\beta_1}{1-\alpha_2+\beta_1}}}.
\]

In particular, when $\,\mathcal{L}_1=\,\mathcal{L}_2$, $a(x)=b(x)$ and $\alpha_1=\alpha_2=\beta_1=\beta_2=\frac{1}{2}$, we have

\[
\lambda_1(-\,\mathcal{L}_1-\lambda a(x))=\lambda_0=\rho_1^{-\frac{1}{4}}=\rho_2^{-\frac{1}{4}}\ \ \ \ \ \text{and} \ \ \ \ \ \tilde{t}_c=\frac{\lambda_1(-\,\mathcal{L}_1-\lambda a(x))}{c^{\frac{1}{2}}}.
\]

This way, one recovers the notation from \cite{MR1765542},
\[
\lambda_1(c)=\tilde{t}_c,\  \mu_1(c)=c\tilde{t}_c.
\]

Therefore, $\mathcal{C}_1=\{\Lambda_{\lambda_0}(c): c>0\}$.

For each $c>0$, define
\begin{equation}
t_c^{\ast} = \sup\{t>0 : (t,ct) \in O^{a,b}_{\alpha_1,\beta_1,\alpha_2,\beta_2}(\Omega)\}. \label{3.8}
\end{equation}
According to Lemma \ref{lem:small-eps-existence} and Proposition \ref{prop:ray-cutoff}, $t_c^{\ast}$ is a well-defined positive number.

\begin{theorem} \label{theo1} For each $c>0$ the numbers $t_c^{\ast}$ and $\tilde{t}_c$ are equal.
\end{theorem}
\begin{proof} Let $c>0$. By the definition (\ref{3.8}), there is a sequence $(t_k)_{k \geq 1}$ of positive numbers converging to $t_c^{\ast}$ and such that $(t_k,ct_k) \in O^{a,b}_{\alpha_1,\beta_1,\alpha_2,\beta_2}(\Omega)$ for all $k \geq 1$. Therefore, given a fixed pair $(f_0,g_0) \in (C(\overline{\Omega}))^2$ with $f_0, g_0 > 0$ in $\overline{\Omega}$, the system
\begin{equation}\label{3.5}
\left\{
\begin{array}{llll}
-\,\mathcal{L}_1 u=t_k a(x) \vert u\vert^{\alpha_1-1}u\vert v\vert^{\beta_1-1}v +f_0(x) & {\rm in} \ \ \Omega,\\
-\,\mathcal{L}_2 v=ct_k b(x) \vert v\vert^{\alpha_2-1}v\vert u\vert^{\beta_2-1}u +g_0(x) & {\rm in} \ \ \Omega,\\
u=v=0 & {\rm on} \ \ \partial\Omega
\end{array}
\right.
\end{equation}
possesses a positive strong solution $(u_k,v_k)$ for each $k$. We claim that $\Vert(u_k,v_k)\Vert_X \rightarrow +\infty$ as $k \rightarrow + \infty$. Otherwise, there is $C > 0$ such that, up to a subsequence, we get $\| u_k \|_{L^\infty(\Omega)}, \| v_k \|_{L^\infty(\Omega)} \leq C$ for all $k \geq 1$. Using $L^p$ estimate in each equation of the system (\ref{3.5}), we have $$\| u_k \|_{W^{2,p}(\Omega)}, \| v_k \|_{W^{2,p}(\Omega)} \leq C_0$$ for some positive constant $C_0$ independent of $k$. Since $\Omega$ is bounded, by Arzelá-Ascoli theorem, up to a subsequence, we obtain the uniform convergence $u_k \rightarrow u$ and $v_k \rightarrow v$ locally in $\Omega$. Applying $L^p$ estimates, we have
\begin{eqnarray*}
&& \Vert u_m - u_k \Vert_{W^{2,p}(\Omega)}\leq c_1\left(\vert t_m - t_k \vert + \Vert u_m^{\alpha_1}v_m^{\beta_1} - u_k^{\alpha_1}v_k^{\beta_1} \Vert_{L^p(\Omega)}\right), \\
&& \Vert v_m - v_k \Vert_{W^{2,p}(\Omega)}\leq c_1 \left(\vert t_m - t_k \vert + \Vert v_m^{\alpha_2}u_m^{\beta_2} - v_k^{\alpha_2}u_k^{\beta_2} \Vert_{L^p(\Omega)}\right),
\end{eqnarray*}
for some positive constant $c_1$ independent of $k$. Then, $u_k \rightarrow u$  and $v_k \rightarrow v$ in $W^{2,p}(\Omega)$, so that $(u,v)$ is a positive strong solution of the system

\[
\left\{
\begin{array}{llll}
-\,\mathcal{L}_1 u= t_c^{\ast} a(x)  \vert u\vert^{\alpha_1-1}u\vert v\vert^{\beta_1-1}v +f_0(x) & {\rm in} \ \ \Omega,\\
-\,\mathcal{L}_2 v=c t_c^{\ast} b(x) \vert v\vert^{\alpha_2-1}v\vert u\vert^{\beta_2-1}u +g_0(x) & {\rm in} \ \ \Omega,\\
u=v=0 & {\rm on} \ \ \partial\Omega.
\end{array}
\right.
\]
So, thanks to Lemma \ref{lem:OequalsOtilde_new}, we have $(t_c^{\ast},c t_c^{\ast}) \in O^{a,b}_{\alpha_1,\beta_1,\alpha_2,\beta_2}(\Omega)$, and to Lemma \ref{lem:O-open}, we get $(t_c^{\ast} + \varepsilon,c(t_c^{\ast} + \varepsilon)) \in O^{a,b}_{\alpha_1,\beta_1,\alpha_2,\beta_2}(\Omega)$ for sufficiently small $\varepsilon>0$, contradicting the definition of $t_c^{\ast}$. This proves the statement.

Define

\[
\varphi_k\ =\ \frac{u_k}{\| u_k \|_{L^\infty(\Omega)}\ +\ \| v_k \|^{r}_{L^\infty(\Omega)}} \ \ \ \ {\rm and}\ \ \ \ \psi_k\ =\ \frac{v_k}{\left(\| u_k \|_{L^\infty(\Omega)}\ +\ \| v_k \|^{r}_{L^\infty(\Omega)}\right)^{\frac{1}{r}}}.
\]
Notice that $\| \varphi_k \|_{L^\infty(\Omega)}+ \| \psi_k \|^{r}_{L^\infty(\Omega)}= 1$. Moreover, we obtain
\[
\left\{
\begin{array}{llll}
-\,\mathcal{L}_1 \varphi_k = t_k a(x) \varphi_k^{\alpha_1}\psi_k^{\beta_1} + f_0(x) \left(\| u_k \|_{L^\infty(\Omega)}\ +\ \| v_k \|^{r}_{L^\infty(\Omega)}\right)^{-1} & {\rm in} \ \ \Omega,\\
-\,\mathcal{L}_2 \psi_k = ct_k b(x) \psi_{k}^{\alpha_2}\varphi_k^{\beta_2} + g_0(x) \left(\| u_k \|_{L^\infty(\Omega)} + \| v_k \|^{r}_{L^\infty(\Omega)}\right)^{-\frac{1}{r}} & {\rm in} \ \ \Omega,\\
\varphi_k =\psi_k =0 & {\rm on} \ \ \partial\Omega.
\end{array}
\right.
\]
Proceeding as above, up to a subsequence, we have $\varphi_k \rightarrow \varphi$ and $\psi_k \rightarrow \psi$ in $W^{2,p}(\Omega)$. Then, $\| \varphi \|_{L^\infty(\Omega)} + \| \psi \|^{r}_{L^\infty(\Omega)} = 1$ and $(\varphi,\psi) \in (W^{2,p}(\Omega))^2$ is a strong solution of the system

\[
\left\{
\begin{array}{llll}
-\,\mathcal{L}_1 \varphi = t_c^{\ast} a(x) \varphi^{\alpha_1}\psi^{\beta_1}  & {\rm in} \ \ \Omega,\\
-\,\mathcal{L}_2 \psi =c t_c^{\ast} b(x) \psi^{\alpha_2}\varphi_k^{\beta_2}  & {\rm in} \ \ \Omega,\\
\varphi = \psi = 0 & {\rm on} \ \ \partial\Omega.
\end{array}
\right.
\]
So, by the Scalar Strong Maximum Principle, it follows that $(t_c^{\ast}, c t_c^{\ast}) \in \mathcal{C}_1$. This completes the proof.
\end{proof}

\begin{corollary} \label{closeness}
The set of principal eigenvalues $\mathcal{C}_1$ is closed in $\R_+^2$.
\end{corollary}
\begin{proof} Let $((\lambda_k, \mu_k))_{k \geq 1} \subset \mathcal{C}_1$ be a sequence converging to $(\lambda,\mu)$. For each $k$, consider a positive eigenfunction pair $(\varphi_k, \psi_k)$ associated to $(\lambda_k, \mu_k)$ such that $||(\varphi_k, \psi_k)||_X = 1$. Proceeding in a similar way as in the proof of Theorem \ref{theo1} with $f_0 = 0 = g_0$ and using $L^p$ estimates, we obtain, up to a subsequence, that $\varphi_k \rightarrow \varphi$ and $\psi_k \rightarrow \psi$ in $W^{2,p}(\Omega)$. Letting $k \rightarrow + \infty$ and applying the Scalar Strong Maximum Principle, we obtain that $(\varphi, \psi) \in (W^{2,p}(\Omega))^2$ is a positive eigenfunction pair associated to $(\lambda,\mu)$, that is, $(\lambda,\mu) \in \mathcal{C}_1$ as desired.
\end{proof}

\begin{corollary} \label{existence}
The sets $O^{a,b}_{\alpha_1,\beta_1,\alpha_2,\beta_2}(\Omega)$ and $\{(t, ct):\; c>0 \ {\rm and}\ 0 < t < \tilde{t}_c\}$ are equal. In particular, we have $\mathcal{C}_1 = \partial O^{a,b}_{\alpha_1,\beta_1,\alpha_2,\beta_2}(\Omega)$ and $$\mathcal{R}_1 = O^{a,b}_{\alpha_1,\beta_1,\alpha_2,\beta_2}(\Omega).$$
\end{corollary}
\begin{proof} Thanks to Lemma \ref{lem:O-open}, $O^{a,b}_{\alpha_1,\beta_1,\alpha_2,\beta_2}(\Omega)$ is open in $\R_+^2$. So, by the definition of $t_c^{\ast}$, we get $O^{a,b}_{\alpha_1,\beta_1,\alpha_2,\beta_2}(\Omega) \subset \{(t, ct):\; c>0 \ {\rm and}\ 0 < t < t_c^{\ast}\}$.

Conversely, let $(t, ct)$ with $c > 0$ and $0 < t < \tilde{t}_c$. Again, by the definition of $t_c^{\ast}$ there exists $t < t_1 < t_c^{\ast}$ such that $(t_1, ct_1) \in O^{a,b}_{\alpha_1,\beta_1,\alpha_2,\beta_2}(\Omega)$. Then, the system

\[
\left\{
\begin{array}{llll}
-\,\mathcal{L}_1 u=t_1 a(x) u^{\alpha_1}v^{\beta_1} + 1 & {\rm in} \ \ \Omega,\\
-\,\mathcal{L}_2 v=ct_1 b(x) v^{\alpha_2}u^{\beta_2} + 1 & {\rm in} \ \ \Omega,\\
u=v=0 & {\rm on} \ \ \partial\Omega
\end{array}
\right.
\]
possesses a positive strong solution $(u_1,v_1) \in (W^{2,p}(\Omega))^2$. Note that $(u_1,v_1)$ is a positive strong supersolution of

\[
\left\{
\begin{array}{llll}
-\,\mathcal{L}_1 u=t a(x) u^{\alpha_1}v^{\beta_1} + 1 & {\rm in} \ \ \Omega,\\
-\,\mathcal{L}_2 v=ct b(x) v^{\alpha_2}u^{\beta_2} + 1 & {\rm in} \ \ \Omega,\\
u=v=0 & {\rm on} \ \ \partial\Omega.
\end{array}
\right.
\]
Since $(0,0)$ is a strong subsolution of the latter problem, applying Lemma 2.1 of \cite{MR1765542} we obtain a positive strong solution for the above system, so that $(t, ct) \in \tilde{O}^{a,b}_{\alpha_1,\beta_1,\alpha_2,\beta_2}(\Omega)$. Therefore, by Lemma \ref{lem:OequalsOtilde_new}, $(t, ct) \in O^{a,b}_{\alpha_1,\beta_1,\alpha_2,\beta_2}(\Omega)$. This completes the proof.
\end{proof}

\begin{lemma}\label{lemma1}
 The sets $G_{\alpha_1,\beta_1,\alpha_2,\beta_2}^{a,b}(\Omega)$ and $O_{\alpha_1,\beta_1,\alpha_2,\beta_2}^{a,b}(\Omega)$ are equal.
\end{lemma}
\begin{proof} By definition, we have $G_{\alpha_1,\beta_1,\alpha_2,\beta_2}^{a,b}(\Omega)\subset O_{\alpha_1,\beta_1,\alpha_2,\beta_2}^{a,b}(\Omega)$. Let $ (\lambda, \mu) \in O_{\alpha_1,\beta_1,\alpha_2,\beta_2}^{a,b}(\Omega)$. Fix $ f_0, g_0 \in L^p(\Omega) $ with $ f_0, g_0 \geq 0 $ in $ \Omega$. Let $(f_k,g_k)_k$ be sequences in $X$ such that  $f_k \rightarrow f_0 $ and $ g_k \rightarrow g_0 $ in $ L^p(\Omega)$. Let $ (u_k,v_k) $ be a nonnegative solution of the problem
\[
\left\{
\begin{array}{llll}
-\,\mathcal{L}_1 u = \lambda a(x) u^{\alpha_1}v^{\beta_1} +f_k(x) & {\rm in} \ \ \Omega,\\
-\,\mathcal{L}_2 v = \mu b(x) v^{\alpha_2}u^{\beta_2} +g_k(x)& {\rm in} \ \ \Omega,\\
u= v=0 & {\rm on} \ \ \partial\Omega.
\end{array}
\right.
\]
We assert that $ (u_k,v_k)_k $ is bounded in $X$. Otherwise, module a subsequence, we get $ \| u_k \|_{L^\infty(\Omega)} + \| v_k \|_{L^\infty(\Omega)} \rightarrow + \infty$. Argumenting Proceeding as in the proof of Theorem \ref{theo1}, we derive $ (\lambda,\mu) \in \mathcal{C}_1$. But this is a contradiction, because $\mathcal{C}_1$ and $O_{\alpha_1,\beta_1,\alpha_2,\beta_2}^{a,b}(\Omega)$ are disjoint. Then, $(u_k,v_k)_k $ is bounded in $ X$. Again proceeding as in Theorem \ref{theo1}, we have $ (u_k,v_k)_k $ converges to $ (u,v) $ in $(W^{2,p}(\Omega))^2$ and $ (u,v) $ satisfies (\ref{eq:new-system-ab}). So, $ (\lambda,\mu) \in G_{\alpha_1,\beta_1,\alpha_2,\beta_2}^{a,b}(\Omega)$.
\end{proof}

Conclusion
\[
G^{a,b}_{\alpha_1,\beta_1,\alpha_2,\beta_2}(\Omega)=
O^{a,b}_{\alpha_1,\beta_1,\alpha_2,\beta_2}(\Omega)=
\widetilde O^{a,b}_{\alpha_1,\beta_1,\alpha_2,\beta_2}(\Omega)=\mathcal{R}_1.
\]

\begin{proposition} For each pair $ (\lambda, \mu) \in \mathcal{R}_1$ and  $(f_0, g_0) \in (L^p(\Omega))^2$, the system (\ref{eq:new-system-ab}) possesses a unique nonnegative solution.
\end{proposition} 
\begin{proof}
Given $(\lambda, \mu)$ in $\mathcal{R}_1$, we divide the proof into two cases. First, suppose that $(0,0)$ is a solution of the problem (\ref{eq:new-system-ab}). If the problem (\ref{eq:new-system-ab}) possesses a positive solution, then the pair $ (\lambda, \mu)$ belongs to $\mathcal{C}_1$, which is a contradiction. Now assume that $(0,0)$ is not a solution of (\ref{eq:new-system-ab}) and suppose on the contrary that the problem (\ref{eq:new-system-ab}) possesses two positive solutions $(u_1,v_1)$ and $(u_2,v_2)$. Consider the set $ \Gamma =\{\gamma>0 :u_1 > \gamma u_2\ {\rm and}\ v_1 > \gamma^{r} v_2\ {\rm in}\ \Omega\}$ and define $\gamma^{\ast}:=\sup \Gamma$. Permuting $(u_1,v_1)$ and $(u_2,v_2)$ in the definition of $\Gamma$, if necessary, we can assume that $\gamma^{\ast}\leq 1$. We assert that $ u_1 \equiv \gamma^{\ast} u_2 $ and $v_1 \equiv \gamma^{\ast r} v_2$ in $\Omega$. Otherwise, if $ u_1 \not\equiv \gamma^{\ast} u_2$ or $ v_1 \not\equiv \gamma^{\ast r} v_2$ in $\Omega$, by the Scalar Strong Maximum Principle in problem
\begin{eqnarray*}
-\,\mathcal{L}_1 (u_1 - \gamma^{\ast} u_2) &=& \lambda a(x)u_1^{\alpha_1}v_1^{\beta_1} - \gamma^{\ast} \lambda a(x)u_2^{\alpha_1} v_2^{\beta_1} +(1-\gamma^{\ast})f_0(x)\geq 0, \\
-\,\mathcal{L}_2 (v_1 - \gamma^{\ast r} v_2) &=& \mu b(x) v_1^{\alpha_2}u_1^{\beta_2} - \gamma^{\ast r} \mu b(x) v_2^{\alpha_2}u_2^{\beta_2}+(1 - \gamma^{\ast r})g_0(x)\geq 0,
\end{eqnarray*}
we obtain $u_1 > (\gamma^{\ast} + \varepsilon) u_2$ or $v_1 > ({\gamma^{\ast}} + \varepsilon)^{r} v_2$ in $\Omega$ for sufficiently small $\varepsilon>0$. Then, by the Scalar Strong Maximum Principle again, we obtain $u_1 > (\gamma^{\ast} + \varepsilon) u_2$ and $v_1 > {(\gamma^{\ast}} + \varepsilon)^{r} v_2$ in $\Omega$ for $\varepsilon \sim 0$, a contradiction. To finish the proof, it is sufficient to show that $ \gamma^{\ast}=1$. Assume on the contrary that $\gamma^{\ast} \in (0,1)$. If $f_0 \not\equiv 0$ in $\Omega$, applying Scalar Strong Maximum Principle in problem 
\[
-\,\mathcal{L}_1 (u_1 - \gamma^{\ast} u_2)=(1 - \gamma^{\ast})f_0(x),
\]
we have $u_1 > (\gamma^{\ast} + \varepsilon) u_2$ in $\Omega\ $ for sufficiently small $\varepsilon>0$, which is a contradiction. If $g_0 \not\equiv 0$ in $\Omega$, the reasoning is analogous. Then, we get $\gamma^{\ast}=1$.
\end{proof}

\section{Proof of Theorem \ref{ABP}}

Let $(f,g) \in (L^p(\Omega))^2$ satisfy $f, g \geq 0$ and $f+g\not\equiv 0$ in $\Omega$, and let $(u,v) \in (W^{2,p}(\Omega))^2$ be a nonnegative nontrivial solution of \eqref{abp}. Since $(\lambda, \mu) \in {\cal R}_1$, by Theorem \ref{E1}, there is a unique positive solution $(z,w) \in (W^{2,p}(\Omega))^2$ of the system

\[
\left\{
\begin{array}{llll}
-{\cal L}_1z = \lambda a(x) z^{\alpha_1} w^{\beta_1} + f(x) & {\rm in} \ \ \Omega,\\
-{\cal L}_2w = \mu b(x) w^{\alpha_2} z^{\beta_2} + g(x) & {\rm in} \ \ \Omega,\\
z=0,\ w=0 & {\rm on} \ \ \partial\Omega.
\end{array}
\right.
\]
Moreover, by Theorem \ref{thm:apriori-general}, we obtain
\begin{eqnarray*}
\|z\|_{L^\infty(\Omega)}+\|w\|_{L^\infty(\Omega)}^{\,s}
		\;\le\;
		A\Big(
		\|f\|_{L^{p}(\Omega)}
		+\|g\|_{L^{p}(\Omega)}^{\,s}
		\Big).
\end{eqnarray*}

In order to finish the proof, it suffices to prove that $u \leq z$ and $v\leq w$ in $\Omega$. Assume by contradiction that $u > z$ or $v > w$ somewhere in $\Omega$. Consider the set $\Gamma := \{\gamma>0 : z > \gamma u\ {\rm and} \ w > \gamma^{r} v\ {\rm in} \ \Omega\}$. Note that $\Gamma$ is nonempty by Hopf's Lemma and is also bounded above. Define $\gamma^{\ast}=\sup \Gamma > 0$. Hence, $z \geq \gamma^{\ast} u$ and $w \geq \gamma^{\ast r} v$ in $\Omega$. Note that  $\gamma^{\ast} < 1$. We define
\begin{eqnarray*}
&& f_1(x):= {\cal L}_1 z + \lambda a(x) z^{\alpha_1}w^{\beta_1}, \ \ g_1(x):= {\cal L}_1 u + \lambda a(x) u^{\alpha_1}v^{\beta_1},\\
&& f_2(x):= {\cal L}_2 w + \mu b(x) w^{\alpha_2} z^{\beta_2}, \ \ g_2(x):= {\cal L}_2 v + \mu b(x) v^{\alpha_2}u^{\beta_2}.
\end{eqnarray*}
Thus, $f_1,f_2\leq 0$, $f_1 \leq g_1$ and $f_2 \leq g_2$ in $\Omega$. Then, since $\gamma^{\ast} < 1$, we have

\[
\left\{
\begin{array}{lll}
{\cal L}_1(z - \gamma^{\ast} u) &=& \lambda a(x) \left[(\gamma^{\ast} u)^{\alpha_1} (\gamma^{\ast r} v)^{\beta_1} - z^{\alpha_1}w^{\beta_1}\right] + f_1(x) - \gamma^{\ast} g_1(x)\\
&\leq& f_1(x) - \gamma^{\ast} g_1(x) \leq (1 - \gamma^{\ast}) f_1(x) \leq 0 \ \ {\rm in}\ \Omega\ {\rm a.e.}\\
{\cal L}_2(w - \gamma^{\ast r} v) &=& \mu b(x) \left[(\gamma^{\ast r} v)^{\alpha_2}(\gamma^{\ast} u)^{\beta_2} - w^{\alpha_2}z^{\beta_2}\right] + f_2(x) - \gamma^{\ast r} g_2(x)\\
&\leq& f_2(x) - \gamma^{\ast\beta} g_2(x) \leq (1 - \gamma^{\ast\beta}) f_2(x) \leq 0 \ \ {\rm in}\ \Omega\ {\rm a.e.}
\end{array}
\right. 	
\]
Then, using that ${\cal L}_1$ and ${\cal L}_2$ satisfy the Scalar Strong Maximum Principle in $\Omega$, we have $z \geq (\gamma^{\ast} + \varepsilon) u$ and $w \geq (\gamma^{\ast}+ \varepsilon)^r v$ in $\Omega$ for any $\varepsilon > 0$ small enough, contradicting the definition of $\gamma^{\ast}$. 

So, $u \leq z$ and $v \leq w$ in $\Omega$. Combining this fact with the above estimate, we finally obtain \eqref{estABP}, which completes the proof.

\noindent\section*{Acknowledgments}
\addcontentsline{toc}{section}{Acknowledgments}

\noindent E.J.F. Leite has been partially supported by CNPq-Grant 303019/2025-5. The first author gratefully acknowledges the invitation of Prof. Edir Leite to visit UFSCar, and thanks the department for its warm hospitality. He also thanks the Department of Mathematics at UFRR for granting leave to carry out his postdoctoral research.

We are thankful to Professor Gustavo Ferron Madeira for useful discussions, suggestions and corrections of abstract.

\vspace{0.5em}

\textbf{Data availability} \quad Data sharing not applicable to this article as no datasets were generated or analyzed during the current study. 

\section*{Declarations}

\textbf{Conflict of interest} \quad The authors declare that they have no conflict of interest.

\printbibliography
	
\end{document}